\documentclass[11pt]{article}
\usepackage{latexsym}
\usepackage{amssymb}
\usepackage{amsmath}
\usepackage{amsthm}
\usepackage{bm}
\usepackage{mathrsfs}

\allowdisplaybreaks[1]

\theoremstyle{plain}
\newtheorem{thm}{Theorem}[section]
\newtheorem{lem}[thm]{Lemma}
\newtheorem{cor}[thm]{Corollary}
\newtheorem{pro}[thm]{Proposition}

\theoremstyle{definition}
\newtheorem{rem}[thm]{Remark}
\newtheorem{defin}[thm]{Definition}
\newtheorem{exm}[thm]{Example}
\newtheorem{assumption}[thm]{Assumption}

\newcommand{\RR}{{\mathbb R}}

\newcommand{\ep}{\varepsilon}

\newcommand{\XX}{{\mathbb X}}

\newcommand{\C}{{\mathcal C}}
\newcommand{\nn}{{\nonumber}}
\newcommand{\tZ}{\tilde{Z}}
\newcommand{\tPsi}{\tilde{\Psi}}
\newcommand{\tPhi}{\tilde{\Phi}}
\newcommand{\tY}{\tilde{Y}}
\newcommand{\ta}{\tilde{\alpha}}
\newcommand{\tb}{\tilde{\beta}}
\newcommand{\zb}{\beta_0}
\newcommand{\tL}{\tilde{L}}
\newcommand{\tA}{\tilde{A}}
\newcommand{\tF}{\tilde{F}}
\newcommand{\tC}{\tilde{C}}
\newcommand{\ua}{\underline{\alpha}}
\newcommand{\oa}{\overline{\alpha}}
\newcommand{\B}{{\mathcal B}}
\newcommand{\W}{{\mathcal W}}

\newcommand{\Lip}{{\rm Lip}}
\newcommand{\CLip}{\mathscr{C}_{{\rm Lip}}(\RR^d)}

\newcommand{\tomega}{\tilde{\omega}}

\newcommand{\bomega}{\bar{\omega}}
\newcommand{\hyp}{\mathchar`-}
\makeatletter
\newcommand{\opnorm}{\@ifstar\@opnorms\@opnorm}
\newcommand{\@opnorms}[1]{%
  \left|\mkern-1.5mu\left|\mkern-1.5mu\left|
   #1
  \right|\mkern-1.5mu\right|\mkern-1.5mu\right|
}
\newcommand{\@opnorm}[2][]{%
  \mathopen{#1|\mkern-1.5mu#1|\mkern-1.5mu#1|}
  #2
  \mathclose{#1|\mkern-1.5mu#1|\mkern-1.5mu#1|}
}
\makeatother

\newcommand{\wopnorm}{\widetilde{\opnorm{\mathbf{X}}_{\beta}}}

\numberwithin{equation}{section}

\begin{document}
\newcounter{aaa}
\newcounter{bbb}
\newcounter{ccc}
\newcounter{ddd}
\newcounter{eee}
\newcounter{xxx}
\newcounter{xvi}
\newcounter{x}
\setcounter{aaa}{1}
\setcounter{bbb}{2}
\setcounter{ccc}{3}
\setcounter{ddd}{4}
\setcounter{eee}{32}
\setcounter{xxx}{10}
\setcounter{xvi}{16}
\setcounter{x}{38}
\title
{Rough
differential equations\\
containing path-dependent bounded variation terms
\footnote{This version is accepted for publication in Journal of Theoretical
Probability\\https://doi.org/10.1007/s10959-024-01319-3}}
\author{Shigeki Aida
\\
Graduate School of Mathematical Sciences\\
The University of Tokyo\\
3-8-1 Komaba Meguro-ku Tokyo, 153-8914, JAPAN\\
e-mail: aida@ms.u-tokyo.ac.jp}
\date{}
\maketitle
 \begin{abstract}
  We consider rough differential equations
  whose coefficients contain path-dependent\\
  bounded variation
  terms and prove the existence and
  a priori estimate of solutions.
  These equations include classical path-dependent
  SDEs containing running maximum processes and normal reflection terms.
  We apply these results to determine the topological support of
the solution processes.
\end{abstract}
  
 \medskip
 
 {\it Keywords:}\, path-dependent rough differential equations,
reflected stochastic differential equation,
running maximum, 
Skorohod equation, rough path, controlled path

{\it MSC2020 subject classifications:}\,
60L20, 60H10, 60H20, 60G22

\section{Introduction}

  In the framework of It\^o's calculus,
  path-dependent stochastic differential equations(=SDEs)
  are naturally
  formulated and the existence and uniqueness hold under
  suitable standard assumptions on the coefficients.
  For example, reflected SDEs and SDEs containing running maximum and
  minimum processes are typical examples.
  In one dimensional cases, very simple SDEs containing
  the maximum and minimum processes and reflection term
  have been studied in detail.
    In this paper, we consider rough differential equations (=RDEs)
  whose coefficients contain path-dependent bounded variation
  terms and prove the existence and a priori estimate of solutions.
  This class of equations include the classical path-dependent
  SDEs mentioned above.
Although the solutions are not unique in general, the uniqueness holds
for smooth rough paths in many cases.
  Under the uniqueness assumption, we prove a continuity property of solution
  mappings at smooth rough paths which is useful to determine
  the topological support of the solution processes.

The structure of this paper is as follows.
In Section 2, we introduce a class of
RDEs containing bounded variation terms:
\begin{align}
 Z_t&=\xi+\int_0^t\sigma(Z_s,A(Z)_s)d\mathbf{X}_s,\label{rde1}
\end{align}
where $\mathbf{X}_t$ is a $1/\beta$ rough path ($1/3<\beta\le 1/2$)
and $A(Z)_t$ is 
a continuous bounded variation path 
which depends on
the past path $(Z_s)_{s\le t}$.
After that, we state our main theorem (Theorem~\ref{main theorem})
which proves the existence and a priori estimate of
solutions under $\sigma\in \Lip^{\gamma-1}$ ($\gamma>1/\beta$) and
suitable assumptions on $A$.
Note that the regularity assumption on $\sigma$ for the existence of solutions
is standard in the case of usual RDEs which corresponds to
$A\equiv 0$.
The solution $Z_t$
is a controlled path of
the driving rough path $\mathbf{X}$.
Actually, we solve this equation in 
product Banach spaces consisting of
$Z$ and $\Psi=A(Z)$ by applying
Schauder's fixed point theorem.

To this end, we introduce
H\"older continuous path spaces
$\C^{\theta}$ and
Banach spaces $\C^{q\hyp var, \theta}$
consisting of $\Psi$
based on the
control function $\omega$
of $\mathbf{X}$.
The latter is a set of paths whose $q$-variation norms
$(q\ge 1)$
are finite and
satisfy a certain
H\"older continuity defined by
$\omega$.
We also study basic properties of
the functional $A$.
We briefly explain examples but
we will discuss the detail in Section~\ref{examples}.

In Section 3, we prove our main theorem.
The uniqueness does not hold in general.
See Remark~\ref{remark on main theorem} (6).

In Section 4, we consider usual
$\beta$-H\"older rough path $\mathbf{X}$
with the control function
$\omega(s,t)=|t-s|$.
We show that the (generally multivalued)
solution mapping is continuous
at a rough path for which the solution is unique
in Proposition~\ref{continuity of Z}
using a priori estimate of solutions.
We use this result to prove support theorems in
Section 6.

In Section 5, we give examples.
In Subsection 5.1, we consider reflected rough differential equations
on a domain $D$ in $\RR^n$:
\begin{align}
 Y_t&=\xi+\int_0^t\sigma(Y_s)d\mathbf{X}_s+\Phi_t,\quad
\xi\in \bar{D},\label{rde2}
\end{align}
where $\Phi_t$ is the reflection term which forces $Y_t\in \bar{D}$.
This equation looks different from the equation studied in the main
theorem.
However, it is well-known that reflected It\^o (Stratonovich)
SDEs can be transformed to 
certain path-dependent It\^o (Stratonovich) SDEs without reflection
term.
This is used to prove Freidlin-Wentzell type large deviation principle
(\cite{anderson-orey}) and the support theorem (\cite{doss})
for reflected diffusions on domains with smooth boundary.
We prove the existence theorem (Theorem~\ref{reflected case})
under standard assumptions
(A) and (B) on $D$ and
$\sigma\in \Lip^{\gamma-1}$ by transforming the equation
(\ref{rde2})
to the corresponding path-dependent RDE
(\ref{rde1}).
This is an extension of the result in
\cite{aida-rrde} in which
we proved the existence of
solutions of (\ref{rde2}) under stronger assumptions that
$D$ satisfies the condition (H1) and $\sigma\in C^3_b$.

In 1-dimensional cases, perturbed SDEs and perturbed reflected SDEs
were studied by many people.
See {\it e.g.} 
\cite{cpy,chaumont-doney,davis,davis2, doney-zhang,perman-werner,yue-zhang}.
In Subsection 5.2, we give a short review of these subjects.

In Subsection 5.3, we consider multidimensional and rough path versions
of $1$-dimensional perturbed SDEs and perturbed reflected SDEs.
In the study of the latter one, we need to consider
an implicit Skorohod equation as in \cite{aida-rrde}.
As for perturbed reflected SDE whose driving process 
is the standard Brownian motion, we can extend the existence and
uniqueness result of the solution due to Doney and
Zhang~\cite{doney-zhang} by using our approach.
See Remark~\ref{remark on sde}.

Path-dependent functional $A(x)_t$ which we are mainly concerned with
in this paper is a
kind of generalization of the maximum process $\max_{0\le s\le t}|x_s|$
and the local time term $L(x)_t$.
The maximum process $\max_{0\le s\le t}|x_s|$ is obtained as the
limit of $\|x\|_{L^p([0,t])}$ as $p\to\infty$.
Hence it may be natural to study the case where
$A(x)_t=\|x\|_{L^p([0,t])}$.
In Subsection 5.4, we study such examples.

In Section 6,
we prove support theorems for solution processes
by using Proposition~\ref{continuity of Z}
and Wong-Zakai theorems.
In this section, except Theorem~\ref{general support theorem},
we consider the Brownian rough path
$\mathbf{W}$ which implies that we consider the usual
Stratonovich SDEs driven by the standard Brownian motion.

Section 7 is an appendix.
The solution $Y_t$ studied in Section 5
is a sum of a controlled path $Z_t$ and a continuous bounded variation
path $\Phi_t$.
For a given controlled path $Z$,
the Gubinelli derivative $Z'$ is uniquely determined
if the first level path $X$ of $\mathbf{X}$ is truly rough in the sense of
\cite{friz-hairer}.
In our case, $\Phi$ is certainly bounded variation but does not
have good regularity property in H\"older norm.
Hence it is natural to ask whether
$Z'$ is unique or not for $Y$ in our setting.
We study this problem by using a certain rough
property of the path $X$ in Subsection 7.1.
In Subsection 7.2, we make a remark on path-dependent rough
differential equations with drift.
This consideration is necessary for the study of the reflected diffusions with
the drift terms.

\section{Preliminary and Main Theorem}\label{preliminary}
Let us fix a positive number $T$.
Let $\omega(s,t)$~$(0\le s\le t\le T)$ be a control function.
That is, $(s,t)\mapsto \omega(s,t)\in \RR^{+}$ is a continuous function
and $\omega(s,u)+\omega(u,t)\le \omega(s,t)$~$(0\le s\le u\le t\le T)$ holds.
We introduce a mixed norm by using $\omega$ and $p$-variation norm.
We refer the readers to \cite{fp} for the related studies.
Let $E$ be a finite dimensional normed linear space.
For a continuous path $(x_t)$\, $(0\le t\le T)$ on 
$E$, we define
for $[s,t]\subset [0,T]$,
\begin{align}
\|x\|_{\infty, [s,t]}&=\max_{s\le u\le t}|x_u|,\\
\|x\|_{\infty\hyp var,[s,t]}&=\max_{s\le u\le v\le t}|x_{u,v}|,\\
\|x\|_{p\hyp var,[s,t]}&=\left\{
\sup_{{\cal P}}\sum_{k=1}^N|x_{t_{k-1},t_{k}}|^p\right\}^{1/p},
\label{total p-variation}
\end{align}
where ${\cal P}=\{s=t_0<\cdots<t_N=t\}$ is a partition 
of the interval $[s,t]$ and $x_{u,v}=x_v-x_u$.
When $[s,t]=[0,T]$, we may omit denoting $[0,T]$.
For $0<\theta\le 1, q\ge 1$, $0\le s\le t\le T$ 
and a continuous path $x$, we define
\begin{align}
\|x\|_{\theta,[s,t]}&=
\inf\left\{C>0~|~|x_{u,v}|\le C\omega(u,v)^{\theta}
\quad s\le u\le v\le t\right\},\\
\|x\|_{q\hyp var,\theta,[s,t]}&=
\inf\left\{
C>0~\Big |~
\|x\|_{q\hyp var,[u,v]}\le C\omega(u,v)^{\theta}\quad
s\le u\le v\le t
\right\}.
\end{align}
We use the convention that
$\inf \emptyset=+\infty$.
When $\omega(s,t)=|t-s|$, $\|x\|_{\theta, [s,t]}<\infty$ 
is equivalent to that $x_u$~$(s\le u\le t)$ 
is a H\"older continuous path with
the exponent $\theta$ in usual sense.
Hence we may say $x$ is an $\omega$-H\"older continuous
path with the exponent $\theta$ ($(\omega,\theta)$-H\"older continuous
path in short).
For two parameter function
$F_{s,t}$ $(0\le s\le t\le T)$, we define
$\|F\|_{\theta,[s,t]}$ similarly.

We denote by $\C^{\theta}([0,T], E)$ the set of 
$\omega$-H\"older continuous paths $x$ with values in
$E$ satisfying 
$\|x\|_{\theta}=\|x\|_{\theta,[0,T]}<\infty$.
We may denote the function space
by $(\C^{\theta}([0,T], E),\omega)$ to specify the control function.
$\C^{\theta}([0,T], E)$ is a Banach space with the norm 
$|x_0|+\|x\|_{\theta}$.
We may just write $\C^{\theta}(E)$ if there is no
confusion.
Let
$\C^{q\hyp var,\theta}(E)$ denote the set of $E$-valued 
continuous paths of finite $q$-variation defined on 
$[0,T]$
satisfying 
$\|x\|_{q\hyp var,\theta}:=\|x\|_{q\hyp var,\theta,[0,T]}<\infty$.
Note that $\C^{q\hyp var,\theta}(E)$ is a Banach 
space with the norm $|x_0|+\|x\|_{q\hyp var,\theta}$.
Obviously, any path $x\in \C^{q\hyp var,\theta}(E)$ satisfy
$|x_{s,t}|\le \|x\|_{q\hyp var,\theta}\omega(s,t)^{\theta}$.
We may write $\C^{\theta}, \C^{q\hyp var,\theta}$
for simplicity.

We next introduce the notation
for mappings between normed linear spaces.
Let $E, F$ be finite dimensional normed linear spaces.
For $\gamma=n+\theta$ $(n\in {\mathbb N}\cup \{0\},
0<\theta\le 1)$, $\Lip^{\gamma}(E,F)$
denotes the set of bounded functions $f$ on $E$ with values in $F$
which are $n$-times continuously differentiable and
whose derivatives up to $n$-th order are bounded and
$D^nf$ is a H\"older continuous function with the exponent
$\theta$ in usual sense.

We use the following lemma.
The compact embedding in $(2)$ is necessary for the application of
the Schauder fixed point theorem.

\begin{lem}\label{comparison of norms}
\begin{enumerate}
 \item[$(1)$] Let $1\le q'\le q$.
For a continuous path $x$, we have
\begin{align}
 \|x\|_{q\hyp var,[s,t]}\le
\|x\|_{q'\hyp var,[s,t]}^{q'/q}
\|x\|_{\infty\hyp var,[s,t]}^{(q-q')/q}
\le \|x\|_{q'\hyp var,[s,t]}.\label{comparison norms1}
\end{align}
 \item[$(2)$]
Let $1\le q'\le q$.
Let $0<\theta, \theta'\le 1$ be positive numbers such that
	     $q\theta\le q'\theta'$.
	     Then for any $x\in \C^{q'\hyp var,\theta'}$,
	     we have
	     \begin{align}
	      \|x\|_{q\hyp var,\theta}\le
	      \omega(0,T)^{(q'\theta'-q\theta)/q}
	      \|x\|_{q'\hyp var,\theta'}^{q'/q}
	      \|x\|_{\infty\hyp var}^{(q-q')/q}.
	      \label{comparison norms2}
	     \end{align}
Further if $q'<q$ holds, then
the inclusion $\C^{q'\hyp var,\theta'}\subset \C^{q\hyp var,\theta}$
is compact.
\item[$(3)$] If $\|x\|_{q\hyp var,[s,t]}<\infty$ for some $q$, then
$\lim_{q\to\infty}\|x\|_{q\hyp var,[s,t]}=\|x\|_{\infty\hyp var,[s,t]}$.
\end{enumerate}
\end{lem}

\begin{proof}
(1) We have
\begin{align}
 \|x\|_{q\hyp var,[s,t]}&=\left\{\sup_{{\cal P}}\sum_{i}
|x_{t_{i-1},t_i}|^{q}\right\}^{1/q}\nn\\
&\le \left\{\sup_{{\cal P}}\sum_{i}
|x_{t_{i-1},t_i}|^{q'}
\max_i |x_{t_{i-1},t_i}|^{q-q'}
\right\}^{1/q}\nn\\
&\le \|x\|_{q'\hyp var,[s,t]}^{q'/q}
\|x\|_{\infty,[s,t]}^{(q-q')/q}.
\end{align}
The second inequality follows from the trivial bound
$\|x\|_{\infty\hyp var,[s,t]}\le\|x\|_{q'\hyp var,[s,t]}$.

 \noindent
 (2)
 By (1), we have
 \begin{align}
\|x\|_{q\hyp var,[s,t]}&\le \|x\|_{q'\hyp var,\theta',[s,t]}^{q'/q}
\omega(s,t)^{(\theta'q')/q-\theta}
\|x\|_{\infty\hyp var,[s,t]}^{(q-q')/q}\omega(s,t)^{\theta}.
\label{estimate q and theta}
 \end{align}
 This implies $(\ref{comparison norms2})$.
If $\sup_n|(x_n)_0|+\|x_n\|_{q'\hyp var,\theta'}<\infty$, then by their
 equicontinuities,
there exists a subsequence such that 
$x_{n_k}$ converges to a certain function
$x_{\infty}$ in the uniform norm.
 By $(\ref{comparison norms2})$,
 we can conclude that the convergence 
takes place with respect to the norm on $\C^{q\hyp var,\theta}$.

\noindent
(3)
We need only to prove $\limsup_{q\to \infty}\|x\|_{q\hyp var,[s,t]}\le 
\|x\|_{\infty\hyp var,[s,t]}$.
Suppose $\|x\|_{q_0\hyp var,[s,t]}<\infty$.
Then for $q>q_0$,
\begin{align}
 \sup_{{\cal P}}\left(\sum_{i}|x_{t_{i-1},t_i}|^q\right)^{1/q}
&\le \sup_{{\cal P}}\left(\sum_{i}|x_{t_{i-1},t_i}|^{q_0}\right)^{1/q}
\sup_{{\cal P}}\max_i|x_{t_{i-1},t_i}|^{(q-q_0)/q}.
\end{align}
Taking the limit $q\to\infty$, we obtain the desired estimate.
 \end{proof}

Throughout this paper, $\beta$ is a positive number satisfying
$1/3<\beta\le 1/2$ if there are no further comments.
Let $\omega$ be a control function and
let 
$\mathbf{X}_{s,t}=(X_{s,t},\XX_{s,t})$\, $(0\le s\le t\le T)$
be a $(\omega,\beta)$-H\"older rough path on $\RR^d$.
That is, $\mathbf{X}$ satisfies Chen's relation and
the path regularity conditions,
\begin{align}
|X_{s,t}|\le \|X\|_{\beta}\omega(s,t)^{\beta},\quad
|\XX_{s,t}|\le \|\XX\|_{2\beta}\omega(s,t)^{2\beta},
\qquad 0\le s\le t\le T,
\end{align}
where $\|X\|_{\beta}(<\infty)$ and $\|\XX\|_{2\beta}(<\infty)$ 
denote the $\omega$-H\"older norm.
 We denote by $\mathscr{C}^{\beta}(\RR^d)$
the set of all $(\omega,\beta)$-H\"older rough paths,
where $\omega$ moves in the set of all control functions.
When $\omega(s,t)=|t-s|$, $\mathbf{X}_{s,t}$
is a usual $\beta$-H\"older rough path.
If $\mathbf{X}_{s,t}$ is a rough path with finite
$1/\beta$-variation,
setting
$\omega(s,t)=\|X\|_{1/\beta\hyp var, [s,t]}^{1/\beta}+
\|{\mathbb X}\|_{1/(2\beta)\hyp var, [s,t]}^{1/(2\beta)}$,
$\|X\|_{\beta}\le 1$ and $\|{\mathbb X}\|_{2\beta}\le 1$ hold.
We refer the reader
 to \cite{friz-hairer, friz-victoir, lq, lyons98, bailleul1}
 for the references of rough paths.

We use the following quantity,
\begin{align}
\wopnorm=\sum_{i=1}^3\opnorm{\mathbf{X}}_{\beta}^i,\quad
\opnorm{\mathbf{X}}_{\beta}=\|X\|_{\beta}
+\sqrt{\|\XX\|_{2\beta}}.
\end{align}

We introduce a set of controlled paths ${\mathscr D}^{2\theta}_X(\RR^n)$ 
of
$\mathbf{X}_{s,t}$, where $1/3<\theta\le \beta$ following
\cite{gubinelli, friz-hairer}.
A pair of $\omega$-H\"older continuous paths
$(Z,Z')\in \C^{\theta}([0,T],\RR^n)\times \C^{\theta}([0,T],
{\cal L}(\RR^d,\RR^n))$ with the exponent
$\theta$
is called a controlled path of $X$, if
the remainder term $R^Z_{s,t}=Z_t-Z_s-Z'_sX_{s,t}$ satisfies
$\|R^Z\|_{2\theta}<\infty$.
The set of controlled paths
${\mathscr D}^{2\theta}_X(\RR^n)$ 
is a Banach space with the norm
\begin{align}
\|(Z,Z')\|_{2\theta}
&=|Z_0|+|Z'_0|+\|Z'\|_{\theta}+\|R^Z\|_{2\theta}\quad
(Z,Z')\in {\mathscr D}^{2\theta}_X(\RR^n).
\end{align}

The rough differential equations which we will study contain
path dependent bounded variation term
$A(x)_t$.
We consider the following condition on $A$.
Note that the function space $\C^{\beta}$ in the following
statement depends on the control function $\omega$.

\begin{assumption}\label{assumption on A}
Let $\xi\in \RR^n$.
Let $\omega$ be a control function.
Let $A$ be a mapping from
$\C^{\beta}([0,T], \RR^n~|~x_0=\xi)$
 to $C([0,T], \RR^n)$ satisfying the following.
\begin{enumerate}
\item[$(1)$]$(\text{Adaptedness})$ 
$\left(A(x)_s\right)_{0\le s\le t}$ depends only on
$(x_s)_{0\le s\le t}$ for all $0\le t\le T$.
 \item[$(2)$] $(\text{Continuity})$ There exists $1/3<\zb<\beta$ such that
$A$ can be extended to a continuous mapping from
$\C^{\zb}([0,T],\RR^n~|~x_0=\xi)$ to 
$(C([0,T], \RR^n), \|~\|_{\infty,[0,T]})$.
We use the same notation $A$ to denote the extended mapping on
$\C^{\zb}$.
 \item[$(3)$] There exists a non-decreasing
	      positive continuous function
$F$ on $[0,\infty)$ such that for all
$x\in \C^{\zb}([0,T],\RR^n~|~x_0=\xi)$,
\begin{align}
 \|A(x)\|_{1\hyp var,[s,t]}\le
 F(\|x\|_{(1/\zb)\hyp var,[s,t]})
\|x\|_{\infty\hyp var,[s,t]}, \quad 0\le s\le t\le T
\label{estimate of Ax}
\end{align}
hold.
\end{enumerate}
\end{assumption}

 \begin{rem}
The conditions (1), (2) are natural.
In many cases, $A$ is defined on
continuous path spaces and is continuous with
respect to the uniform norm.
The condition (3) is strong assumption.
This implies that the total variation of
$A(x)$ on $[s,t]$ can be estimated by
the norm of the path $(x_u-x_s)$ on $s\le u\le t$.
Note that this does not exclude the case where
$A(x)_u$~$(s\le u\le t)$ depends on
$x_v$~$(v\le s)$.
\end{rem}

We have the following simple result.

\begin{lem}\label{property of A}
Let $\omega$ be a control function and let
$\C^{\beta}([0,T],\RR^n)$ be the
corresponding H\"older space.
\begin{enumerate}
 \item[$(1)$] 
Suppose $A : \C^{\beta}([0,T], \RR^n~|~x_0=\xi)\to
C([0,T], \RR^n)$ satisfies Assumption~$\ref{assumption on A}$
$(1)$, $(2)$.
Then the initial value $A(x)_0$ is independent of 
$x\in \C^{\beta}([0,T], \RR^n~|~x_0=\xi)$.
\item[$(2)$] Let $0<T'<T$ and set
$\omega_{T'}(s,t)=\omega(T'+s,T'+t)$ $(0\le s\le t\le T-T')$.
Then $\omega_{T'}$ is a control function.
\item[$(3)$]
Let $\C^{\beta}_{T'}([0,T-T'],\RR^n)$ be the  $(\omega_{T'},\beta)$-
H\"older space.
Let $y\in \C^{\beta}([0,T'],\RR^n)$ and $x\in \C^{\beta}_{T'}([0,T-T'],\RR^n)$
and suppose $y_{T'}=x_0$.
Set
\begin{align*}
\tilde{x}_t=
\begin{cases}
 y_t & t\le T',\\
x_{t-T'} &T'\le t\le T.
\end{cases}
\end{align*}
Then $\tilde{x}\in \C^{\beta}([0,T],\RR^n)$.
Let
\[
 \tilde{A}_{y,T'}(x)_t=A(\tilde{x})_{T'+t},\quad 0\le t\le T-T',\quad
x\in \C^{\beta}_{T'}([0,T-T'],\RR^n~|~x_0=y_{T'}).
\]
Then $\tilde{A}_{y,T'}$ satisfies Assumption~$\ref{assumption on A}$
replacing $\omega$ and $T$ by $\omega_{T'}$ and $T-T'$.
In particular, $(\ref{estimate of Ax})$ holds for the same function $F$.
\end{enumerate}
\end{lem}

\begin{proof}
(1) For $x\in C([0,T],\RR^n)$, let $x^t_u=x_{t\wedge u}$.
Then by Assumption~\ref{assumption on A}
(1), $A(x)_u=A(x^t)_u$ $(0\le u\le t)$ holds.
By a simple calculation, for any $x,y\in C([0,T],\RR^n)$,
we have
\begin{align*}
 \|x^t-y^t\|_{\C^{\zb}}\le \left(\|x\|_{\C^{\beta}}
+\|y\|_{\C^{\beta}}\right)\omega(0,t)^{\beta-\zb}.
\end{align*}
Since $(y^0)^t=y^0$, this implies $\lim_{t\to +0}\|x^t-y^0\|_{\C^{\zb}}=0$.
Hence
\begin{align*}
 |A(x)_0-A(y)_0|&=|A(x^t)_0-A(y^0)_0|\le
\|A(x^t)-A(y^0)\|_{\infty,[0,T]}\to
0\quad \text{as $t\downarrow 0$}.
\end{align*}
(2) and (3) are easy to check.
\end{proof}

Actually, the condition (3) automatically implies the following
stronger estimate.
By this result, we may assume that
the growth rate of $F(u)$ is at most
of order $u^{1/\beta}$, that is, a polynomial order.

\begin{lem}\label{lemma for A}
 Assume the mapping $A : \C^{\beta}([0,T], \RR^n~|~x_0=\xi)
 \to C([0,T], \RR^n)$ satisfies the condition
 {\rm (3)} in Assumption~$\ref{assumption on A}$.
\begin{enumerate}
\item[{\rm (1)}]
There exists $C>0$ such that
 \begin{align}
  \|A(x)\|_{1\hyp var,[s,t]}\le
C\left(\|x\|_{(1/\zb)\hyp var, [s,t]}^{1/\zb}+1\right)
  \|x\|_{\infty\hyp var,[s,t]}
  \quad 0\le s\le t\le T.\label{polynomial order}
 \end{align}
 \item[{\rm (2)}]
	      Let us choose positive numbers $\ta$ and $q$
	      such that $\ta\le\beta$ and $1\le q\le \beta/\ta$.
	      Then for any $x,x'\in \C^{\beta}$, we have
	      \begin{align}
	       &   \|A(x)-A(x')\|_{q\hyp var, \ta}\nn\\
	       &\le
	       \omega(0,T)^{\frac{\beta}{q}-\ta}
	       \left(F(\|x\|_{\zb}\omega(0,T)^{\zb})
	       \|x\|_{\beta}+
	       F(\|x'\|_{\zb}\omega(0,T)^{\zb})
	       \|x'\|_{\beta}
	       \right)^{1/q}\nn\\
	       &\qquad\qquad \times
	       \|A(x)-A(x')\|_{\infty\hyp var}^{1-(1/q)}.
	      \end{align}
	      \end{enumerate}
 \end{lem}

\begin{proof}
 Let
 $
  \omega_{1/\zb}(s,t)=\|x\|_{1/\zb\hyp var, [s,t]}^{1/\zb}.
 $
 For $\ep>0$,
 we choose the points
 $s=t_0<t_1<\cdots<t_N=t$ such that
 $\omega_{1/\zb}(t_{i-1},t_i)=\ep$ $(1\le i\le N-1)$
 and $\omega_{1/\zb}(t_{N-1},t_N)\le \ep$.
 By the super additivity of
 $\omega_{1/\zb}$, we have
 $
(N-1)\ep\le\sum_{i=1}^N\omega_{1/\zb}(t_{i-1},t_i)\le \omega_{1/\zb}(s,t)
 $
 and
 $N\le \omega_{1/\zb}(s,t)/\ep+1$.
 By the additivity property of the bounded variation norm,
 we have
 \begin{align*}
  \|A(x)\|_{1\hyp var,[s,t]}&=
  \sum_{i=1}^N\|A(x)\|_{1\hyp var, [t_{i-1},t_i]}\\
  &\le \sum_{i=1}^NF\left(\omega_{1/\zb}(t_{i-1},t_i)^{\zb}\right)
  \|x\|_{\infty\hyp var,[t_{i-1},t_i]}\\
  &\le F(\ep^{\zb})\left(\frac{\omega_{1/\zb}(s,t)}{\ep}+1\right)
  \|x\|_{\infty\hyp var, [s,t]}\\
  &=F(\ep^{\zb})\left(
  \frac{\|x\|_{1/\zb\hyp var, [s,t]}^{1/\zb}}{\ep}+1\right)
  \|x\|_{\infty\hyp var, [s,t]}
 \end{align*}
 which implies the desired estimate.

 \noindent
 (2) Applying Lemma~\ref{comparison of norms} (2)
 in the case where $q'=1, \theta'=\beta, \theta=\ta$,
 we have
  \begin{align}
  & \|A(x)-A(x')\|_{q\hyp var, \ta}\\
  &\le
  \omega(0,T)^{(\beta/q)-\ta}\left(
\|A(x)\|_{1\hyp var, \beta}+\|A(x')\|_{1\hyp var, \beta}
  \right)^{1/q}
   \|A(x)-A(x')\|_{\infty\hyp var}^{1-(1/q)}.
  \end{align}
Note that
\begin{align*}
 \|x\|_{1/\zb\hyp var, [s,t]}&=\sup\left\{\left|
\sum_{i}|x_{t_{i-1},t_i}|^{1/\zb}\right|^{\zb}\right\}\\
&\le \sup\left\{\left|
\sum_{i}\left(\|x\|_{\zb,[s,t]}\omega(t_{i-1},t_i)^{\zb}\right)^{1/\zb}
\right|^{\zb}\right\}\\
&\le \|x\|_{\zb,[s,t]}\omega(s,t)^{\zb}.
\end{align*}
 By the assumption on $A$, we have
\begin{align}
 \|A(x)\|_{1\hyp var, \beta}&\le
 F\left(\|x\|_{\zb}\omega(0,T)^{\zb}\right)\|x\|_{\beta}.
\end{align}
 This completes the proof.
 \end{proof}

\begin{rem}
Of course, we may optimize the estimate (\ref{polynomial order}) as follows:
\[
 \|A(x)\|_{1\hyp var, [s,t]}\le
\tilde{F}\left(\|x\|_{(1/\zb)\hyp var, [s,t]}\right)
\|x\|_{\infty\hyp var, [s,t]},
\]
where $\tilde{F}(u)=\inf_{\ep>0}F(\ep)\left\{\left(\frac{u}{\ep}\right)^{1/\zb}+1\right\}$.
\end{rem}

We now introduce our RDEs and state our main theorem.

 \begin{thm}\label{main theorem}
  Let $\gamma>1/\beta$.
Let $\mathbf{X}$ be a $(\omega,\beta)$-H\"older rough path.
Let $\sigma\in \Lip^{\gamma-1}(\RR^n\times \RR^n, {\cal
 L}(\RR^d,\RR^n))$ and $\xi\in \RR^n$.
Assume that the mapping $A : \C^{\beta}([0,T], \RR^n~|~x_0=\xi)
\to C([0,T], \RR^n)$ satisfies the condition in
 Assumption~$\ref{assumption on A}$.
Then the following hold.
\begin{enumerate}
 \item[$(1)$] There exists a controlled path $(Z,Z')\in 
{\mathscr D}^{2\beta}_X(\RR^n)$
such that
\begin{align}
 Z_t&=\xi+\int_0^t\sigma(Z_s,A(Z)_s)d\mathbf{X}_s,\quad
Z'_t=\sigma(Z_t,A(Z)_t),\quad 0\le t\le T.\label{path-rde1}
\end{align}
\item[$(2)$] All solutions $(Z,Z')$ of $(\ref{path-rde1})$
satisfy the following estimate:\\
there exist positive constants $K$ and $\kappa_1,\kappa_2,\kappa_3$
  which depend only on $\sigma, \beta, \gamma$, $F$ such that
\begin{align}
 \|Z\|_{\beta}+\|Z'\|_{\beta}+
\|A(Z)\|_{1\hyp var,\beta}+
\|R^Z\|_{2\beta}\le
K\left\{1+\left(1+\wopnorm\right)^{\kappa_1}\omega(0,T)\right\}^{\kappa_2}
\wopnorm^{\kappa_3}.
\label{a priori estimate}
\end{align}
\end{enumerate}
\end{thm}

First we make some remarks for this theorem and
after that we explain some examples.

\begin{rem}\label{remark on main theorem}
\noindent
(1)~From now on, we always set $\gamma>1/\beta$ for
$1/3<\beta\le 1/2$ if there is no further comment.
 
 \noindent
 (2)~Let $(Z,Z')\in {\mathscr D}^{2\theta}_X(\RR^n)$ $(1/3<\theta\le \beta)$.
 Let $\{\Psi_t\}_{0\le t\le T}$
 be a continuous bounded variation path on $\RR^n$.
 Then we can define the integral
 $\int_0^t\sigma(Z_s, \Psi_s)d\mathbf{X}_s$ in a similar way to
 the usual rough integral.
 We denote the derivative of $\sigma=\sigma(\xi,\eta)$~$(\xi\in \RR^n,
\eta\in \RR^n)$
with respect to $\xi$ by $D_1\sigma$ and $\eta$ by $D_2\sigma$.
Let
\begin{align*}
\Xi^{\Psi}_{s,t}&=
\sigma(Z_s,\Psi_s)X_{s,t}+(D_1\sigma)(Z_s,\Psi_s)Z'_s\XX_{s,t}+
(D_2\sigma)(Z_s,\Psi_s)\int_s^t\Psi_{s,r}\otimes dX_r
\end{align*}
 and $\tilde{\Xi}^{\Psi}_{s,t}=\Xi^{\Psi}_{s,t}
 -(D_2\sigma)(Z_s,\Psi_s)\int_s^t\Psi_{s,r}\otimes dX_r$.
 Let ${\cal P}=\{s=t_0<\cdots <t_N=t\}$ and
write $|{\cal P}|=\max_i|t_{i+1}-t_i|$.
  Then it is easy to check that
 $\lim_{|{\cal P}|\to 0}\sum_{i=1}^N
 \Xi^{\Psi}_{t_{i-1},t_i}$
 converges by the Sewing lemma using
(\ref{estimate of delta}).
 Actually $\lim_{|{\cal P}|\to 0}\sum_{i=1}^{N}
 \tilde{\Xi}^{\Psi}_{t_{i-1},t_i}$ also converges to the same limit value.
  We denote the limit by $\int_s^t\sigma(Z_u,\Psi_u)d\mathbf{X}_u$.
 Hence the sum of the term $\int_s^t\Psi_{s,r}\otimes dX_r$
 does not have any effect on the integral.
 However, we need to consider $\Xi^{\Psi}$ instead of
 $\tilde{\Xi}^{\Psi}$ to obtain estimates of
 the integral in Lemma~\ref{estimate of rough integral}
 which is necessary for the proof of the main theorem.

\noindent
(3) Let us consider the case $\sigma(\xi,\eta)=\tilde{\sigma}(\xi+\eta)$,
where $\tilde{\sigma}\in \Lip^{\gamma-1}(\RR^n,\mathcal{L}(\RR^d,\RR^n))$.
 Let $Y$ be a continuous path on $\RR^n$.
 Suppose that
 there exist $(Z, Z')\in {\mathscr D}^{2\theta}_X(\RR^n)$
 and continuous bounded variation path
 $(\Psi_t)_{0\le t\le T}$
 such that $Y_t=Z_t+\Psi_t$~$(0\le t\le T)$.
 Clearly, the decomposition of $Y$ to controlled path part $Z$
 and the bounded variation part $\Psi$ is not unique.
 We should note that our definition of
 $\int_0^t\tilde{\sigma}(Y_s)d\mathbf{X}_s$
 depends on $Z'$ and $Y$.
 However, under a natural assumption, the Gubinelli derivative
$Z'_t$ is uniquely defined for $Y$
and the integral does not depend on
 the decomposition $(Z,\Psi)$.
 We discuss this problem in the appendix.

\noindent
(4)  Theorem~\ref{main theorem} implies that the solution $Z_t$ satisfies
the following estimate:
\begin{align}
&\Bigl|Z_t-Z_s-\sigma(Z_s,A(Z)_s)X_{s,t}
-
(D_1\sigma)(Z_s,A(Z)_s)\sigma(Z_s,A(Z)_s)\XX_{s,t}\nonumber\\
&\qquad-
(D_2\sigma)(Z_s,A(Z)_s)\int_s^tA(Z)_{s,r}\otimes dX_r\Bigr|
\le G(\omega(0,T),\wopnorm)\omega(s,t)^{\theta},\quad 
0\le s\le t\le T.\label{davie}
\end{align}
Here $G$ is a certain polynomial function which depends on 
$\sigma,\beta,\gamma,F$.
Also $\theta(>1)$ is a positive constant which depends on $\beta$ and 
$\gamma$ (When $\gamma=3$, $\theta=3\beta$ holds).
Clearly, a path $Z_t$ which satisfies (\ref{davie}) is a solution of
(\ref{path-rde1}).

\noindent
(5)
Let $\tomega$ be a control function and $C_i$ be positive constants.
Actually, under the assumption that for
all $0\le s\le t\le T$,
 \begin{align*}
 \|A(x)\|_{1\hyp var, [s,t]}\le
C_1\left(1+\|x\|_{1/\beta_0\hyp var, [s,t]}^{\beta_0}\right)
\|x\|_{\infty\hyp var, [s,t]}+
C_2\tomega(s,t)+
C_3|t-s|^{\beta},
 \end{align*}
we can prove similar results
 to Theorem~\ref{main theorem} for
$\beta$-H\"older rough paths $\mathbf{X}$ with
$\omega(s,t)=|t-s|$ by a similar proof of the main theorem.
This extension is necessary to treat the examples in 
Example~\ref{example of A} (3) and (4).
However, we need to change the upper bound function 
in (\ref{a priori estimate}).
The reason is as follows.
The $\beta$-H\"older rough path $\mathbf{X}$ can be regarded as a
$(\bomega,\beta)$-H\"older rough path, where
$\bomega(s,t)=\tomega(s,t)+|t-s|$.
We can do the same proof as in the main theorem
in this setting.
The control function $\omega$ in (\ref{a priori estimate}) should be changed
to this $\bomega$ and accordingly $\wopnorm$ also should be changed
to the corresponding quantity.
Also we should replace the term $\wopnorm^{\kappa_3}$ by
$(1+\wopnorm)^{\kappa_3}$.

\noindent
(6)~
If $A\equiv 0$, the uniqueness of the solutions hold
under the assumption $\sigma\in \Lip^{\gamma}$.
However, even if $A\equiv 0$, 
 the uniqueness does not hold in general
under $\sigma\in
 \Lip^{\gamma-1}$.
See Davie~\cite{davie}.
Gassiat~\cite{gassiat} gave an example which showed
that the uniqueness does not hold
for reflected RDE even if the coefficient is smooth
and the domain is just a half space.
Contrary to this, in one dimensional case (note that the driving noise is
multidimensional one),
the uniqueness of the solutions of reflected RDEs were proved 
by Deya-Gubinelli-Hofmanov\'a-Tindel in \cite{dght}.
It may be interesting problem to find natural class of solutions for which
the uniqueness hold and a non-trivial class of
reflected RDEs or more generally path-dependent RDEs
for which the uniqueness hold in an appropriate sense.
See also Subsection~\ref{subsection 5.3} for some examples for which
the uniqueness hold.

The situation is different if $\beta>1/2$.
Ferrante and Rovira~\cite{ferrante-rovira} 
proved the existence of solutions of
reflected (Young) ODE on half space 
driven by fractional Brownian motion with the
Hurst parameter $H>1/2$.
 Falkowski and S\l omi\'nski~\cite{falkowski-slominski1}
 proved the Lipschitz continuity of
the Skorohod mapping on a half space in the H\"older space
and proved the uniqueness
in that case.
\end{rem}

We briefly explain examples.
We refer the reader for the detail to
Section~\ref{examples}.

\begin{exm}\label{example of A}

 $(1)$
Let $D$ be a domain in $\RR^n$ satisfying conditions
{\rm (A)} and {\rm (B)}.
Consider the Skorohod equation
$y_t=x_t+\phi_t$, where $x$ is a continuous path whose starting point
is in $\bar{D}$.
Also $y_t \in \bar{D}$~$(0\le t\le T)$ and $\phi_t$ is
the bounded variation term.
 The mapping $L : x\mapsto \phi$ satisfies
 Assumption~\ref{assumption on A}.
 Using this result, we can apply the main theorem to reflected
 rough differential equations.
 
\medskip

 \noindent
 $(2)$
 Let $f_i$ $(1\le i\le n)$ be Lipschitz functions on $\RR^n$
and define
\begin{align}
A(x)_t&=
\left(\max_{0\le s\le t}f_1(x_s),\ldots,
\max_{0\le s\le t}f_n(x_s)\right),
\quad x\in C([0,T],\RR^n).
\end{align}
 This satisfies Assumption~\ref{assumption on A}.
Actually this satisfies the stronger conditions
${\rm (Lip)}_{\rho}$ and 
${\rm (BV)}_{\rho}$ for certain
$\rho$ in Definition~\ref{definition of lip and bv}.
See Proposition~\ref{explicit rho} for the proof.
Note that even if we replace each $\max_{0\le s\le t}f_i(x_s)$
by finite products of maximum functions and minimum functions
of $f(x_s)$, Assumption~\ref{assumption on A} holds.

\noindent
(3)
 Let $c_1,\ldots,c_n$
 be $\beta$-H\"older continuous paths on $\RR^n$
 in usual sense.
Let $f$ be a Lipschitz map from $\RR^n$ to $\RR^n$.
 Let us consider a variant of the example (2) as follows:
\begin{align*}
 A(x)_t=
\left(\max_{0\le s\le t} |f(x_s)-c_1(s)|,\ldots,
\max_{0\le s\le t}|f(x_s)-c_n(s)|\right).
\end{align*}
 This does not satisfy Assumption~\ref{assumption on A} (3).
 However it holds that
 \begin{align*}
  \|A(x)\|_{1\hyp var, [s,t]}\le
C  \left(\|x\|_{\infty\hyp var,[s,t]}+|t-s|^{\beta}\right)
  \quad 0\le s\le t\le T
 \end{align*}
 for some positive constant $C$.
 
\noindent
(4) We consider the case
$\omega(s,t)=|t-s|$, that is,
usual $\beta$-H\"older rough path.
Path-dependent functional $A(x)_t$ which we are mainly concerned with
in this paper is a
kind of generalization of the maximum process $\max_{0\le s\le t}x_s$
and the local time term $L(x)_t$.
The maximum process $\max_{0\le s\le t}|x_s|$ is obtained as the
limit of $\|x\|_{L^p([0,t])}$ as $p\to\infty$.
Hence it may be natural to study the case where
$A(x)_t=\|x\|_{L^p([0,t])}$.
Theorem~\ref{main theorem} cannot be applied to this directly.
We will study this example in Subsection~\ref{subsection 5.3}.

 \noindent
 (5) Let $W_t$ be the $1$-dimensional
 standard Brownian motion starting at $0$.
 Let us consider the following equations,
\begin{align}
  Y_t&=\xi+\int_0^t\sigma(Y_s)dW_s+a\sup_{0\le s\le t}Y_s,
  \label{si-bm0}\\
	  Y_t&=\xi+\int_0^t\sigma(Y_s)dW_s
	  +a\sup_{0\le s\le t}Y_s+\Phi_t,~~\xi\ge 0,~~
  Y_t\ge 0~~\mbox{for all $t$}.
  \label{ml-bm0}
 \end{align}
 Here $a$ denotes a real number.
 
 The equation (\ref{ml-bm0}) contains the local time term
 $\Phi_t$ at $0$.
  These processes have been studied
{\it e.g.} in 
\cite{cpy,chaumont-doney,davis,davis2, doney-zhang,perman-werner,yue-zhang}.
  We see that a multidimensional version of these equations can be
 transformed to the equation of the form $(\ref{path-rde1})$ in
 Section~${\ref{section of perturbed diffusion}}$.
 We also give some brief review of $1$-dimensional cases there.
\end{exm}

\section{Proof of Theorem~\ref{main theorem}}\label{proof of main
 theorem}
 In the calculation below, we assume $\gamma\le 3$ as well as
$\gamma>1/\beta$.

 If we write $A(Z)_t=\Psi_t$, then the equation (\ref{path-rde1})
 reads
\begin{align}
Z_t&=\xi+\int_0^t\sigma(Z_s,\Psi_s)d\mathbf{X}_s,\label{Z equation}\\
 \Psi_t&=A\left(\xi+\int_0^{\cdot}\sigma(Z_s,\Psi_s)d\mathbf{X}_s\right)_t
 \label{Psi equation}.
\end{align}
We solve this equation by using Schauder's fixed point theorem.
First, we give an estimate of the integral
 $\int_s^t\Psi_{s,r}\otimes dx_r$~$(0\le s<t\le T)$, where
 $x\in \C^{\theta}$, $\Psi\in \C^{q\hyp var, \theta'}$
 and $\otimes$ denotes the tensor product.
To this end, we introduce some notations.
Let $0\le s\le t\le T$ and consider a mapping 
$F$ defined on
$\{(u,v)~|~s\le u\le v\le t\}$
with values in a certain vector space.
   Let ${\cal P}=\{s=t_0<\cdots<t_{N}\le t\}$ be a partition of
  $[s,t]$.
  We write
\begin{align*}
 \sum_{\mathcal{P}}F(u,v)=\sum_{i=1}^NF(t_{i-1},t_i).
\end{align*}

We use the following estimate.

  \begin{lem}\label{estimate of phi and w}
   Let $x\in \C^{\theta}(\RR^n)$.
   Let $p$ be a positive number such that
   $\theta p>1$.
   Let $q$ be a positive number such that
   $1/p+1/q\ge 1$ and
   $\Psi\in \C^{q\hyp var, \theta'}(\RR^n)$.
   For any $0\le s<t\le T$,
the integral $\int_s^t\Psi_{s,r}\otimes dx_r$ converges 
in the sense of Young integral and
it holds that
\begin{align}
 \left|\int_s^t\Psi_{s,r}\otimes dx_r\right|&
  \le
C_{\theta,q}\|\Psi\|_{q\hyp var, \theta'}
   \|x\|_{\theta}\omega(s,t)^{\theta+\theta'},
\end{align}
where 
$C_{\theta,q}=2^{\theta+\frac{1}{q}}\zeta\left(\theta+\frac{1}{q}\right)$.
\end{lem}

 \begin{proof}
The assumption implies $x$ is finite $1/\theta$-variation.
Moreover $\theta+1/q>1$ holds. Hence the Young integral
of $\int_s^t\Psi_{s,r}\otimes dx_r$ converges and the following estimate holds:
\begin{align*}
 \left|\int_s^t\Psi_{s,r}\otimes dx_r\right|
&\le C_{\theta,q}\|\Psi\|_{q\hyp var,[s,t]}
\|x\|_{1/\theta\hyp var, [s,t]}
\le C_{\theta,q}\|\Psi\|_{q\hyp var, \theta'}\|x\|_{\theta}
\omega^{\theta+\theta'}(s,t),
\end{align*}
which completes the proof.
   \end{proof}

By using this lemma, 
we will give estimates for the integral
$\int_s^t\sigma(Z_u,\Psi_u)d\mathbf{X}_u$.
As we mentioned,
we denote the derivative of $\sigma=\sigma(\xi,\eta)$~$(\xi\in \RR^n,
\eta\in \RR^n)$
with respect to $\xi$ by $D_1\sigma$ and $\eta$ by $D_2\sigma$.
Also we write
$D\sigma(\xi,\eta)(u,v)=D_1\sigma(\xi,\eta)u+D_2\sigma(\xi,\eta)v$.
We write $Y_t=(Z_t,\Psi_t)\in \RR^n\times \RR^n$.
Let $(Z,Z')\in {\mathscr D}^{2\alpha}_X(\RR^n)$
and $\Psi\in \C^{q\hyp var,\ta}(\RR^n)$.

  Until the end of this section, we choose and fix $p>0$ such that
 $1/\beta<p<\gamma$.
For this $p$, we assume $q,\alpha,\ta$ satisfy the following condition.
\begin{align}
q\ge 1,\quad \frac{1}{p}+\frac{1}{q}\ge 1,\quad 
\alpha p>1,\quad
\frac{1}{3}<\alpha\le\ta\le\beta.\label{condition on qata}
\end{align}
As we explained, we consider
\begin{align}
\Xi_{s,t}&=
\sigma(Y_s)X_{s,t}+(D_1\sigma)(Y_s)Z'_s\XX_{s,t}+
(D_2\sigma)(Y_s)\int_s^t\Psi_{s,r}\otimes dX_r.\label{Xi}
\end{align}
By a simple calculation, we have for $s<u<t$,
\begin{align}
(\delta\Xi)_{s,u,t}&:=
\Xi_{s,t}-\Xi_{s,u}-\Xi_{u,t}\nn\\
\hspace{-2.2cm}&=
-\left(\int_0^1(D_1\sigma)(Y_s+\theta
 Y_{s,u})\right)\left(R^Z_{s,u}\otimes X_{u,t}\right)\nn\\
&\quad +
\left\{(D_1\sigma)(Y_s)-\int_0^1(D_1\sigma)(Y_s+\theta Y_{s,u})d\theta\right\}
\left((Z'_sX_{s,u})\otimes X_{u,t}\right)\nn\\
&\quad+
\left\{(D_2\sigma)(Y_s)-\int_0^1(D_2\sigma)(Y_s+\theta Y_{s,u})d\theta\right\}
\left(\Psi_{s,u}\otimes X_{u,t}\right)\nn\\
&\quad +\left((D_1\sigma)(Y_s)Z_s'-(D_1\sigma)(Y_u)Z'_u\right)\XX_{u,t}\nn\\
&\quad +\left((D_2\sigma)(Y_s)-(D_2\sigma)(Y_u)\right)
\int_u^t\Psi_{u,r}\otimes dX_r.
\label{explicit form of delta}
\end{align}
Thus, under the assumption on
$Z, \Psi$, applying
Lemma~\ref{estimate of phi and w} to the case
$\theta=\beta$, $\theta'=\ta$ and
$(a+b+c)^{\gamma-2}\le 
3^{\gamma-2}(a^{\gamma-2}+b^{\gamma-2}+c^{\gamma-2})$,
we obtain 

\begin{align}
 \lefteqn{\left|\left(\delta \Xi\right)_{s,u,t}\right|}\nn\\
&\le
 \|D_1\sigma\|_{\infty}\|R^Z\|_{2\alpha}\|X\|_{\beta}\omega(s,t)^{\beta+2\alpha}\nn\\
&
\quad +\|D\sigma\|_{\gamma-2}|Y_{s,u}|^{\gamma-2}
\left\{\|Z'\|_{\infty}\|X\|_{\beta}
\omega(s,u)^{\beta}+\|\Psi\|_{q\hyp var,\ta}\omega(s,u)^{\ta}
\right\}\|X\|_{\beta}\omega(u,t)^{\beta}\nn\\
&\quad +\left\{\|D_1\sigma\|_{\infty}\|Z'\|_{\alpha}\omega(s,u)^{\alpha}+
\|D_1\sigma\|_{\gamma-2}|Y_{s,u}|^{\gamma-2}\|Z'\|_{\infty}\right\}
\|\XX\|_{2\beta}\omega(u,t)^{2\beta}\nn\\
&\quad +C_{\beta,q}
\|D_2\sigma\|_{\gamma-2}|Y_{s,u}|^{\gamma-2}
\|\Psi\|_{q\hyp var,\ta}\|X\|_{\beta}\omega(u,t)^{\ta+\beta}\nn\\
&
\le
\|D\sigma\|_{\infty}\|R^Z\|_{2\alpha}\|X\|_{\beta}\omega(s,t)^{\beta+2\alpha}
+\|D\sigma\|_{\infty}\|Z'\|_{\alpha}\|\XX\|_{2\beta}\omega(s,t)^{\alpha+2\beta}\nn\\
& \quad +
C\|D\sigma\|_{\gamma-2}
\Bigl\{
\left(\|Z'\|_{\infty}
\|X\|_{\beta}\omega(s,t)^{\beta-\alpha}\right)^{\gamma-2}
+\left(\|R^Z\|_{2\alpha}\omega(s,t)^{\alpha}\right)^{\gamma-2}\nn\\
&\qquad \qquad\qquad \qquad \qquad +
\left(\|\Psi\|_{q\hyp var,\ta}\omega(s,t)^{\ta-\alpha}\right)^{\gamma-2}
\Bigr\}\cdot\nn\\
&\Bigl\{
\left(
\|Z'\|_{\infty}\|X\|_{\beta}\omega(s,t)^{\beta}+
\|\Psi\|_{q\hyp var,\ta}\omega(s,t)^{\ta}
\right)\|X\|_{\beta}\omega(s,t)^{\beta}
+\|Z'\|_{\infty}\|\XX\|_{2\beta}\omega(s,t)^{2\beta}
 \Bigr\}\nn\\
 &\qquad\qquad\qquad\cdot
\omega(s,t)^{\alpha(\gamma-2)}
,\label{estimate of delta 0}
\end{align}
where $C=3^{\gamma-2}(2+C_{\beta,q})$.
Therefore, there exists a positive constant $C$ which depends only on
$\gamma,\beta,p$ such that
\begin{align}
\left|\left(\delta\Xi\right)_{s,u,t}\right|
&\le
CK_{\sigma}
f\left(
\|R^Z\|_{2\alpha},\|Z'\|_{\alpha},
\|Z'\|_{\infty},
\|\Psi\|_{q\hyp var,\ta}\right)
\wopnorm\left(1+\omega(s,t)^{1/2}\right)
\omega(s,t)^{\beta+(\gamma-1)\alpha}\nn\\
&\qquad\quad\quad 0\le s\le t\le T\label{estimate of delta}
\end{align}
where
\begin{align}
K_{\sigma}&=\|D\sigma\|_{\gamma-2}+\|D\sigma\|_{\infty},\\
f(a,b,c,d)&=a+b+
\left(a^{\gamma-2}+c^{\gamma-2}+d^{\gamma-2}\right) (c+d).
\end{align}

Let ${\cal P}=\{t_k\}_{k=0}^N$ be a partition of
$[s,t]$. 
Since $\beta+(\gamma-1)\alpha>\beta+p\alpha-\alpha\ge p\alpha>1$,
by the Sewing lemma (see {\it e.g.}\cite{lq, friz-victoir, friz-hairer}), 
the following limit exists,
\begin{align}
I((Z,Z'),\Psi)_{s,t}&:=
\lim_{|{\cal P}|\to 0}\sum_{{\cal P}}\Xi_{u,v}.
\end{align}
We may denote $I\left((Z,Z'),\Psi\right)$ by
\begin{align}
 I(Z,\Psi)_{s,t}\quad \mbox{or}\quad \int_s^t\sigma(Z_u,\Psi_u)d\mathbf{X}_u
\end{align}
simply if there are no confusion.
This integral satisfies the additivity property 
\begin{align}
I(Z,\Psi)_{s,u}+I(Z,\Psi)_{u,t}=I(Z,\Psi)_{s,t}\qquad
0\le s\le u\le t\le T.
\end{align}
The pair $\left(I(Z,\Psi), \sigma(Y_t)\right)$ is
actually a controlled path of $X$.
In fact, we have the following estimates.

\begin{lem}\label{estimate of rough integral}
Assume
$(Z,Z')\in {\mathscr D}^{2\alpha}_X(\RR^n)$
and $\Psi\in \C^{q\hyp var,\ta}(\RR^n)$ and
$q, \alpha, \ta$ satisfy $(\ref{condition on qata})$.
For any $0\le s\le t\le T$, we have the following estimates.
The constant $K$ below depends only on
 $\|\sigma\|_{\infty}$,
 $\|D\sigma\|_{\infty}, \|D\sigma\|_{\gamma-2}, \alpha,\beta, p,\gamma$
and may change line by line.
 
\begin{enumerate}
 \item[$(1)$] 
\begin{align}
|\Xi_{s,t}|&\le
\Bigl\{\|\sigma\|_{\infty}\|X\|_{\beta}
+\|D\sigma\|_{\infty}\|Z'\|_{\infty}
\|\XX\|_{2\beta}\omega(s,t)^{\beta}\nn\\
&\qquad
+C_{\beta,q}
\|D\sigma\|_{\infty} \|\Psi\|_{q\hyp var,\ta}
\|X\|_{\beta}\omega(s,t)^{\ta}\Bigr\}\omega(s,t)^{\beta}.
\end{align}
\item[$(2)$] 
\begin{align}
\lefteqn{\left|I(Z,\Psi)_{s,t}-\Xi_{s,t}\right|}\nn\\
&\le K 
f\left(\|R^Z\|_{2\alpha}, \|Z'\|_{\alpha},
\|Z'\|_{\infty}, \|\Psi\|_{q\hyp var,\ta}\right)
\wopnorm\left(1+\omega(s,t)^{1/2}\right)
\omega(s,t)^{\gamma\alpha+\beta-\alpha}
\label{approximation of integral}
\end{align}
and
\begin{align}
&\|I(Z,\Psi)\|_{\beta}
\le
\|\sigma\|_{\infty}\|X\|_{\beta}\nonumber\\
&\quad+
K 
g\left(\|R^Z\|_{2\alpha}, \|Z'\|_{\alpha},
 \|Z'\|_{\infty}, \|\Psi\|_{q\hyp var,\ta}\right)
\left(1+\omega(0,T)^{1/2}
+\omega(0,T)^{\alpha(\gamma-1)-\ta+\frac{1}{2}}\right)\nonumber\\
&\qquad\qquad\times\omega(0,T)^{\ta}
\wopnorm,
\label{beta Holder continuity}
\end{align}
where 
\begin{align}
f(a,b,c,d)&=a+b+(a^{\gamma-2}+c^{\gamma-2}+d^{\gamma-2})(c+d),\\
g(a,b,c,d)&=
f(a,b,c,d)+
c+d.
\end{align}
\item[$(3)$] 
\begin{align}
 \lefteqn{\left|I(Z,\Psi)_{s,t}-\sigma(Y_s)X_{s,t}\right|}\nn\\
&\le 
\Biggl\{
K 
f\left(\|R^Z\|_{2\alpha}, \|Z'\|_{\alpha},
\|Z'\|_{\infty}, \|\Psi\|_{q\hyp var,\ta}\right)
\wopnorm\left(1+\omega(s,t)^{1/2}\right)
\omega(s,t)^{\gamma\alpha-2\ta+\beta-\alpha}\nn\\
&\quad\quad +
\|D\sigma\|_{\infty}
\|Z'\|_{\infty}
\|\XX\|_{2\beta}\omega(s,t)^{2(\beta-\ta)}\nn\\
&\quad\quad
+C_{\beta,q}
\|D\sigma\|_{\infty}\|\Psi\|_{q\hyp var,\ta}
\|X\|_{\beta}\omega(s,t)^{\beta-\ta}
\Biggr\}\omega(s,t)^{2\ta}.\label{RZ}
\end{align}
\item[$(4)$]
\begin{align}
\lefteqn{|\sigma(Y_t)-\sigma(Y_s)|}\nn\\
&\le \|D\sigma\|_{\infty}
\left\{\|Z'\|_{\infty}
\|X\|_{\beta}
\omega(s,t)^{\beta-\ta}+
\|R^Z\|_{2\alpha}\omega(s,t)^{2\alpha-\ta}
+\|\Psi\|_{q\hyp var,\ta}\right\}\omega(s,t)^{\ta}.
\end{align}
\item[$(5)$] $(I(Z,\Psi),\sigma(Z,\Psi))\in {\mathscr D}^{2\ta}_X$
holds.
\end{enumerate}
\end{lem}

\begin{rem}\label{remark on estimate of rough integral}
 (1)
 Under the condition $(\ref{condition on qata})$,
 $(\gamma-1)\alpha+\beta>1$ holds
 as we noted.

 \noindent
(2)
 If $\Psi\in {\cal C}^{1\hyp var,\beta}$, then
 $I(Z,\Psi)\in {\mathscr D}^{2\beta}_X$.

\noindent
(3) 
We give estimates of paths on $[0,T]$ in
Lemma~\ref{estimate of rough integral}.
However, a similar estimate holds on small interval
$[0,\tau]$ $(0<\tau<T)$ by replacing
the norms and $\omega(0,T)$ in Lemma~\ref{estimate of rough integral}
by the norms on $[0,\tau]$ and $\omega(0,\tau)$.

\noindent
(4)
Let $1/3<\tb<\beta$.
Then $\mathbf{X}$ can be regarded as a $1/\tb$-rough path.
It is easy to check that
Lemma~\ref{estimate of rough integral} still holds
under the condition (\ref{condition on qata}) by replacing $\beta$ by $\tb$.
Suppose $\omega(0,T)\le 1$.
Then $\|X\|_{\tb}\le \|X\|_{\beta}$ and
$\|\XX\|_{2\tb}\le \|\XX\|_{2\beta}$ holds.
We use these results to prove a priori estimate in
Theorem~\ref{main theorem}.
\end{rem}

\begin{proof}
(1)\, This follows from the explicit form of (\ref{Xi})
and Lemma~\ref{estimate of phi and w}.

\noindent
(2)\, This follows from (\ref{estimate of delta}) and the Sewing lemma.

\noindent
(3)\, This follows from (2) and Lemma~\ref{estimate of phi and w}.

\noindent
(4)\, This follows from the definition of $Y_t$.

\noindent
(5)\, This follows from (3) and (4) and $2\alpha\ge \ta$.
\end{proof}

We consider the product Banach space
$
{\mathscr D}^{2\theta_1}_{X}\times
\C^{q\hyp var,\theta_2},
$
where $1/3<\theta_1\le 1/2$ and $0<\theta_2\le 1$.
The norm is defined by
\begin{align}
\|((Z,Z'),\Psi)\|= 
|Z_0|+|Z'_0|+\|Z'\|_{\theta_1}+\|R^Z\|_{2\theta_1}+
|\Psi_0|+\|\Psi\|_{q\hyp var,\theta_2}.
\end{align}
Let $\xi$ be the starting point of $Z$ and let
$\eta=A(x)_0\in \RR^n$.
Note that $\eta$ is independent of $x$.
Let
\begin{align}
 \W_{T,\theta_1,\theta_2,q,\xi,\eta}&=
\Big\{\left((Z,Z'), \Psi\right)\in {\mathscr D}^{2\theta_1}_X\times
\C^{q\hyp var,\theta_2}\,|\, Z_0=\xi, Z'_0=\sigma(\xi,\eta), \Psi_0=\eta\Big\}.
\end{align}
The solution of RDE could be obtained as a fixed point of the mapping,
\begin{align}
{\cal M} &: \left((Z,Z'),\Psi\right)\left(\in \W_{T,\alpha,\ta,q,\xi,\eta}
\right)
\mapsto \left((\xi+I(Z,\Psi),\sigma(Y)),A(\xi+I(Z,\Psi))\right)
(\in \W_{T,\alpha,\ta,q,\xi,\eta}).
\end{align}

We prove a continuity property
of ${\cal M}$.

\begin{lem}[Continuity]\label{continuity}
Assume 
\begin{align}
\frac{1}{3}<\zb\le \alpha<\ta<\beta,\quad
\alpha p>1,\quad
 1<q<\min\left(\frac{p}{p-1}, \,\,\frac{\beta}{\ta}\right),
\label{condition on qalphata}
\end{align}
 where $\zb$ is the number in
 Assumption~$\ref{assumption on A}$.
Then ${\cal M}$ is continuous.
\end{lem}

We already proved the compactness of the
inclusion $\C^{q'\hyp var,\theta'}\subset \C^{q\hyp var,\theta}$,
where $1\le q'<q,\, q\theta\le q'\theta'$.
We need the following compactness result also.

\begin{lem}\label{compact criterion}
 Let $\frac{1}{3}<\theta<\theta'\le \frac{1}{2}$.
Then ${\mathscr D}_X^{2\theta'}\subset {\mathscr D}^{2\theta}_X$
and the inclusion is compact.
\end{lem}

\begin{proof}[Proof of Lemma~$\ref{compact criterion}$]
Suppose 
\begin{align}
 \sup_n\|(Z(n),Z(n)')\|_{\theta'}=
\sup_n\{|Z(n)_0|+|Z(n)'_0|+
\|Z(n)'\|_{\theta'}+\|R^{Z(n)}\|_{2\theta'}\}<\infty.
\end{align}
This implies $\{Z(n)'\}$ is bounded and equicontinuous.
Since $Z(n)_t-Z(n)_s=Z(n)'_sX_{s,t}+R^{Z(n)}_{s,t}$,
$\{Z(n)\}$ is also bounded and equicontinuous.
Hence certain subsequence $\{Z(n_k), Z(n_k)'\}$
converges uniformly.
This implies
$\{(Z(n_k)',R^{Z(n_k)})\}$ converges in ${\mathscr D}^{2\theta}_X$.
\end{proof}

\begin{proof}[Proof of Lemma~$\ref{continuity}$]
First note that 
\begin{align}
 {\cal M}(\W_{T,\alpha,\ta,q,\xi,\eta})
\subset \W_{T,\alpha,\ta,q,\xi,\eta}.\label{invariant set}
\end{align}
$
(\xi+I(Z,\Psi),\sigma(Y))\in
\mathscr{D}^{2\alpha}_X
$ 
follows from Lemma~\ref{estimate of rough integral}.
By Assumption~\ref{assumption on A},
we have
\begin{align*}
 \|A(\xi+I(Z,\Psi))\|_{q\hyp var,[s,t]}&\le
 \|A(\xi+I(Z,\Psi))\|_{1\hyp var,[s,t]}\\
&\le F\left(\|I(Z,\Psi))\|_{1/\beta\hyp var,[s,t]}\right)
\|I(Z,\Phi)\|_{\infty\hyp var, [s,t]}\\
&\le F\left(\|I(Z,\Psi)\|_{\beta}\omega(s,t)^{\beta}\right)
\|I(Z,\Psi)\|_{\beta}\omega(s,t)^{\beta},
\end{align*}
which shows
\begin{align}
 \|A(\xi+I(Z,\Psi))\|_{q\hyp var,\ta}&\le
F\left(\|I(Z,\Psi)\|_{\beta}\omega(0,T)^{\beta}\right)
\|I(Z,\Psi)\|_{\beta}\omega(0,T)^{\beta-\ta}<\infty.
\label{estimate of A(I)}
\end{align}
Thus we have proved (\ref{invariant set}).
We estimate $\|I(Z,\Psi)'-I(\tZ,\tPsi)'\|_{\alpha}$.
We have
\begin{align}
\lefteqn{\left|\left(\sigma(Y_t)-\sigma(\tY_t)\right)-
\left(\sigma(Y_s)-\sigma(\tY_s)\right)\right|}\nn\\
&=\int_0^1
\Bigl\{
(D\sigma)(Y_s+\theta Y_{s,t})(Y_{s,t})-
(D\sigma)(\tY_s+\theta \tY_{s,t})(\tY_{s,t})
\Bigr\}\nn\\
&\le
\|D\sigma\|_{\infty}|Y_{s,t}-\tY_{s,t}|+
\|D\sigma\|_{\gamma-2}2^{\gamma-2}
\left(|Y_s-\tY_s|^{\gamma-2}+
|Y_{s,t}-\tY_{s,t}|^{\gamma-2}
\right)|Y_{s,t}|\nn\\
&\le\|D\sigma\|_{\infty}
\left(\|Z'-\tZ'\|_{\alpha}\omega(0,s)^{\alpha}\|X\|_{\beta}\omega(s,t)^{\beta}+
\|R^Z-R^{\tZ}\|_{2\alpha}\omega(s,t)^{2\alpha}+
\|\Psi-\tPsi\|_{q\hyp var,\ta}\omega(s,t)^{\ta}\right)\nn\\
&\quad+2^{\gamma-2}\|D\sigma\|_{\gamma-2}
\Bigl\{\Bigl(
 \|R^Z-R^{\tZ}\|_{2\alpha}\omega(0,s)^{2\alpha}
 +
\|\Psi-\tPsi\|_{q\hyp var,\ta}\omega(0,s)^{\ta}\Bigr)^{\gamma-2}\nn\\
&\quad+
\Bigl(
\|Z'-\tZ'\|_{\alpha}\omega(0,s)^{\alpha}\|X\|_{\beta}\omega(s,t)^{\beta}+
\|R^Z-R^{\tZ}\|_{2\alpha}\omega(s,t)^{2\alpha}+
\|\Psi-\tPsi\|_{q\hyp var,\ta}\omega(s,t)^{\ta}
\Bigr)^{\gamma-2}
\Bigr\}
\nn\\
&\quad\quad \times
\left((|\sigma(\xi)|+\|Z'\|_{\alpha}\omega(0,s)^{\alpha})
\|X\|_{\beta}\omega(s,t)^{\beta}+
\|R^Z\|_{2\alpha}\omega(s,t)^{2\alpha}+
\|\Psi\|_{q\hyp var,\ta}\omega(s,t)^{\ta}\right).
\end{align}
Since $\beta>\ta>\alpha$,
this shows the continuity of
the mapping $((Z,Z'),\Psi)\mapsto 
I(Z,\Psi)'$.

We next estimate 
$\|R^{I(Z,\Psi)}-R^{I(\tZ,\tPsi)}\|_{2\alpha}$.

\begin{align}
|R^{I(Z,\Psi)}_{s,t}-R^{I(\tZ,\tPsi)}_{s,t}|
&=\left|\left(
I(Z,\Psi)_{s,t}-\sigma(Y_s)X_{s,t}\right)-
\left(I(\tZ,\tPsi)_{s,t}-\sigma(\tY_s)X_{s,t}\right)
\right|\nn\\
&\le\left|\left(
I(Z,\Psi)_{s,t}-\Xi(Z,\Psi)_{s,t}\right)-
\left(I(\tZ,\tPsi)_{s,t}-\Xi(\tZ,\tPsi)_{s,t}\right)
\right|\nn\\
&\quad +
\left|(D_1\sigma)(Y_s)(Z'_s\XX_{s,t})-
(D_1\sigma)(\tY_s)(\tZ '_s\XX_{s,t})\right|\nn\\
&\quad +
\left|(D_2\sigma)(Y_s)\left(\int_s^t\Psi_{s,u}\otimes d\mathbf{X}_u\right)
-(D_2\sigma)(\tY_s)\left(\int_s^t\tPsi_{s,u}\otimes
 d\mathbf{X}_u\right)\right|.
\end{align}
 
  We argue in a similar way to the sewing lemma
 for the estimate of the first term.
Let ${\cal P}_N=\{t^N_k=s+\frac{k(t-s)}{2^N}\}$ 
be a usual dyadic partition of $[s,t]$.
 We have
\begin{align}
 \lefteqn{\left|\left(
I(Z,\Psi)_{s,t}-\Xi(Z,\Psi)_{s,t}\right)-
\left(I(\tZ,\tPsi)_{s,t}-\Xi(\tZ,\tPsi)_{s,t}\right)
\right|}\nn\\
&\le
\left|\left(\sum_{{\cal P}_N}\Xi(Z,\Psi)_{u,v}-
\Xi(Z,\Psi)_{s,t}\right)-
\left(\sum_{{\cal P}_N}\Xi(\tZ,\tPsi)_{u,v}-
\Xi(\tZ,\tPsi)_{s,t}\right)
\right|\nn\\
&\quad +
\left|\left(I(Z,\Psi)_{s,t}-
\sum_{{\cal P}_N}\Xi(Z,\Psi)_{u,v}\right)\right|\nn\\
&\quad +\left|\left(I(\tZ,\tPsi)_{s,t}-
\sum_{{\cal P}_N}\Xi(\tZ,\tPsi)_{u,v}\right)\right|.
\end{align}
By (\ref{approximation of integral}),
\begin{align}
\lefteqn{
\left|\left(I(Z,\Psi)_{s,t}-
\sum_{{\cal P}_N}\Xi(Z,\Psi)_{u,v}\right)\right|
+\left|\left(I(\tZ,\tPsi)_{s,t}-
\sum_{{\cal P}_N}\Xi(\tZ,\tPsi)_{u,v}\right)\right|}
\nn\\
&\le K
\left\{
f(\|R^Z\|_{2\alpha},\|Z'\|_{\alpha},\|Z'\|_{\infty},
\|\Psi\|_{q\hyp var,\ta})+
f(\|R^{\tZ}\|_{2\alpha},\|{\tZ}'\|_{\alpha},\|{\tZ}'\|_{\infty},
\|\Psi'\|_{q\hyp var,\ta})
\right\}\nn\\
&\quad\qquad\qquad
\times (1+\omega(s,t)^{1/2})
\wopnorm
\max_{{\cal P}_N}\omega(u,v)^{(\gamma-1)\alpha+\beta-1}K(T)
\omega(s,t).
\end{align}
Hence this term is small in the $\omega$-H\"older space
$\C^{2\alpha}$ on a bounded set of
$\W_{T,\alpha,\ta,q,\xi,\eta}$
if $N$ is large.
We fix a partition so that this term is small.
Although the partition number may be big,
\begin{align}
\lefteqn{\left(
\sum_{{\cal P}_N}\Xi(Z,\Psi)_{u,v}
-\Xi(Z,\Psi)_{s,t}\right)-
\left(\sum_{{\cal P}_N}\Xi(\tZ,\tPsi)_{u,v}
-\Xi(\tZ,\tPsi)_{s,t}\right)}\nn\\
&=
\sum_{k=0}^N\sum_{{\cal P}_k}
\left(\delta\Xi(Z,\Psi)_{u,(u+v)/2,v}
-\delta\Xi(\tZ,\tPsi)_{u,(u+v)/2,v}\right)
\end{align}
is a finite sum,
and by the explicit form of $\delta\Xi$ as in 
(\ref{explicit form of delta}), we see that this difference 
is small in $\C^{2\alpha}$ if
$((Z,Z'), \Psi)$ and $((\tZ,\tZ'), \tPsi)$ are sufficiently close in 
$\W_{T,\alpha,\ta,q,\xi,\eta}$.
The estimate of the second and the 
third terms are similar to the above
and we obtain the continuity 
of the mapping
\begin{align}
 ((Z,Z'),\Psi)(\in \W_{T,\alpha,\ta,q,\xi,\eta})
\mapsto \left(\xi+I(Z,\Psi),\sigma(Y)\right)(\in 
{\mathscr D}^{2\alpha}_X).\label{continuity of I}
\end{align}
We next prove the continuity of the mapping
\begin{align}
 ((Z,Z'),\Psi)(\in \W_{T,\alpha,\ta,q,\xi,\eta})
\mapsto A(\xi+I(Z,\Psi))(\in \C^{q\hyp var,\ta}).
\end{align}

Since we choose $\zb\le \alpha$,
it suffices to apply Lemma~\ref{lemma for A}
 (2) to the case where
 $x=\xi+I(Z,\Psi)$ and $x'=\xi+I(\tZ,\tPsi)$
 because of Lemma~\ref{estimate of rough integral} (2) and
 the continuity (\ref{continuity of I}).
\end{proof}

By using the above lemmas, we prove the existence of solutions 
on small interval $[0,T']$.
 Since the interval can be chosen independent of
the initial condition, we obtain the global existence
of solutions and the estimate for solutions.
We consider balls with radius $1$ centered at
$\left((\xi+\sigma(\xi,\eta)X_{t}, \sigma(\xi,\eta)), \eta\right)$
\, $(0\le t\le T')$,
\begin{align}
\B_{T',\theta_1,\theta_2,q}=\left\{
((Z,Z'), \Psi)\in \W_{T',\theta_1,\theta_2,q,\xi,\eta}\,|\,
\|Z'\|_{\theta_1,[0,T']}+
\|R^Z\|_{2\theta_1,[0,T']}+
\|\Psi\|_{q\hyp var,\theta_2, [0,T']}\le 1
\right\}.
\end{align}

\begin{lem}[Invariance and compactness]
\label{invariance and compactness}
Assume $(\ref{condition on qalphata})$ and
let $\alpha<\ua<\ta<\oa<\beta$.
Also we choose $q'>1$ such that $\displaystyle{\frac{\ta}{\oa}q<q'<q}$.
\begin{enumerate}
 \item[$(1)$] For sufficiently small $T'$,
we have
\begin{align}
{\cal M}(\B_{T',\alpha,\ta,q})\subset
\B_{T',\ua,\oa,q'}\subset \B_{T',\alpha,\ta,q}.
\end{align}
Moreover $T'$ does not depend on $\xi$.
\item[$(2)$] $\B_{T',\ua,\oa,q'}$ is a compact subset of
$\B_{T',\alpha,\ta,q}$.
\end{enumerate}
\end{lem}

\begin{proof}
(1)~
The second inclusion is immediate because 
$\omega(0,T')\le 1$ and the definition of the norms.
We prove the first inclusion.
Let $((Z,Z'),\Psi)\in \B_{T',\alpha,\ta,q}$.
Recall that $I(Z,\Psi)'_t=\sigma(Z_t,\Psi_t)$ and note that
$\|Z'\|_{\infty,[0,T']}\le \|\sigma\|_{\infty}
+\|Z'\|_{\alpha}\omega(0,T')^{\alpha}$.
From Lemma~\ref{estimate of rough integral} (4), we have
\begin{align*}
\|I(Z,\Psi)'\|_{\ua,[0,T']}
&\le
\|D\sigma\|_{\infty}
\Bigl\{\|Z'\|_{\infty,[0,T']}
\|X\|_{\beta}
\omega(0,T')^{\beta-\ua}
+\|R^Z\|_{2\alpha,[0,T']}\omega(0,T')^{2\alpha-\ua}\\
&\qquad\qquad\qquad+
\|\Psi\|_{q\hyp var,\ta,[0,T']}\omega(0,T')^{\ta-\ua}\Bigr\}\\
&\le\|D\sigma\|_{\infty}
\Bigl\{
\left(\|\sigma\|_{\infty}+1\right)\|X\|_{\beta}
+2
\Bigr\}\omega(0,T')^{\ta-\ua}
\end{align*}
 We next estimate $R^{I(Z,\Psi)}$.
 Let $0<s<t<T'$.
By Lemma~\ref{estimate of rough integral} (3), we have
\begin{align*}
&\|R^{I(Z,\Psi)}\|_{2\ua,[0,T']}\nn\\
&\le
\|D\sigma\|_{\infty}\|Z'\|_{\infty,[0,T']}
\|\XX\|_{2\beta}\omega(0,T')^{2(\beta-\ua)}\nn\\
&\quad +
C_{\beta,q}
\|D\sigma\|_{\infty}\|\Psi\|_{q\hyp var,\ta,[0,T']}\|X\|_{\beta}
\omega(0,T')^{\ta+\beta-2\ua}\nn\\
&\quad +
2Kf\left(\|R^Z\|_{2\alpha,[0,T']},
\|Z'\|_{\alpha,[0,T']}, \|Z'\|_{\infty,[0,T']},
\|\Psi\|_{q\hyp var,\ta,[0,T']}\right)
\wopnorm
\omega(0,T')^{\gamma\alpha+\beta-\alpha-2\ua}\\
&\le
\Bigl\{
\|D\sigma\|_{\infty}\left(\|\sigma\|_{\infty}+1\right)
\|\mathbb{X}\|_{2\beta}
+C_{\beta,q}\|D\sigma\|_{\infty}\|X\|_{\beta}+
2Kf(1,1,\|\sigma\|_{\infty}+1,1)
\wopnorm
\Bigr\}\\
&\qquad \times \omega(0,T')^{2(\ta-\ua)}.
\end{align*}
We turn to the estimate of $A(\xi+I(Z,\Psi))$.
By (\ref{estimate of A(I)}) and Lemma~\ref{estimate of rough integral} (2),
we have
\begin{align*}
&\|A(\xi+I(Z,\Psi))\|_{q'\hyp var,\oa,[0,T']}\\
&\le
F\left(\|I(Z,\Psi)\|_{\beta,[0,T']}\omega(0,T')^{\beta}\right)
\|I(Z,\Psi)\|_{\beta,[0,T']}
\omega(0,T')^{\beta-\oa}\nn\\
&\le F\left((1+2g(1,1,\|\sigma\|_{\infty}+1,1))K\wopnorm\right)
(1+2g(1,1,\|\sigma\|_{\infty}+1,1))K\wopnorm
\omega(0,T')^{\beta-\oa}.
\end{align*}
Thus, noting Lemma~\ref{lemma for A} (1), 
there exists a positive number $K'$ which depends on
$K$, $\|\sigma\|_{\infty}$, $\|D\sigma\|_{\infty}$, $f$, $g$ and
a positive number $\kappa_0$ which depends on
$\beta-\ua$ and $\ta-\ua$ such that
if $\omega(0,T')\le
\{K'(1+\wopnorm)\}^{-\kappa_0}$, then
$
{\cal M}(\B_{T',\alpha,\ta,q})\subset
\B_{T',\ua,\oa,q'}
$
holds.
This completes the proof.

\noindent
(2)~This follows from Lemma~\ref{comparison of norms} (2)
and Lemma~\ref{compact criterion}.
 \end{proof}

We are in a position to prove our main theorem.

 \begin{proof}[Proof of Theorem~$\ref{main theorem}$]
(1)
Let us take $\alpha, \ta, p, q, \ua, \oa$  as in 
Lemma~\ref{invariance and compactness}.
By Lemma~\ref{continuity} and
Lemma~\ref{invariance and compactness}, applying
Schauder's fixed point theorem,
we obtain a fixed point for small interval $[0,T']$
  if
$\omega(0,T')\le \{K(1+\wopnorm)\}^{-\kappa_0}$,
where $K$ is a certain positive constant.
That is, there exists a solution on $[0,T']$.
We now consider the equation on $[T',T]$.
We can rewrite the equation as
\begin{align}
 Z_{T'+t}&=Z_{T'}+\int_{T'}^{T'+t}
\sigma(Z_u,\Psi_u)d\mathbf{X}_u\qquad 0\le t\le T-T',
\label{Z equation'}\\
\Psi_{T'+t}&=A\left(\xi+
\int_0^{\cdot}\sigma(Z_u,\Psi_u)d\mathbf{X}_u\right)_{T'+t}
\qquad 0\le t\le T-T'.\label{Psi equation'}
\end{align}
Let $\omega_{T'}(s,t)=\omega(T'+s,T'+t)$
  for $0\le s<t\le T-T'$.
  We see that
  $\tilde{Z}_t:=Z_{T'+t}$ and $\tPsi_{t}:=\Psi_{T'+t}$
 $(0\le t\le T-T')$ is a solution to
 \begin{align}
 \tilde{Z}_{t}&=\tilde{Z}_0+\int_{0}^{t}
\sigma(\tilde{Z}_u,\tilde{\Psi}_u)d\mathbf{X}_{T'+u}\qquad 0\le t\le T-T',\\
  \tilde{\Psi}_{t}&=\tilde{A}_{y,T'}\left(
  \int_0^{\cdot}\sigma(\tilde{Z}_u,\tilde{\Psi}_u)d\mathbf{X}_{T'+u}\right)_{t}
\qquad 0\le t\le T-T'.
 \end{align}
 where
\begin{align*}
 y_t&=\xi+\int_0^t\sigma(Z_u,\Psi_u)d\mathbf{X}_u, \quad 0\le t\le T'.
\end{align*}
Note that we already defined
$ \tilde{A}_{y,T'}(x)_t$ $(0\le t\le T-T')$
$(x\in \C^{\beta}([0,T-T'],\RR^n
  ~|~ x_0=Z_{T'}, \omega_{T'}))$
in Lemma~\ref{property of A} (4).

Thanks to Lemma~\ref{property of A}, we can do the same argument as
 $[0,T']$ for small interval.
 By iterating this procedure finite time, say, $N$-times,
we obtain 
a controlled path $(Z_t,Z_t')$\, $(0\le t\le T)$.
This is a solution to (\ref{path-rde1}).
  Clearly,
  \begin{align}
   N\le 1+
   \omega(0,T)\{K(1+\wopnorm)\}^{\kappa_0}
   \label{N}
   \end{align}
  We need to show $(Z,\Psi)\in \W_{T,\beta,\beta,1,\xi,\eta}$
and its estimate with respect to the norm
$\|\cdot\|_{\beta}$.
We give the estimate of the solution on $[0,T']$.
The solution $(Z,Z')$ which we obtained satisfies
\begin{align}
\|Z'\|_{\alpha,[0,T']}+\|R^Z\|_{\alpha,[0,T']}+\|\Psi\|_{q\hyp var,\ta,[0,T']}
\le 1.\label{norm 1}
\end{align}
Let $0\le u\le v\le T'$.
From (\ref{norm 1}), (\ref{beta Holder continuity}) and (\ref{Z equation}),
we have
\begin{align}
 \|Z\|_{\beta,[u,v]}&\le
K\wopnorm.\label{estimate of Z}
\end{align}
Second, by (\ref{estimate of Ax}) and
(\ref{estimate of Z}),
we have
\begin{align}
\|A(Z)\|_{1\hyp var,[u,v]}&\le
F(K\wopnorm)
K\wopnorm\omega(u,v)^{\beta}.\label{estimate of AZ}
\end{align}
Therefore $Z$ and $A(Z)$ are $(\omega,\beta)$-H\"older continuous paths.
Hence, we have
$
 \|Z'\|_{\beta}\le K\wopnorm.
$
Moreover, we can apply Lemma~\ref{estimate of rough integral} to 
$Z$ and $\Phi=A(Z)$ in the case
where $\alpha=\ta=\beta$ and $q=1$.
Thus, by substituting the estimates (\ref{estimate of Z}) and
(\ref{estimate of AZ}) for (\ref{RZ}),
we obtain for $0\le u\le v\le T'$,
$
|R^Z_{u,v}|\le
K\wopnorm
\omega(u,v)^{2\beta}.
$
These local estimates hold on other small intervals.
   By the estimate (\ref{N}),
   we obtain the desired estimate.

\noindent
(2) 
Let $(Z,Z')\in \mathscr{D}^{2\beta}_X(\RR^n)$ be a solution of
(\ref{path-rde1}).
Let $\zb<\tb<\beta$.
The constants $K, \kappa_1,\kappa_2,\kappa_3$ 
which will appear in the calculation below
depend only on $\sigma$ and $F$
and may change line by line.
As we already noted in Remark~\ref{remark on estimate of rough integral} (4),
Lemma~\ref{estimate of rough integral} still holds replacing
$\beta$ by $\tb$.
We take $0<\tau\le T$ so that
$\omega(0,\tau)\le 1$.
Using $\|X\|_{\tb,[0,\tau]}\le \|X\|_{\beta,[0,\tau]}$ and
$\|\XX\|_{\tb,[0,\tau]}\le \|\XX\|_{\beta,[0,\tau]}$ which follows from
$\omega(0,\tau)\le 1$, we have
\begin{align}
 \|Z\|_{\tb,[0,\tau]}\le \|\sigma\|_{\infty}
\|X\|_{\beta,[0,\tau]}+\|R^Z\|_{2\tb,[0,\tau]}\omega(0,\tau)^{\tb}.
\label{estimate of Ztau}
\end{align}
By Lemma~\ref{lemma for A} (1),
we have
\begin{align}
 &\|A(Z)\|_{1\hyp var, [s,t]}\nonumber\\
&\le
C\left(\|Z\|_{1/\zb\hyp var, [s,t]}^{1/\zb}+1\right)
\|Z\|_{\infty\hyp var, [s,t]}\nonumber\\
&\le C\left\{\left(\|\sigma\|_{\infty}\|X\|_{\beta,[0,\tau]}+
\|R^Z\|_{2\tb,[0,\tau]}\omega(0,\tau)^{\tb}\right)^{1/\zb}
+1\right\}
\|Z\|_{\tb, [s,t]}\omega(s,t)^{\tb},
\end{align}
which implies
\begin{align}
&\|A(Z)\|_{1\hyp var, \tb, [0,\tau]}\nonumber\\
&\le
K\left(\|X\|_{\beta,[0,\tau]}^{1/\zb}+\|R^Z\|_{2\tb,[0,\tau]}^{1/\zb}
\omega(0,\tau)^{\tb/\zb}+1\right)
\left(\|X\|_{\beta,[0,\tau]}+\|R^Z\|_{2\tb, [0,\tau]}
\omega(0,\tau)^{\tb}\right)\nonumber\\
&\le
K\Biggl\{\|X\|_{\beta,[0,\tau]}+\|X\|_{\beta,[0,\tau]}^2
+\|X\|_{\beta,[0,\tau]}^{2/\zb}+
\|R^Z\|_{2\tb, [0,\tau]}\omega(0,\tau)^{\tb}\nonumber\\
&\qquad+\left(\|R^Z\|_{2\tb, [0,\tau]}\omega(0,\tau)^{\tb}\right)^2
+\left(\|R^Z\|_{2\tb,[0,\tau]}^{1/\zb}\omega(0,\tau)^{\tb/\zb}\right)^2
\Biggr\}
.\label{estimate of AZtau}
\end{align}
By (\ref{estimate of Ztau}) and (\ref{estimate of AZtau}), we obtain
\begin{align}
 \|Z'\|_{\tb,[0,\tau]}
&=\|\sigma(Z,A(Z))\|_{\tb,[0,\tau]}\nonumber\\
&\le 
K\Biggl\{\|X\|_{\beta,[0,\tau]}+\|X\|_{\beta,[0,\tau]}^2
+\|X\|_{\beta,[0,\tau]}^{2/\zb}+
\|R^Z\|_{2\tb, [0,\tau]}\omega(0,\tau)^{\tb}\nonumber\\
&\qquad+\left(\|R^Z\|_{2\tb, [0,\tau]}\omega(0,\tau)^{\tb}\right)^2
+\left(\|R^Z\|_{2\tb,[0,\tau]}^{1/\zb}\omega(0,\tau)^{\tb/\zb}\right)^2
\Biggr\}
\label{estimate of Zdashtau}
\end{align}
We apply Lemma~\ref{estimate of rough integral} (3)
to the estimate of $R^Z$
in the case where $\Psi=A(Z)$, $q=1$ and $\alpha=\ta=\tb$.
By combining the estimates obtained above, we see that there exist
$\kappa_1>0, \kappa_2>1,\kappa_3>0$ and $K>0$ which can be taken
independent of $\tb$ such that
\begin{align}
 \|R^Z\|_{2\tb,[0,\tau]}&\le
K\Biggl\{\wopnorm^{\kappa_1}+
\|R^Z\|_{2\tb,[0,\tau]}\omega(0,\tau)^{\kappa_3}
+\left(\|R^Z\|_{2\tb,[0,\tau]}\omega(0,\tau)^{\kappa_3}
\right)^{\kappa_2}\nonumber\\
&\quad+
\|Z'\|_{\tb,[0,\tau]}\omega(0,\tau)^{\kappa_3}
+\left(\|Z'\|_{\tb,[0,\tau]}\omega(0,\tau)^{\kappa_3}\right)^{\kappa_2}
\Biggr\}.\label{estimate of RZtau}
\end{align}
Let
$
 z_{\tb,\tau}=\|Z'\|_{\tb,[0,\tau]}+\|R^Z\|_{2\tb,[0,\tau]}.
$
Then using (\ref{estimate of Zdashtau}) and (\ref{estimate of RZtau}),
we see that there exist (possibly different) 
$\kappa_1\ge 1, \kappa_2>1, \kappa_3>0, K>0$
which can be taken independent of $\tb$
such that
\begin{align}
 z_{\tb,\tau}&\le K\left\{\wopnorm^{\kappa_1}
+\omega(0,\tau)^{\kappa_3}\left(z_{\tb,\tau}
+z_{\tb,\tau}^{\kappa_2}\right)\right\},\quad
\text{for all $\tau$ with $\omega(0,\tau)\le 1$}.\label{inequality of z}
\end{align}
Since $\tb<\beta$, the function $\tau\mapsto z_{\tb,\tau}$ $(0\le \tau\le 1)$
is an increasing continuous function and 
$\lim_{\tau\to +0}z_{\tb, \tau}=0$.
If $\wopnorm=0$, then by the definition,
$Z_t=\xi$ for all $0\le t\le T$
and $\|Z'\|_{\beta}=\|R^Z\|_{2\beta}=0$ hold.
The desired estimate holds.
So we assume $\wopnorm\ne 0$.
We now define
\begin{align*}
 \tau_{1}=\sup\left\{\tau~\Big |~\tau\le T,\,\, \omega(0,\tau)\le 1,\,\,
z_{\tb,\tau}\le 2K\wopnorm^{\kappa_1}\right\}.
\end{align*}
There are two cases $\tau_{1}=T$ and
$\tau_{1}<T$.
Suppose $\tau_{1}=T$.
Then $z_{\tb, [0,T]}\le 2K\wopnorm^{\kappa_1}$ holds.
If this is not the case, 
$z_{\tb, \tau_{1}}=2K\wopnorm^{\kappa_1}$ holds.
Hence by the inequality (\ref{inequality of z}), we get
\begin{align}
 \omega(0,\tau_{1})&\ge
\left(\frac{K\wopnorm^{\kappa_1}}{2K\wopnorm^{\kappa_1}
+\left(2K\wopnorm^{\kappa_1}\right)^{\kappa_2}}\right)^{1/\kappa_3}.
\label{omegatau1}
\end{align}
After establishing this estimate, we proceed in a similar way 
to the argument in the proof of (1) replacing
$T'$ by $\tau_1$.
In this way,
we obtain an increasing time sequence 
$
 0=\tau_0<\tau_1<\cdots<\tau_{N-1}<\tau_N=T
$
 $(N\ge 2)$ and the estimate (\ref{omegatau1}) hold
for $\omega(\tau_{i-1},\tau_i)$ $(1\le i\le N-1)$.
Also we have
\begin{align}
 \|Z'\|_{\tb,[\tau_{i-1},\tau_i]}+\|R^Z\|_{2\tb,[\tau_{i-1},\tau_i]}&=
2K\wopnorm^{\kappa_1},\quad 1\le i\le N-1,\label{estimate of ztaui}\\
 \|Z'\|_{\tb,[\tau_{N-1},T]}+\|R^Z\|_{2\tb,[\tau_{N-1},T]}&\le
2K\wopnorm^{\kappa_1}.\label{estimate of zT}
\end{align}
By using $\sum_{i=1}^{N-1}\omega(\tau_{i-1},\tau_i)\le \omega(0,T)$,
we get the estimate of $N$ as follows.
\begin{align}
 N\le 
\left(2+2^{\kappa_2}K^{\kappa_2-1}\wopnorm^{\kappa_1\kappa_2-1}\right)
^{1/\kappa_3}\omega(0,T)+1.\label{estimate of N}
\end{align}
Using (\ref{estimate of ztaui}), (\ref{estimate of zT}),
(\ref{estimate of N}) and simple estimates
\begin{align*}
\|Z'\|_{\tb,[0,T]}&\le \sum_{i=1}^N\|Z'\|_{\tb,[\tau_{i-1},\tau_i]},\\
\|R^Z\|_{2\tb,[0,T]}&\le
\sum_{i=1}^N\|R^Z\|_{2\tb,[\tau_{i-1,\tau_i}]}+
\sum_{i=1}^N\sum_{j=1}^{i-1}\|Z'\|_{\tb,[\tau_{j-1},\tau_j]}
\|X\|_{\beta,[0,T]},
\end{align*}
we obtain 
\begin{align}
 \|Z'\|_{\tb,[0,T]}+\|R^Z\|_{2\tb,[0,T]}&\le
K\left\{\left(1+\wopnorm^{\kappa_1}\right)\omega(0,T)+1\right\}^2
\wopnorm^{\kappa_2}.
\end{align}
Since $\tb<\beta$ and $\|Z'\|_{\beta, [0,T]}+\|R^Z\|_{2\beta, [0,T]}<\infty$,
taking the limit $\tb\uparrow \beta$,
this estimate hold for the norms 
$\|\cdot\|_{\beta}$ and $\|\cdot\|_{2\beta}$ as well.
The estimates of $Z$ and $A(Z)$ follow from this estimate and the estimates
similar to (\ref{estimate of Ztau}) and (\ref{estimate of AZtau}).
This completes the proof.
\end{proof}

   \section{A continuity property of the solution mapping}
   \label{a continuity property}
   
 In this section, we consider the case where
$\omega(s,t)=|t-s|$.
That is, we consider
usual H\"older rough paths.
Also let us denote the set of $\beta$-H\"older geometric rough paths
($1/3<\beta\le 1/2$)
by $\mathscr{C}^{\beta}_g(\RR^d)$ which is the closure of
the set of smooth rough paths in the topology
of $\mathscr{C}^{\beta}(\RR^d)$.
In this paper, smooth rough path means the rough path ${\bf h}$
defined by a Lipschitz path $h\in \C^1$
and its iterated integral $\bar{h}^2_{s,t}=
\int_s^t(h_u-h_s)\otimes dh_u$.
We identify ${\bf h}$ and the Lipschitz path $h$.
Also we denote the set of smooth rough paths by
$\CLip$.

Let $Z({\bf h})$ be a solution to $(\ref{path-rde1})$
for $\mathbf{X}={\bf h}$.
Then $Z({\bf h})$ is a solution to the usual integral equation
 \begin{align}
  Z_t=\xi+\int_0^t\sigma(Z_s, A(Z)_s)dh_s. \label{path-ode6}
 \end{align}
As already explained, we cannot expect the uniqueness of the solution 
of the RDEs in our setting driven by general rough path $\mathbf{X}$.
However, the uniqueness hold in many cases 
when the driving rough path is a smooth rough path and
$\sigma$ is sufficiently smooth.
If the solution to the ODE 
$(\ref{path-ode6})$ is unique,
 then $Z({\bf h})$ is uniquely defined and
 $(Z({\bf h}),R^{Z({\bf h})},A(Z({\bf h})))$
 satisfies the same estimate as in
 Theorem~\ref{main theorem}.
 We use the notation $Z(h)_t$ instead of $Z({\bf h})_t$
 in this case.

We denote the set of solutions $(Z,Z')$ of our RDE
(\ref{path-rde1}) by
$Sol(\mathbf{X})$.
We prove a certain continuity property of
multivalued mapping $\mathbf{X}\mapsto Z(\mathbf{X})\in Sol(\mathbf{X})$
at the rough path $\mathbf{X}$ for which the solution is unique.
Thus, this multivalued map is continuous in such a sense 
at any smooth rough path
if the uniqueness holds on the set of smooth rough paths.
 
We write
$\C^{\theta-}=\cap_{0<\ep<\theta}\C^{\theta-\ep},
\C^{1+\hyp var,\theta-}=\cap_{q>1,0<\ep<\theta}\C^{q\hyp var,\theta-\ep}$.
 Clearly, these spaces are 
 Fr\'echet spaces with the naturally defined semi-norms.
Also note that $Z(\mathbf{X})\in \C^{\beta-}([0,T],\RR^n)$.

 \begin{lem}\label{limit point}
  We consider the equation $(\ref{path-rde1})$
  and assume the same assumption on $A$ and $\sigma$
  in Theorem~$\ref{main theorem}$.
~Let $\mathbf{X}\in \mathscr{C}^{\beta}(\RR^d)$.
Let $\{\mathbf{X}_N\}\subset\mathscr{C}^{\beta}(\RR^d) $ 
be a sequence such that
$\lim_{N\to\infty}\opnorm{\mathbf{X}_N-\mathbf{X}}_{\beta}=0$.
Let us choose solutions $Z(\mathbf{X}_N)\in Sol(\mathbf{X}_N)$
$(N=1,2,\ldots)$.
Then there exists a subsequence
$N_k\uparrow \infty$ such that the limit
$Z=\lim_{k\to\infty}Z(\mathbf{X}_{N_k})$ 
exits
in 
$\C^{\beta-}([0,T], \RR^n)$.
Further for such $Z$, $(Z,\sigma(Z,A(Z)))\in Sol(\mathbf{X})$ and
$
 \lim_{k\to\infty}\left\|R^{Z(\mathbf{X}_{N_k})}-R^{Z}\right\|_{2\beta-}
=0
$
hold.
\end{lem}

\begin{proof}
By the estimate in Theorem~\ref{main theorem} (2), we can choose
$\{N_k\}$ such that $Z(\mathbf{X}_{N_k}), A(Z(\mathbf{X}_{N_k}))$ 
converges in
$\C^{\beta-}$ and $\C^{1+\hyp var, \beta-}$ respectively.
This implies
$\lim_{k\to\infty}\int_s^tA(Z(\mathbf{X}_{N_k}))_{s,r}dX_{N_k}(r)
=\int_s^tA(Z(\mathbf{X}))_{s,r}dX_r$ which shows
the limit $Z$ satisfies the inequality (\ref{davie}).\\
This proves $(Z,\sigma(Z,A(Z)))\in Sol(\mathbf{X})$.
We have
\begin{align*}
 R^{Z(\mathbf{X}_{N_k})}_{s,t}=
Z_{s,t}(\mathbf{X}_{N_k})-\sigma\left(Z_s(\mathbf{X}_{N_k})_s,
A(Z(\mathbf{X}_{N_k}))_s\right)(X_{N_k})_{s,t}.
\end{align*}
Hence $\lim_{k\to\infty}R^{Z(\mathbf{X}_{N_k})}_{s,t}=R^{Z(\mathbf{X})}_{s,t}$
for all $(s,t)$.
Combining the uniform estimates of $(\omega,2\beta)$-H\"older estimates
of them, this completes the proof.
\end{proof}

The following proposition follows from the above lemma

\begin{pro}\label{continuity of Z}
  We consider the equation $(\ref{path-rde1})$
  and assume the same assumption on $A$ and $\sigma$
  in Theorem~$\ref{main theorem}$.
Assume the solution of $(\ref{path-rde1})$ is unique for
the rough path $\mathbf{X}_0\in \mathscr{C}^{\beta}(\RR^d)$.
Then the multivalued mapping $\mathbf{X}
	     (\in {\mathscr C}^{\beta}(\RR^d))\to 
Sol(\mathbf{X})$
	     is continuous at $\mathbf{X}_0$ in the following sense.
For any $\ep>0$, there exists $\delta>0$ such that
for any $\mathbf{X}$ satisfying 
$\opnorm{\mathbf{X}-\mathbf{X}_0}_{\beta}\le \delta$ 
and any $Z(\mathbf{X})\in Sol(\mathbf{X})$,
it holds that
\begin{align*}
\|Z(\mathbf{X})-Z(\mathbf{X}_0)\|_{\beta-}+
\|R^{Z(\mathbf{X})}-R^{Z(\mathbf{X}_0)}\|_{2\beta-}
\le \ep.
\end{align*}
\end{pro}

\begin{rem}\label{Solinfty}
Let $\mathbf{X}\in \mathscr{C}^{\beta}_g(\RR^d)$.
It holds that for any sequence $\{\mathbf{h}_N\}\subset \CLip$
satisfying $\lim_{N\to}\opnorm{\mathbf{h}_N-\mathbf{X}}=0$,
any accumulation points of
$\{Z(h_{N})\}$ belong to
$Sol(\mathbf{X})$.
The set $Sol_{\infty}(\mathbf{X})$ which
consists of such all accumulation points is a subset of
$Sol(\mathbf{X})$ and may be a natural class of solutions.
However $Sol_{\infty}(\mathbf{X})=Sol(\mathbf{X})$ may hold.
\end{rem}

By a similar argument to the proof of Theorem 4.9 in 
\cite{aida-rrde}, we can prove the existence of
universally measurable selection mapping of
solutions as follows.

 \begin{pro}\label{selection theorem}
    We consider the equation $(\ref{Z equation})$ and $(\ref{Psi equation})$
  and assume the same assumption on $A$ and $\sigma$
  in Theorem~$\ref{main theorem}$.
Then there exists a universally measurable mapping 
\begin{align*}
 {\cal I} : \mathscr{C}^{\beta}_g(\RR^d)\ni \mathbf{X}
  &\mapsto
\left(
\Bigl(Z(\mathbf{X}),\sigma(Y(\mathbf{X}))\Bigr),
\Psi(\mathbf{X})
 \right)\in \C^{\beta-}\times
 \C^{\beta-}\times
 \C^{1\hyp var+,\beta-}
 \end{align*}
which satisfies the following.
\begin{enumerate}
 \item[$(1)$] $\left(Z(\mathbf{X}),\sigma(Y(\mathbf{X}))\right)\in
	      {\mathscr D}^{2\beta}_X(\RR^n)$ and
	      $\Bigl(
\left(Z(\mathbf{X}),\sigma(Y(\mathbf{X}))\right),
\Psi(\mathbf{X})
\Bigr)$ is a solution in Theorem~$\ref{main theorem}$ and
satisfies the estimate in $(\ref{estimate of Z Z' Phi})$.
\item[$(2)$] There exists a sequence of Lipschitz paths
$h_N$ such that
$\opnorm{\mathbf{X}-{\bf h}_N}_{\beta}\to 0$
and
${\cal I}({\bf h}_N)$
converges to ${\cal I}(\mathbf{X})$ in
$\C^{\beta-}(\RR^n)\times
 \C^{\beta-}({\cal L}(\RR^d,\RR^n))\times
	     \C^{1\hyp var+,\beta-}(\RR^d)$.
	     \end{enumerate}
\end{pro}

 \begin{proof}
  Below, we omit writing $\xi$.
 We consider the product space,
\begin{align}
 E={\mathscr C}^{\beta}_g(\RR^d)\times
 \C^{\beta-}(\RR^n)\times
 \C^{\beta-}({\cal L}(\RR^d,\RR^n))\times
	     \C^{1\hyp var+,\beta-}(\RR^d)
\end{align}
 and its subset
 \begin{align}
  E_0=\left\{
  \Bigl({\bf h}, Z({\bf h}), \sigma(Y({\bf h})), \Psi({\bf h})\Bigr)
  \in E~|~\mbox{{\bf h} is a smooth rough path}\right\}
 \end{align}
  Let $\bar{E}_0$ be the closure of $E_0$ in $E$.
  Then $\bar{E}_0$ is a separable closed subset of
  $E$.
  The separability follows from the continuity of the mapping
  $h\mapsto \left((Z(h), \sigma(Y(h))), \Psi(h)\right)$.
  Note that $Sol_{\infty}(\mathbf{X})$
  coincides with the projection of the subset of
  $\bar{E}_0$ whose first component is $\mathbf{X}$.
  Hence by the measurable selection theorem (See 13.2.7.\,Theorem in
  \cite{dudley}), there exists a
  universally measurable mapping
  ${\cal I} : {\mathscr C}^{\beta}_g(\RR^d)\to E
  $ such that ${\cal I}(\mathbf{X})\in \left\{\mathbf{X}\right\}\times
  Sol_{\infty}(\mathbf{X})$.
  This mapping satisfies the required properties in
  (1) and (2).
 \end{proof}

\begin{rem}
 It is not clear that we could obtain the adapted measurable solution 
mapping $\mathcal{I}$.
\end{rem}

\section{Examples}\label{examples}

\subsection{Reflected rough differential equations}\label{rrde}

Let $D$ be a connected domain in $\RR^n$.
As in \cite{saisho, lions-sznitman},
we consider the following conditions (A), (B)
on the boundary.
See also \cite{tanaka}.

\begin{defin}
We write $B(z,r)=\{y\in \RR^n\,|\, |y-z|<r\}$,
where $z\in \RR^n$, $r>0$.
The set ${\cal N}_x$
of inward unit normal vectors at the boundary
point $x\in \partial D$
is defined by
\begin{align}
 {\cal N}_x&=\cup_{r>0}{\cal N}_{x,r},\\
{\cal N}_{x,r}&=\left\{{\bm n}\in \RR^n~|~|{\bm n}|=1,
 B(x-r{\bm n},r)\cap D=\emptyset\right\}.
\end{align}

\begin{enumerate}
 \item[{\rm (A)}] 
 There exists a constant
$r_0>0$ such that
\begin{align}
{\cal N}_x={\cal N}_{x,r_0}\ne \emptyset \quad \mbox{for any}~x\in
 \partial D.\nonumber
\end{align}

\item[{\rm (B)}]
There exist constants $\delta>0$ and $0<\delta'\le 1$
satisfying:

for any $x\in\partial D$ there exists a unit vector $l_x$ such that
\begin{align}
 (l_x,{\bm n})\ge \delta'
\qquad \mbox{for any}~{\bm n}\in 
\cup_{y\in B(x,\delta)\cap \partial D}{\cal N}_y.\nonumber
\end{align}
\end{enumerate}
\end{defin}

Let us recall the Skorohod equation.
The Skorohod equation associated with a continuous path
$x\in C([0,\infty), \RR^n)$ with $x_0\in \bar{D}$
is given by
\begin{align}
 y_t&=x_t+\phi_t,\quad  y_t\in \bar{D}\qquad t\ge 0,\label{SP1}\\
\phi_t&=\int_0^t1_{\partial D}(y_s){\bm n}(s)d\|\phi\|_{1\hyp var,[0,s]}\quad
t\ge 0,\qquad
{\bm n}(s)\in {\cal N}_{y_s}~\mbox{if $y_s\in \partial D$}
\label{SP2}
\end{align}
Under the assumptions (A) and (B) on $D$, the Skorohod equation is
uniquely solved.
This is due to Saisho~\cite{saisho}.
We write
$\Gamma(x)_t=y_t$ and $L(x)_t=\phi_t$.
By the uniqueness, we have the following flow property.

\begin{lem}\label{flow property}
Assume {\rm (A)} and {\rm (B)}.
For any continuous path $x$ on $\RR^n$ with $x_0\in {\bar D}$,
we have for all $\tau, s\ge 0$
\begin{align}
\Gamma(x)_{\tau+s}&=\Gamma\left(y_s+\theta_sx\right)_{\tau},\\
L(x)_{\tau+s}&=L(x)_s+L\left(y_s+\theta_sx\right)_{\tau},
\end{align}
where $(\theta_sx)_{\tau}=x_{\tau+s}-x_s$.
\end{lem}

We obtain the following estimate of
$L(x)$.

\begin{lem}\label{estimate of phi}
Assume conditions {\rm (A)} and {\rm (B)} hold.
Let $x_t$ be a continuous path of finite $q$-variation
$(q\ge 1)$.
Then we have the following estimate.
\begin{align}
\|L(x)\|_{1\hyp var,[s,t]}
 &\le
 C\left(\|x\|_{q\hyp var, [s,t]}^q+1\right)
 \|x\|_{\infty\hyp var,[s,t]},
\label{estimate of phi 1}
\end{align}
where 
 $C$ is a positive constant which depends
 on the constants $\delta, \delta', r_0$ in
conditions {\rm (A)} and {\rm (B)}.
\end{lem}

\begin{proof}
 We proved the following estimate in \cite{aida-rrde, aida-sasaki}
 following the argument in \cite{saisho}.
 \begin{align}
\|L(x)\|_{1\hyp var,[s,t]}
&\le \delta'^{-1}
\left(\left\{\delta^{-1}G(\|x\|_{\infty\hyp var,[s,t]})
+1\right\}^{q}\|x\|_{q\hyp var, [s,t]}^q+1
\right)\nn\\
&\qquad\qquad \times
\left(G(\|x\|_{\infty\hyp var,[s,t]})+2\right)\|x\|_{\infty\hyp var,[s,t]},
\end{align}
where 
\begin{align}
G(u)&=4\left\{1+\delta'^{-1}
\exp\left\{\delta'^{-1}\left(2\delta+u\right)/(2r_0)\right\}
\right\}\exp\left\{\delta'^{-1}\left(2\delta+u\right)/(2r_0)\right\},
\quad u\in\RR.
\end{align}
 By combining this and Lemma~\ref{lemma for A},
 we complete the proof.
\end{proof}

\begin{lem}\label{Holder continuity}
Assume {\rm (A)} and {\rm (B)}.
Consider two Skorohod equations
$y=x+\phi$, $y'=x'+\phi'$.
Then
\begin{align}
|y_t-y'_t|^2&\le
\left\{|x_t-x'_t|^2+
4\left(\|\phi\|_{1\hyp var,[0,t]}+\|\phi'\|_{1\hyp var,[0,t]}\right)
\max_{0\le s\le t}|x(s)-x'(s)|\right\}\nonumber\\
& \quad
\exp\left\{\left(\|\phi\|_{1\hyp var,[0,t]}
+\|\phi'\|_{1\hyp var,[0,t]}\right)/r_0\right\}.
\label{Holder continuous map}
\end{align}
\end{lem}

The estimate (\ref{Holder continuous map}) can be found in
Remark~4.1 (i) in \cite{saisho}.
Lemma~\ref{estimate of phi} shows that if
$x$ is a $(\omega,\theta)$-H\"older continuous path, 
$L(x)\in \C^{1\hyp var,\theta}$ holds true.
Actually, $\|L(x)\|_{1\hyp var, [s,t]}$
can be estimated by the modulus of continuity of $x$
and $\|x\|_{\infty\hyp var,[s,t]}$.
For example, see \cite{saisho} and the proof
of Lemma~2.3 in \cite{aida-sasaki}.
Hence, we see that
$L$ is a $1/2$-H\"older continuous map on
$C([0,\infty), \RR^n)$.
Note that $\Gamma$ is Lipschitz continuous if $D$ is a convex
polyhedron(\cite{dupuis-ishii}).

Let $\mathbf{X}\in \mathscr{C}^{\beta}(\RR^d)$.
We assume $D$ satisfies the condition (A) and (B).
We now consider reflected RDE:
 \begin{align}
  Y_t&=\xi+\int_0^t\sigma(Y_s)d\mathbf{X}_s+\Phi_t,
\quad \Phi_t=L\left(\xi+\int_0^{\cdot}\sigma(Y_s)
d\mathbf{X}_s\right)_t,\quad \xi\in \bar{D}.\label{rrde1}
 \end{align}
We need to make clear the definition of the solution 
$(Y_t)$ of (\ref{rrde1}).

\begin{defin}
 We call $Y_t$ is a solution of (\ref{rrde1}) if and only if
the following holds:
\begin{itemize}
 \item[(i)] There exist a $Z\in \mathscr{D}^{2\beta}([0,T],\RR^n)$
and a continuous bounded variation path $\Phi_t$ such that
$Y_t=Z_t+\Phi_t$ $(0\le t\le T)$.
\item[(ii)] $\Phi_t=L(Z)_t$ $(0\le t\le T)$.
\item[(iii)] $Z$ satisfies 
\begin{align}
 Z_t&=\xi+\int_0^t\sigma(Z_s+L(Z)_s)d\mathbf{X}_s,\quad
Z_t'=\sigma(Z_t+L(Z)_t)
\qquad 0\le t\le T.\label{rrde2}
\end{align}
\end{itemize}
\end{defin}

Note that if $Y$ is a solution in the above sense, $Z$ is uniquely 
determined by
$Y$ and $\mathbf{X}$ since
$Z_t=\xi+\int_0^t\sigma(Y_s)d\mathbf{X}_s$
and $Z'_t=\sigma(Y_t)$ hold.
See also Remark~\ref{second remark} (1).

By applying Theorem~\ref{main theorem}, we obtain the following
result.

\begin{thm}\label{reflected case}
Let $\mathbf{X}\in \mathscr{C}^{\beta}(\RR^d)$.
Assume $D$ satisfies conditions {\rm (A)} and {\rm (B)}.\\
Let
$\sigma\in \Lip^{\gamma-1}(\RR^n, {\cal L}(\RR^d,\RR^n))$
and $\xi\in \bar{D}$.
Then 
there exist $(Z,Z')\in 
{\mathscr D}^{2\beta}_X(\RR^n)$
and $\Phi\in \C^{1\hyp var,\beta}(\RR^n)$ with
$\Phi_0=0$ such that
$Y_t=Z_t+\Phi_t$ is a solution of $(\ref{rrde1})$.
Moreover the following estimate holds,
\begin{align}
 \|Z\|_{\beta}+\|R^Z\|_{2\beta}+
\|\Phi\|_{1\hyp var,\beta}&\le 
K\left\{1+\left(1+\wopnorm\right)^{\kappa_1}\omega(0,T)\right\}^{\kappa_2}
\wopnorm^{\kappa_3},
\label{estimate of Z Z' Phi}
\end{align}
where 
$K, \kappa_i$ are constants
which depend only on $\sigma,\beta,\gamma, \delta, \delta', r_0$.
\end{thm}

\begin{proof}
 By applying Theorem~\ref{main theorem}, we have at least one solution
 $Z$ and the estimate of (\ref{rrde2}).
 Let $Y_t=Z_t+L(Z)_t$ and $\Phi_t=L(Z)_t$.
 Then this pair is a solution to the original equation. 
\end{proof}

\begin{rem}\label{second remark}
$(1)$ Let $(Y_t,\Phi_t)$ be a solution of (\ref{rrde1}).
Then there exists $\theta>1$ such that
\begin{align}
& \left|
Y_{s,t}-\Phi_{s,t}-
\left(
\sigma(Y_s)X_{s,t}+(D\sigma)(Y_s)[\sigma(Y_s)]\mathbb{X}_{s,t}+
(D\sigma)(Y_s)\left(\int_s^t\Phi_{s,u}\otimes dX_u\right)
\right)\right|\nonumber\\
&\qquad\quad\le
C\omega(s,t)^{\theta},
\quad 0\le s<t\le T.\label{davie type formulation}
\end{align}
Conversely, suppose
\begin{itemize}
 \item[(i)] $(Y_t,\Phi_t)$ is a pair of continuous paths satisfying
(\ref{davie type formulation}) and 
$(\Phi_t)$ is a bounded variation path satisfying
$\|\Phi\|_{1\hyp var,[s,t]}\le C\omega(s,t)^{\beta}$ $(0\le s\le t\le T)$.
\item[(ii)] $Y_t\in \bar{D}$ $(0\le t\le T)$.
\item[(iii)] $(Y_t,\Phi_t)$ satisfies
\[
 \Phi_t=\int_0^t1_{\partial D}(Y_s){\bm n}(s)d\|\Phi\|_{1\hyp var,[0,s]}
\quad\quad 0\le t\le T,\quad
({\bm n}(s)\in \mathcal{N}_{Y_s}\quad 
\text{if $Y_s\in \partial D$}).
\]
\end{itemize}
Let $\Xi_{s,t}=\sigma(Y_s)X_{s,t}+(D\sigma)(Y_s)[\sigma(Y_s)]\mathbb{X}_{s,t}+
(D\sigma)(Y_s)\left(\int_s^t\Phi_{s,u}\otimes dX_u\right)$.
Then $|(\delta\Xi)_{s,u,t}|\le C\omega(s,t)^\theta$ $(0\le s\le u\le t\le T)$
holds and $Z_{0,t}\in \C^{\beta}([0,T],\RR^n ; x_0=0)$ exists such that
$|(Z_{0,t}-Z_{0,s})-\Xi_{s,t}|\le C\omega(s,t)^\theta$.
Further, by the assumption on $\Phi$,
$(Z_{0,t})\in \mathscr{D}^{2\beta}_X(\RR^n)$ with $Z_{0,t}'=\sigma(Y_t)$
and $Y_t=Y_0+Z_{0,t}+\Phi_t$ holds.
Clearly, $Z_{0,t}=\int_0^t\sigma(Y_s)d\mathbf{X}_s$.
By the definition of $L$, we have
$L(Y_0+Z_{0,\cdot})_t=\Phi_t$.
Hence, $(Y_t,\Phi_t)$ is a solution of (\ref{rrde1}).

\noindent
$(2)$
In \cite{aida-rrde}, we consider the following condition
(H1) on $D$ :
\begin{enumerate}
 \item[(i)] The condition (A) holds,
\item[(ii)] There exists a positive constant $C$ such that for any
$x$, it holds that
\begin{align*}
\|L(x)\|_{1\hyp var, [s,t]}&\le
C\|x\|_{\infty\hyp var,[s,t]}.
\end{align*}
\end{enumerate}
 This condition holds if
 $D$ is convex and there exists a unit vector $l\in \RR^n$
 such that
\[
   \inf\left\{(l,{\bm n}(x))~|~{\bm n}(x)\in {\cal N}_x,\,
  x\in \partial D\right\}>0.
\]
Under (H1) and $\sigma\in C^3_b$, we proved the existence of solutions of
reflected RDEs driven by $1/\beta$ rough paths and
 gave estimates for the solutions in Theorem 4.5 in \cite{aida-rrde}.
We used Euler approximation of the solution modifying
Davie's proof in \cite{davie}.
In the proof, we need to solve the following implicit Skorohod equation
in each step,
\begin{align}
& y_t=\xi+\eta_t+M\left(\int_0^t\Phi_r\otimes dx_r\right)+\Phi_t,
\qquad \xi\in \bar{D},\quad 0\le t\le T',
\label{implicit skorohod equation}\\
& L\left(\xi+\eta_{\cdot}+M\left(\int_0^{\cdot}\Phi_r\otimes
 dx_r\right)\right)_t
=\Phi_t, \quad 0\le t\le T',\qquad \Phi_0=0,\label{implicit skorohod equation2}
\end{align}
where $0<T'<T$, $y_t\in \bar{D}$~$(0\le t\le T')$,
$M\in \mathcal{L}(\RR^n\otimes \RR^d,\RR^n)$ 
and $\Phi_t$ is a continuous bounded variation path.
Also $\eta_t$, $x_t$ are finite $1/\beta$-variation paths
which are defined by $X$ and $\XX$.
If we replace $\int_0^t\Phi_r\otimes dx_r$ in 
(\ref{implicit skorohod equation}) and
(\ref{implicit skorohod equation2}) by
$\int_0^tf(\Phi_r)\otimes dx_r$, where
$f$ is a bounded Lipschitz map between $\RR^n$,
then we can solve the equation under general condition
(A) and (B).
To avoid the explosion problem, that is, to handle the linear growth term of 
$\Phi_t$, we put stronger assumption
(H1)(ii) on D in \cite{aida-rrde}.
Also we used Lyon's continuity theorem of rough integrals in the proof
and so we need to assume $\sigma\in C^3_b$.
In this paper, we adopt different approach to the problem and obtain 
an extension of the previous result 
in the sense that
the assumption on $\sigma$ and $D$ can be relaxed.
\end{rem}

In Section~\ref{a continuity property},
we prove a continuity property of solution mappings at Lipschitz paths
under the uniqueness of the solutions.
For reflected RDEs, we can give more explicit estimate
 of the continuity of the solution mapping $Y$ at the Lipschitz paths.
 As before we consider
a domain $D\subset \RR^n$ which satisfies the conditions
(A) and (B).
Let $h$ be a Lipschitz path on $\RR^d$ starting at $0$.
If $\sigma$ is Lipschitz continuous,
there exists a unique solution $(Y(h,\xi)_t,\Phi(h,\xi)_t)$ to
the reflected ODE in usual sense 
(see Proposition 4.1 in \cite{aida-sasaki} for example),
\begin{align}
 Y_t&=\xi+\int_0^t\sigma(Y_s)dh_s+
 \Phi_t,\quad \xi\in \bar{D},\quad 0\le
 t\le T.
\label{reflected ode}
\end{align}
We may omit denoting $h,\xi$.
Moreover, $Z(h)_t=\xi+\int_0^t\sigma(Y_s(h))dh_s$,
$Z_t(h)'=\sigma(Y_t(h))$ and
$\Phi(h)_t$ are a unique pair of solution to the equation in 
Theorem~\ref{reflected case}
for the smooth rough path
${\bf h}_{s,t}=(h_{s,t},\bar{h}^2_{s,t})$ defined by
$h$.
Hence the solution $(Z(h),R^{Z(h)},\Phi(h))$
satisfies the estimate (\ref{estimate of Z Z' Phi})
with the same constant $C_1, C_2$.

From now on, we will give an explicit estimate
for $Y_t(\xi,\mathbf{X})-Y_t(\eta,h)$.
Let $\mathbf{X}$ be a general (not necessarily geometric)
$\beta$-H\"older rough path.
Let
$\mathbf{X}^{-h}_{s,t}$ be the translated rough path of 
$\mathbf{X}$ by $h$.
That is,
the 1st level path and the second level path are given by,
\begin{align}
X^{-h}_{s,t}&=X_{s,t}-{h}_{s,t}\\
\XX^{-h}_{s,t}&=\XX_{s,t}-\bar{h}^2_{s,t}
-\int_s^tX^{-h}_{s,u}\otimes
     dh_u-\int_s^t{h}_{s,u}\otimes dX^{-h}_{s,u}.
\end{align}
Hence
\begin{align}
\|X^{-h}\|_{\beta}&\le \|X-h\|_{\beta},\label{X minus h 1}\\
 \|{\mathbb X}^{-h}\|_{2\beta}&\le
 \|{\mathbb X}-\bar{h}^2\|_{2\beta}+
 \left(1+\frac{2}{1+\beta}\right)T^{1-\beta}
\|X-h\|_{\beta}\|h\|_1.
\label{X minus h 2}
\end{align}
These imply that
if $\opnorm{{\bf h}-\mathbf{X}}_{\beta}\le 1$, then
\begin{align}
 \opnorm{\mathbf{X}^{-h}}_{\beta}\le
	     \left(1+\sqrt{\left(1+\frac{2}{1+\beta}\right)
	     T^{1-\beta}\|h\|_1}\right)\opnorm{{\bf h}-\mathbf{X}}_{\beta}.
\end{align}
By the definition of controlled paths, we immediately obtain
the following.

\begin{lem}\label{R X minus h}
Let $\mathbf{X}\in \mathscr{C}^{\beta}_g(\RR^d)$.
Let $h$ be a Lipschitz path.
If $(Z,Z')\in \mathscr{D}_X^{2\beta}$, then
$(Z,Z')\in \mathscr{D}_{X-h}^{2\beta}$.
In fact,
\begin{align}
\left|Z_{s,t}-Z'_sX^{-h}_{s,t}\right|&\le
\left(\|R^Z\|_{2\beta}+(|Z'_0|+\|Z'\|_{\beta}s^{\beta})\|h\|_1
(t-s)^{1-2\beta}\right)(t-s)^{2\beta}.
\end{align}
\end{lem}

Let
$(Z,Z')\in {\mathscr D}^{2\alpha}_X(\RR^n)$
and $\Phi\in \C^{q\hyp var,\ta}(\RR^n)$ with $\Phi_0=0$ and
$q, \alpha, \ta$ satisfy the assumptions
in Lemma~$\ref{estimate of phi and w}$.
By the above lemma,
we can define the integral
$\int_s^t\sigma(Y_u)d\mathbf{X}^{-h}_u$
and the estimates in Lemma~\ref{estimate of rough integral}
hold for this integral.
Here $Y_u=Z_u+\Phi_u$.
Moreover,
$\Xi_{s,t}$ in (\ref{Xi}) which is defined by $\mathbf{X}^{-h}_{s,t}$ reads
\begin{align}
\Xi_{s,t}&=\sigma(Y_s)X^{-h}_{s,t}+
(D\sigma)(Y_s)Z'_s\XX^{-h}_{s,t}+
(D\sigma)(Y_s)\int_s^t\Phi_{s,u}\otimes dX^{-h}_u\\
&=\sigma(Y_s)X_{s,t}+
(D\sigma)(Y_s)Z'_s\XX_{s,t}+
(D\sigma)(Y_s)\int_s^t\Phi_{s,u}\otimes dX_u
-\sigma(Y_s)h_{s,t}+\tilde{\Xi}_{s,t},
\end{align}
where
\begin{align}
 \tilde{\Xi}_{s,t}&=
-(D\sigma)(Y_s)Z'_s\left(
\bar{h}^2_{s,t}+\int_s^tX^{-h}_{s,u}\otimes dh_u+
\int_s^th_{s,u}\otimes dX^{-h}_{s,u}
\right)+(D\sigma)(Y_s)\int_s^t\Phi_{s,u}\otimes dh_u.
\end{align}
Since $|\tilde{\Xi}_{s,t}|\le C(t-s)^{1+\ta}$,
the sum of these terms converges to $0$.
Thus we obtain
\begin{align}
\int_s^t\sigma(Y_u)d\mathbf{X}^{-h}_u=
\int_s^t\sigma(Y_u)d\mathbf{X}_u-\int_s^t\sigma(Y_u)dh_u.
\end{align}

We now consider the following condition on the boundary.

\begin{defin}[Condition {\rm (C)}]
There exists a ${\rm Lip}^{\gamma}$ function $f$ on $\RR^n$ 
and a positive constant $k$ such that 
for any $x\in \partial D$, $y\in \bar{D}$, ${\mathbf n}\in {\cal N}_x$
it holds that
\begin{align}
 \left(y-x,{\mathbf n}\right)+\frac{1}{k}\left((D
 f)(x),{\mathbf n}\right)|y-x|^2\ge 0.
\end{align}
\end{defin}

Usually, the function $f$ is assume to be $C^2_b$
in the condition (C).
See \cite{lions-sznitman, saisho}.
Here, we assume $f\in {\rm Lip}^{\gamma}$ to make use of
estimates in Lemma~\ref{estimate of rough integral}.

Under additional condition (C),
we can prove the following explicit modulus of continuity.

\begin{lem}\label{continuity under C}
Let $\mathbf{X}\in \mathscr{C}^{\beta}_g(\RR^d)$.
Assume that $D$ satisfies the conditions {\rm (A)}, {\rm (B)},
{\rm (C)} and
$\sigma\in \Lip^{\gamma-1}$.
Let $Y_t(\mathbf{X},\xi), Z_t(\mathbf{X},\xi), \Phi_t(\mathbf{X},\xi),
 Y_t(h,\zeta), \Phi_t(h,\zeta)$
 be a solution as in Lemma~$\ref{limit point}$.
Assume $\opnorm{{\bf h}-\mathbf{X}}_{\beta}\le 1$.
Then there exists a positive constant
$C$ which depends only on $\sigma, r_0, \delta, \delta', f, k$ such that
\begin{align}
 \sup_{0\le t\le T}|Y_t(\mathbf{X},\xi)-Y(h,\zeta)_t|\le
 Ce^{C\|h\|_1}(|\xi-\zeta|+\opnorm{{\bf h}-\mathbf{X}}_{\beta}).
 \label{Y and Yh}
\end{align}
\end{lem}

\begin{proof}
We write $Y_t=Y_t(\mathbf{X},\xi)$, $\Phi(\mathbf{X},\xi)_t=\Phi_t$
and
$\tilde{Y}_t=Y(h,\zeta)_t$, $\tPhi_t=\Phi(h,\zeta)_t$ 
for simplicity.
Let 
$
Z_t=e^{-\frac{2}{k}\left(f(Y_t)+f(\tY_t)\right)}
|Y_t-\tY_t|^2.
$
We have
\begin{align}
&
 Z_t-Z_0\nonumber\\
 &=
\int_0^t
2e^{-\frac{2}{k}\left(f(Y_s)+f(\tY_s)\right)}
\Bigl\{
\left(Y_s-\tY_s, \left(\sigma(Y_s)-\sigma(\tY_s)\right)h'_s\right)
ds+
\left(Y_s-\tY_s, \sigma(Y_s)dX^{-h}_s\right)\Bigr\}
\nonumber\\
&\quad -\frac{2}{k}\int_0^t
Z_s\left(\sigma(Y_s)^{\ast}Df(Y_s)+
\sigma(\tY_s)^{\ast}Df(\tY_s), h'_s\right)ds
-\frac{2}{k}\int_0^t
Z_s\left(Df(Y_s),\sigma(Y_s)dX^{-h}_s\right)
\nonumber\\
&\quad-\int_0^t2e^{-\frac{2}{k}\left(f(Y_s)+f(\tY_s)\right)}
\Bigl\{
 \left(\tY_s-Y_s, d\Phi_s-d\tPhi_s\right)\nonumber\\
 &
 \qquad\quad\qquad\quad\qquad
 +\frac{1}{k}\left(Df(Y_s), d\Phi_s\right)|Y_s-\tilde{Y}_s|^2+
 \frac{1}{k}\left(Df(\tY_s),d\tPhi_s\right)
 |Y_s-\tilde{Y}_s|^2
\Bigr\}.\label{estimate by C}
\end{align}
Condition (C) implies that the fourth integral on the right-hand side
of the equation (\ref{estimate by C}) is always negative.
By the estimates of the solution
$Y, \tY, \Phi, \tPhi$ in Theorem~\ref{reflected case} and the estimates in
Lemma~\ref{estimate of rough integral} and the Gronwall inequality,
we obtain the desired estimate.
\end{proof}

\subsection{Perturbed reflected SDEs: a short review}
\label{section of perturbed diffusion}
Let us recall basic results for the following equation driven by
a continuous path $x_t$ on $\RR$,
 \begin{align}
  Y_t&=x_t+a\sup_{0\le s\le t}Y_s
  +b\inf_{0\le s\le t}Y_s,
  \label{si-bm2}\\
  Y_t&=x_t+a\sup_{0\le s\le t}Y_s+\Phi_t,~~x_0\ge 0,~~
  Y_t\ge 0~~\mbox{for all $t$}.
  \label{ml-bm2}
 \end{align}
 When $x_t$ is a sample path of a standard Brownian motion,
 the solutions to (\ref{si-bm2}) and (\ref{ml-bm2}) are
 called (doubly) perturbed Brownian motion and perturbed reflected
 Brownian motion respectively.
 
 First we consider the equation (\ref{si-bm2}).
 Clearly, if either $a\ge 1$ or $b\ge 1$, then there are no
 solutions to this equation for certain $x$.
 So we consider the case where $a<1$ and $b<1$.
 Suppose $b=0$.
 Then we have explicitly,
 $Y_t=x_t+\frac{a}{1-a}\sup_{0\le s\le t}x_s$.
  By \cite{cpy}, when $|\frac{ab}{(1-a)(1-b)}|<1$,
  a fixed point argument works and the unique existence holds for
any continuous path $x_t$ with $x_0=0$.
  The unique existence extends to $|\frac{ab}{(1-a)(1-b)}|=1$ by
  \cite{davis}.
  Consider the case where $x_t$ is a sample path of $1$-dimensional
  Brownian motion $W_t$ with $W_0=0$.
  For any $0\le a<1$, $0\le b<1$,
  it is proved in \cite{perman-werner} that
        the pathwise uniqueness holds and the solution is
adapted to the Brownian filtration.
Finally, for any $a<1, b<1$, the same results is proved in
\cite{chaumont-doney}.

  	We consider the equation (\ref{ml-bm2}).
	By a fixed point argument, the unique existence is proved in
	\cite{legall-yor} the case (1) $a<1/2$ and (2) $a<1$ with $x_0>0$.
Next, the pathwise uniqueness is proved by \cite{chaumont-doney}
for $a<1$ when $x_t$ is the Brownian path $W_t$ with $W_0=0$.
	The unique existence for $a<1$ is extended by
	\cite{doney-zhang} for any continuous path $x_t$.

	We next explain results
	for the variable coefficient version driven by
 	a standard $1$-dimensional Brownian motion $W_t$,
	 \begin{align}
  Y_t&=\xi+\int_0^t\sigma(Y_s)dW_s+a\sup_{0\le s\le t}Y_s,
  \label{si-bm3}\\
	  Y_t&=\xi+\int_0^t\sigma(Y_s)dW_s
	  +a\sup_{0\le s\le t}Y_s+\Phi_t,~~\xi\ge 0,~~
  Y_t\ge 0~~\mbox{for all $t$},
  \label{ml-bm3}
 \end{align}
where $\sigma$ is a Lipschitz continuous function on $\RR$ and
 the integral is the It\^o integral.
 The unique existence of the solution to
 (\ref{si-bm3}) is proved for $a<1$ by
 \cite{doney-zhang}.
 The same authors prove the unique existence of the solution
 to (\ref{ml-bm3}) for two cases where
 (1) $a<1$ and $\xi>0$ and (2) $0\le a<1/2$ and $\xi=0$.
 Under the same assumption on $a$, the absolutely continuity of the
 law of $Y_t$ with respect to the Lebesgue measure was studied in
 \cite{yue-zhang}.

\subsection{Perturbed reflected rough differential equations}
 We consider the multidimensional versions of
 (\ref{si-bm3}) and (\ref{ml-bm3}) driven by rough paths.
 Our objectives are the following two equations.
 \begin{align}
 Y_t=\xi+\int_0^t\sigma(Y_s)d\mathbf{X}_s+C(Y)_t,\label{prde1}
\end{align}
\begin{align}
    Y_t=\xi+\int_0^t\sigma(Y_s)d\mathbf{X}_s
  +C(Y)_t+\Phi_te_n,\label{prrde1}
  \end{align}
   where $e_n={}^t(0,\ldots,0,1)$
   and $\sigma\in \Lip^{\gamma-1}(\RR^n,\mathcal{L}(\RR^d,\RR^n))$.
      We assume that $C$ is a mapping from
$C([0,T],\RR^n)$ to the subspace of continuous and bounded variation paths
on $\RR^n$ and
$\{C(x)_s\}_{0\le s\le t}$ is
measurable with respect to $\sigma(\{x_s\}_{0\le s\le t})$ for all
$0\le t\le T$.
The first equation (\ref{prde1})
    is a perturbed rough differential equations and
the second equation (\ref{prrde1}) is a perturbed reflected rough
 differential equation on
 $\bar{D}=\{(x_1,\ldots,x_n)~|~x_n\ge 0\}$.
 $\Phi_te_n$ is the reflected term and $Y_t$ and $\Phi_t$ should satisfy
\begin{align}
\text{$Y^n_t=(Y_t,e_n)\ge 0$ for all $t\ge 0$, where $(\cdot, e_n)$ is
an inner product,}\label{Yn}\\
\text{$(\Phi_t)$ is continuous and nondecreasing, $\Phi_0=0$ and
$\displaystyle{\Phi_t=\int_0^t1_{\{0\}}(Y^n_s)d\Phi_s}$.}\label{Y and Phi}
\end{align}
In both equations, $Y_0\ne \xi$ in general.
Consider the case $t=0$.
Then we have
$
 Y_0=\xi+C(Y)_0.
$
Since $C(Y)$ is adapted, $C(Y)_0$ is a function of $Y_0$ and we may write
$C(Y)_0=C_0(Y_0)$.
Hence $Y_0$ should satisfy $Y_0=\xi+C_0(Y_0)$ and we need to assume
$Y_0\in \bar{D}$.
If we consider the case where $Y_t\in \RR$ and
$C(Y)_t=a\max_{0\le s\le t}Y_s$ $(a<1)$,
$Y_0=\frac{1}{1-a}\xi$ holds.
In this case, $Y_0\ge 0$ and $\xi\ge 0$ are equivalent and 
so $Y_t$ starts from $[0,\infty)$ when $\xi\ge 0$.
Under the assumption that $Y_0=\xi+C_0(Y_0)\in \bar{D}$, 
by the explicit solution of the Skorohod problem, we have
 \begin{align}
  \Phi_t=\max_{0\le s\le t}\left\{
  -\left(\xi+\int_0^t\sigma(Y_s)
  d\mathbf{X}_s+C(Y)_s, e_n\right)\vee 0 \right\},
  \label{half space reflection}
 \end{align}
   where $a\vee b=\max(a,b)$.

We give the definition of the solution of 
(\ref{prde1}) and (\ref{prrde1}).

\begin{defin}\label{definition of the solution of prrde}
 (1) $Y_t$ is a solution of (\ref{prde1}) if the following hold. 
\begin{itemize}
 \item[(i)] There exists a $Z\in \mathscr{D}^{2\beta}_X(\RR^n)$ such that
$Y_t=Z_t+C(Y)_t$ and $Z'_t=\sigma(Y_t)$ $(0\le t\le T)$ hold.
\item[(ii)] $Z_t=\xi+\int_0^t\sigma(Z_s+C(Y)_s)d\mathbf{X}_s$ $(0\le t\le T)$
holds.
\end{itemize}

\noindent
(2) $(Y_t,\Phi_t)$ is a solution of (\ref{prrde1}) if the following holds:
\begin{itemize}
 \item[(i)] $(Y_t,\Phi_t)$ satisfies (\ref{Yn}) and (\ref{Y and Phi}).
\item[(ii)] There exists a $Z\in \mathscr{D}^{2\beta}_X(\RR^n)$ such that
$Y_t=Z_t+C(Y)_t+\Phi_te_n$ and $Z'_t=\sigma(Y_t)$ $(0\le t\le T)$ hold.
\item[(iii)] $Z_t=\xi+\int_0^t\sigma(Z_s+C(Y)_s+\Phi_se_n)d\mathbf{X}_s$
$(0\le t\le T)$ holds.
\end{itemize}
\end{defin}
   We solve these equations by transforming them to the equations
   in Theorem~\ref{main theorem}.
   To this end, we introduce the following conditions.

   \begin{defin}\label{definition of lip and bv}
    For a mapping $C : C([0,T], \RR^n)\to C([0,T], \RR^m)$,
    we consider the following conditions,
    where $\rho$ denotes a positive number.
    \begin{itemize}
 \item[${\rm (Lip)}_{\rho}$] 
       $\|C(x)-C(y)\|_{\infty,[0,t]}\le \rho\|x-y\|_{\infty,[0,t]}$
       for all $x,y\in C([0,T], \RR^n)$ and $0\le t\le T$.
 \item[${\rm (BV)}_{\rho}$] 
	      $\|C(x)\|_{1\hyp var, [s,t]}\le \rho
	      \|x\|_{\infty\hyp var, [s,t]}$
	      for all $0\le s\le t\le T$.
\end{itemize}
    We may write $C\in (\mathrm{Lip})_{\rho}$ simply
    when $C$ satisfies the condition $(\mathrm{Lip})_{\rho}$,
    \textit{etc}.
    Also we denote by $\|C\|_{\mathrm{Lip}}$ the smallest nonnegative
    number $\rho$ for which $(\mathrm{Lip})_{\rho}$ holds.
   \end{defin}

Clearly the conditions ${\rm (Lip)}_{\rho}$ and ${\rm (BV)}_{\rho}$ are
stronger than
the conditions in Assumption~\ref{assumption on A}.
Also the conditions ${\rm (Lip)}_{\rho}$ and ${\rm (BV)}_{\rho}$ imply
the conditions (A1), (A2) and (A3) in \cite{aida-kikuchi-kusuoka}.

As we noted, $C$ which is defined in Example~\ref{example of A} (2)
satisfies the above conditions.

\begin{pro}\label{explicit rho}
Let $\rho>0$.
Let $f : \RR^n\to \RR$ be a Lipschitz function 
satisfying $({\rm Lip})_{\rho}$.
Let $C(x)_t=\max_{0\le s\le t}f(x_s)$ for $x\in C([0,T],\RR^n)$.
Then we have $C\in {\rm (Lip)}_{\rho}\cap {\rm (BV)}_{\rho}$.
\end{pro}

\begin{proof}
We consider the simplest case $C(x)_t=\max_{0\le s\le t}x_s$, where
$x$ is a continuous path on $\RR$.
Let $0\le s<t$.
We take values $0\le s_{\ast}\le s, 0\le t_{\ast}\le t$ such that
$C(x)_{s}=x_{s_{\ast}}$ and $C(x)_{t}=x_{t_{\ast}}$.
Suppose $t_{\ast}\le s$, then
$C(x)_u=C(x)_{t_{\ast}}$ $(s\le u\le t)$ holds.
Hence $\|C(x)\|_{1\hyp var, [s,t]}=0$.
Suppose $s<t_{\ast}\le t$.
Then using $x_s\le x_{s_{\ast}}$,
we have
\begin{align*}
 C(x)_t-C(x)_s&=x_{t_{\ast}}-x_{s_{\ast}}
\le x_{t_{\ast}}-x_s\le \|x\|_{\infty\hyp var, [s,t]},
\end{align*}
which implies the validity of $({\rm BV})_1$.
We next show $({\rm Lip})_1$.
Let $x, x'$ be continuous paths on $\RR$.
Similarly, $t'_{\ast}$ denotes a time at which $x'$ attains its maximum
of $x_u$ $(0\le u\le t)$.
We have $C(x)_t-C(x')_t=x_{t_{\ast}}-x'_{t'_{\ast}}$.
If $x_{t_{\ast}}-x'_{t'_{\ast}}=0$, 
Suppose $x_{t_{\ast}}>x'_{t'_{\ast}}$.
Then, by $x'_{t'_{\ast}}\ge x'_{t_{\ast}}$, we have
\begin{align*}
 0\le C(x)_t-C(x')_{t}=x_{t_{\ast}}-x'_{t'_{\ast}}\le
x_{t_{\ast}}-x'_{t_{\ast}}\le
\|x-x'\|_{\infty,[0,t]}.
\end{align*}
This proves that $({\rm Lip})_{1}$ holds for $C(x)_t=\max_{0\le s\le t}x_s$.
 General cases follow from this simplest case.
\end{proof}

We consider (\ref{prde1}).
To this end, we consider the following condition on
$C$.

\bigskip

\noindent
${\bf (Condition~ \tilde{C})}$\\
~{\rm (i)} For any $x\in C([0,T],\RR^n)$, there exists unique
$y\in C([0,T],\RR^n)$ such that
$y=x+C(y)$. Define $\tilde{C}(x)=y-x$.\\
~{\rm (ii)}  $\tilde{C}$ satisfies ${\rm (Lip)}_{\rho'}$
for certain
$\rho'$.

\bigskip

About this property, we have the following.
The proof is straightforward and so  we omit the proof.

\begin{pro}
 Assume $C$ satisfies ${\rm (Condition~ \tilde{C})}$ ${\rm (i)}$.
Then for any $0\le t\le T$ and $x\in C([0,t],\RR^n)$, there exists a unique
$y\in C([0,t],\RR^n)$ such that
$y=x+C(y)$ on $[0,t]$.
For these $x$ and $y$, we define $\tilde{C}_t(x)=y-x\in C([0,t],\RR^n)$.
Then for any $z\in C([0,T],\RR^n)$ satisfying
$z_s=x_s$ $(0\le s\le t)$, $\tilde{C}(z)_s=\tC_t(x)_s$ $(0\le s\le t)$
holds.
\end{pro}

By this result, given $\xi\in \RR^n$, the solution $\eta\in \RR^n$
of $\eta=\xi+C_0(\eta)$ is unique if
$C$ satisfies ${\rm (Condition~ \tilde{C})}$ ${\rm (i)}$.
We have the following result for
(\ref{prde1}).

\begin{thm}\label{theorem for prde1}
Let $C$ be a continuous mapping between $C([0,T], \RR^n)$.
 Suppose $C$ satisfies ${\rm (Condition~ \tilde{C})}$
and $\tilde{C}$ satisfies $({\rm BV})_{\rho''}$ for
certain $\rho''$.
Let $\mathbf{X}\in \mathscr{C}^{\beta}(\RR^d)$.
\begin{enumerate}
 \item[$(1)$] There exists a controlled path $Z\in
 {\mathscr D}^{2\beta}_X(\RR^n)$ satisfying the equation
 \begin{align}
  Z_t=\xi+\int_0^t\sigma\left(Z_s+\tC(Z)_s\right)
  d\mathbf{X}_s,
\quad Z_t'=\sigma(Z_t+\tC(Z)_t).\label{path-rde5}
 \end{align}
and
$Z$ has the estimate similarly to Theorem~$\ref{main theorem}$.
Moreover $Y_t=Z_t+\tC(Z)_t$ is
 a solution to $(\ref{prde1})$.
\item[$(2)$]  Let $Y_t$ be a solution to $(\ref{prde1})$ defined by
$Z\in \mathscr{D}^{2\beta}_X(\RR^n)$.
Then $Z$ is a solution
to $(\ref{path-rde5})$.
Moreover, such a $Z$ is uniquely determined by $Y$.
\item[$(3)$] The transformations defined in $(1)$ and $(2)$ are inverse mapping
each other and the uniqueness of the solution of $(\ref{prde1})$ and 
$(\ref{path-rde5})$ is equivalent.
\end{enumerate}
\end{thm}

\begin{proof}
(1)~The existence and the estimate of the solution follows from
Theorem~\ref{main theorem}.
By $Y_t=Z_t+\tC(Z)_t$ and by the definition of $\tC$, 
we have $\tC(Z)=C(Y)$.
Hence $Z'_t=\sigma(Z_t+C(Y)_t)$ and
$Y_t$ is a solution to (\ref{prde1}).

\noindent
(2) By the definition of $\tC$, $\tC(Z)_t=C(Y)_t$ holds.
Hence $Z$ is a solution to (\ref{path-rde5}).
Also the uniqueness follows from the assumption on $C$.

\noindent
(3) These follows from the assumption on $C$.
\end{proof}

We give sufficient conditions on $C$ under which
$C$ satisfies ${\rm (Condition~ \tilde{C})}$.

\begin{lem}\label{estimate on tC via C}
 Let $C$ be a continuous mapping between $C([0,T], \RR^n)$.
\begin{enumerate}
 \item[{\rm (1)}]
 Assume $C$ satisfies ${\rm (Lip)}_{\rho_1}$ with
 $\rho_1<1$.
	      Let $x\in C([0,T], \RR^n)$.
	      There exists a unique $y\in C([0,T], \RR^n)$ satisfying
$
	      y=x+C(y).
$
	    	      Then $\tC$ satisfies ${\rm (Lip)}_{\rho_1/(1-\rho_1)}$.

 \item[{\rm (2)}] Suppose that $C$ satisfies
${\rm (Lip)}_{\rho_1}$ with
 $\rho_1<1$
	      and ${\rm (BV)}_{\rho_2}$
	      with $\rho_2<1$.
	      Then $\tC$ satisfies ${\rm (BV)}_{\rho_2/(1-\rho_2)}$.
\item[{\rm (3)}] Suppose that $C$ satisfies
${\rm (Lip)}_{\rho}$ and
${\rm (BV)}_{\rho}$ with $\rho<1/2$.
Then $\tC$ satisfies 
${\rm (Lip)}_{\rho'}$ and
${\rm (BV)}_{\rho'}$ with $\rho'=\frac{\rho}{1-\rho}<1$.
	      \end{enumerate}

\end{lem}

\begin{proof}
 (1)~The existence of $y$ follows from the
 fact that the mapping $y\mapsto x+C(y)$
 is contraction.
We have $\tC(x)=C(y)=C(x+\tC(x))$.
  Therefore,
 \begin{align}
  \|\tC(x)-\tC(x')\|_{\infty, [0,t]}\le \rho_1\left(
\|x-x'\|_{\infty,[0,t]}+\|\tC(x)-\tC(x')\|_{\infty,[0,t]}
  \right)
 \end{align}
 which implies
 $\|\tC(x)-\tC(x')\|_{\infty, [0,t]}\le
 \frac{\rho_1}{1-\rho_1}\|x-x'\|_{\infty, [0,t]}$.

 \noindent
 (2)
 We have
 \begin{align*}
  \|\tC(x)\|_{1\hyp var, [s,t]}&=
  \|C(x+\tC(x))\|_{1\hyp var, [s,t]}\\
  &\le
  \rho_2\left(\|x\|_{\infty\hyp var,[s,t]}
  +\|\tC(x)\|_{\infty\hyp var, [s,t]}\right)\\
  &\le
  \rho_2\left(\|x\|_{\infty\hyp var,[s,t]}
  +\|\tC(x)\|_{1\hyp var, [s,t]}\right)
 \end{align*}
 which implies the desired estimate.

\noindent
(3) This follows from (1) and (2).
\end{proof}
 
\begin{exm}\label{example for prde1}
$(1)$
We consider the following $C$:
\begin{align}
 C^i(Y)_t=
\sum_{j=1}^na^i_j\sup_{0\le s\le t}Y^j_s+
\sum_{j=1}^nb^i_j\inf_{0\le s\le t}Y^j_s,\label{example of C}
\end{align}
where $Y^j_t$ and $C^i(Y)_t$ are the $j$-th coordinate and
$i$-th coordinate of
$Y_t$ and $C(Y)_t$ respectively.
By Proposition~\ref{explicit rho} 
and Lemma~\ref{estimate on tC via C}, we see that
this $C$ satisfies the assumption in Theorem~\ref{theorem for prde1}
for sufficiently small $a^i_j, b^i_j$.
In this paper, we do not consider the subtle case as in 
the previous Subsection, {\it e.g.},
$|ab/(1-a)(1-b)|\le 1$ or $a<1, b<1$, etc.
We just mention the following simple result.

Let $a_i<1$ $(1\le i\le n)$ and consider $C$ defined by
$
 C^i(x)_t=a_i\max_{0\le s\le t}x^i(s),
$
where $x_t=(x^i_t)$.
If $a\le -1$, the mapping $C:x=(x_t)(\in C([0,T],\RR)\to 
(a\max_{0\le s\le t}x_s)\in C([0,T],\RR)$ is not a 
strict contraction mapping, but, 
$y=x+C(y)$ $(y\in C([0,T],\RR))$ is uniquely solved as
$y_t=x_t+\frac{a}{1-a}\max_{0\le s\le t}x_s$.
Therefore,
we have explicitly
\[
 \tilde{C}(x)_t=
\left(\frac{a_1}{1-a_1}\max_{0\le s\le t}x^1_s,\ldots,
\frac{a_n}{1-a_n}\max_{0\le s\le t}x^n_s\right).
\]
Hence, this example satisfies the assumption in 
Theorem~\ref{theorem for prde1}.

\noindent
$(2)$ Let $f_i : \RR^n\to\RR$ $(1\le i\le n)$
be Lipschitz functions satisfying 
$({\rm Lip})_{\rho_i}$.
For $x\in C([0,T],\RR^n)$,
we define $C$ by
$
 C^i(x)_t=\max_{0\le s\le t}f_i(x_s).
$
Then $C$ satisfies ${\rm (Lip)}_{\sqrt{\sum_i\rho_i^2}}$ and
${\rm (BV)}_{\sum_{i}\rho_i}$.
Hence, if $\sum_i\rho_i<1$, then
the assumption in Theorem~\ref{theorem for prde1} holds.
This follows from Proposition~\ref{explicit rho} 
and Lemma~\ref{estimate on tC via C}.

\end{exm}

 We now consider (\ref{prrde1}) on $\bar{D}=\{(x_1,\ldots,x_n)~|~x_n\ge 0\}$.
For the moment, we suppose $C$ satisfies 
${\rm (Condition~ \tilde{C})}$
and $\xi$ is chosen so that the solution $\eta$ of
$\eta=\xi+C_0(\eta)$ satisfies $\eta\in \bar{D}$
as we noted before.
Let $Y_t$ be a solution of (\ref{prrde1})
and suppose $Y_t=Z_t+C(Y)_t+\Phi_te_n$ as in
Definition~\ref{definition of the solution of prrde} (2) (ii).
Let $\tilde{Z}_t=Y_t-C(Y)_t$.
 Using $\tC$, we have $Y_t=\tilde{Z}_t+\tC(\tilde{Z})_t$.
 Then
 \begin{align}
  Y_t&=Z_t+C(Y)_t+\Phi_te_n\nonumber\\
  &=Z_t+\tC(\tilde{Z})_t+\Phi_te_n\nonumber\\
  &=Z_t+\tC(Z+\Phi e_n)_t+\Phi_te_n.\label{expression of Y}
 \end{align}
 By (\ref{half space reflection}), we get an equation for
 $\Phi_t$,
 \begin{align}
  \Phi_t=\max_{0\le s\le t}\Bigl\{-\Bigl(
Z^n_s+\tC^n(Z+\Phi e_n)_s\Bigr)\vee 0\Bigr\},
 \end{align}
where $Z^n_s$ and $\tC^n$ is the $n$-th coordinate of $Z_s$ and
$\tC$ respectively.
 This is a nonlinear implicit Skorohod equation.
 This kind of equation appeared in the study of the Euler approximation
 of the solutions for reflected RDEs in \cite{aida-rrde}.
 
 Fix $x\in C([0,T], \RR^n ; x_0=\xi)$ and
 consider a mapping on $C([0,T],\RR ; \phi_0=0)$:
 \begin{align}
  {\cal M}_x(\phi)_t=\max_{0\le s\le t}\Bigl\{-\Bigl(
  x^n_s+\tC^n(x+\phi e_n)_s\Bigr)\vee 0\Bigr\},\qquad
  \phi\in C([0,T], \RR ; \phi_0=0),
 \end{align}
where $x^n$ is the $n$-th coordinate of $x$.
 Now suppose that $x\mapsto \tC^n(x)$ is
a Lipschitz map belonging to $({\rm Lip})_{\kappa}$.
Then
we have
 for any $\phi,\phi'\in C([0,T], \RR ; \phi_0=0)$
 \begin{align}
  \|{\cal M}_x(\phi)-{\cal M}_x(\phi')\|_{\infty,[0,T]}
  \le \kappa\|\phi-\phi'\|_{\infty, [0,T]}.
 \end{align}
  Hence, if $\kappa<1$, that is, $\tC^n$ is
  strict contraction, then ${\cal M}_x$ is a contraction mapping
 for all $x\in C([0,T], \RR^n ; x_0=\xi)$.
 Let us denote the fixed point by $\tL(x)$.
Then we have $\Phi=\tL(Z)$.
 Thus, under the assumption that
$\tilde{C}^n : C([0,T],\RR^n)\to C([0,T],\RR)$ satisfies
${\rm (Lip)}_{\rho}$ with $\rho<1$, we obtain a mapping
 $x(\in C([0,T], \RR^n ; x_0=\xi))\mapsto 
\tL(x)\in C([0,T], \RR ; \phi_0=0)$
 and the equation for $Z$:
 \begin{align}
  Z_t=\xi+\int_0^t\sigma\left(Z_s+\tilde{C}(Z+\tL(Z)e_n)_s
  +\tL(Z)_se_n\right)
  d\mathbf{X}_s.
\label{prrde2}
 \end{align}

We have the following estimate of $\tilde{L}$.

\begin{lem}\label{estimate of tA0}
 Suppose
\begin{itemize}
 \item[{\rm (i)}]  $C$ satisfies $({\rm Condition~ \tilde{C}})$
and $\tC$ satisfies $({\rm BV})_{\rho''}$ for some $\rho''>0$.
\item[{\rm (ii)}] $\tC^n$ satisfies $({\rm Lip})_{\kappa}$ with $\kappa<1$
and $\tilde{C}^n$ satisfies $({\rm BV})_{\kappa'}$ with
$\kappa'<1$.
\end{itemize}
Let
$
\tA(x)_t=\tC(x+\tL(x)e_n)_t+\tL(x)_te_n.
$
Then the following hold.

\begin{enumerate}
 \item[$(1)$] ${\displaystyle
 \|\tL(x)-\tL(x')\|_{\infty,[0,t]}\le
\frac{1+\kappa}{1-\kappa}\|x-x'\|_{\infty,[0,t]}.
}$

\item[$(2)$] $\displaystyle{
   \|\tL(x)\|_{1\hyp var, [s,t]}\le
\frac{1+\kappa'}{1-\kappa'}
\|x\|_{\infty\hyp var,[s,t]}.}
$
\item[$(3)$] $\displaystyle{
  \|\tA(x)-\tA(x')\|_{\infty,[0,t]}\le
\left(\rho'+(1+\rho')\frac{1+\kappa}{1-\kappa}\right)
\|x-x'\|_{\infty,[0,t]}}$,
\item[$(4)$]
$\displaystyle{
\|\tA(x)\|_{1\hyp var, [s,t]}\le
\left(\rho''+(1+\rho'')\frac{1+\kappa'}{1-\kappa'}\right)
\|x\|_{\infty\hyp var, [s,t]}.
}$
\end{enumerate}
 \end{lem}

\begin{proof}
(1)~
 Since $\tL(x)$ satisfies
 \begin{align}
  \tL(x)_t=\max_{0\le s\le t}\left\{-\left(
x^n_s+\tC^n(x+\tL(x)e_n)_s\right)\vee 0\right\},\label{explicit form of tL}
 \end{align}
 we have
 \begin{align*}
  \|\tL(x)-\tL(x')\|_{\infty,[0,t]}&\le
  \|x-x'\|_{\infty, [0,t]}+\|\tC^n(x+\tL(x)e_n)-
  \tC^n(x'+\tL(x')e_n)\|_{\infty, [0,t]}\\
  &\le \|x-x'\|_{\infty, [0,t]}+
  \kappa\left(\|x-x'\|_{\infty, [0,t]}+
\|\tL(x)-\tL(x')\|_{\infty, [0,t]}
  \right),
 \end{align*}
which implies
\begin{align*}
 \|\tL(x)-\tL(x')\|_{\infty,[0,t]}&\le
\frac{1+\kappa}{1-\kappa}\|x-x'\|_{\infty,[0,t]}.
\end{align*}

\noindent
(2)~We have
 \begin{align*}
  \|\tL(x)\|_{1\hyp var, [s,t]}&\le
  \|x^n+\tC^n(x+\tL(x)e_n)\|_{\infty\hyp var, [s,t]}\\
  &\le \|x\|_{\infty\hyp var, [s,t]}+
  \|\tC^n(x+\tL(x)e_n)\|_{\infty\hyp var, [s,t]}\\
  &\le \|x\|_{\infty\hyp var, [s,t]}+
  \kappa'\left(\|x\|_{\infty\hyp var, [s,t]}+
  \|\tL(x)\|_{1\hyp var, [s,t]}\right).
 \end{align*}
Thus, we obtain
$\displaystyle{
   \|\tL(x)\|_{1\hyp var, [s,t]}\le
\frac{1+\kappa'}{1-\kappa'}
\|x\|_{\infty\hyp var,[s,t]}.}
$

\noindent
(3) By using (1) and (2), we have
\begin{align*}
 \|\tA(x)-\tA(x')\|_{\infty,[0,t]}&\le
\rho'\left(\|x-x'\|_{\infty,[0,t]}+\frac{1+\kappa}{1-\kappa}
\|x-x'\|_{\infty,[0,t]}\right)+
\frac{1+\kappa}{1-\kappa}
\|x-x'\|_{\infty,[0,t]},
\end{align*}
which implies the desired result.

\noindent
(4) We have
\begin{align*}
 \|\tA\|_{1\hyp var, [s,t]}&\le
\|\tilde{C}(x+\tL(x)e_n)\|_{1\hyp var, [s,t]}+
\|\tL(x)\|_{1\hyp var, [s,t]}\\
&\le \rho''\left(\|x\|_{\infty\hyp var,[s,t]}+
\|\tL(x)\|_{\infty\hyp var, [s,t]}\right)+
\frac{1+\kappa'}{1-\kappa'}\|x\|_{\infty\hyp var, [s,t]}\\
&\le 
\rho''\left(\|x\|_{\infty\hyp var,[s,t]}+
\frac{1+\kappa'}{1-\kappa'}\|x\|_{\infty\hyp var, [s,t]}
\right)+
\frac{1+\kappa'}{1-\kappa'}\|x\|_{\infty\hyp var, [s,t]}\\
&=\left(\rho''+(1+\rho'')\frac{1+\kappa'}{1-\kappa'}\right)
\|x\|_{\infty\hyp var,[s,t]}.
\end{align*}
\end{proof}

The following lemma follows from
Lemma~\ref{estimate on tC via C}.

\begin{lem}\label{sufficient condition for lemma}
Suppose $C$ satisfies 
 ${\rm (Lip)}_{\rho}$ and
${\rm (BV)}_{\rho}$ with $\rho<1/2$.
Then the assumption of Lemma~$\ref{estimate of tA0}$ 
{\rm (i)} and {\rm (ii)}
hold with $\rho'=\rho''=\kappa=\kappa'=\frac{\rho}{1-\rho}<1$.
\end{lem}

We now state our theorem for (\ref{prrde1}) and give the proof.

\begin{thm}\label{theorem for prrde1}
Let $C$ be a continuous mapping between $C([0,T], \RR^n)$.
Suppose $C$ satisfies the same assumption in Lemma~$\ref{estimate of tA0}$.
Moreover we assume that the solution $\eta$ of $\eta=\xi+C_0(\eta)$ 
satisfies $\eta\in \bar{D}$.
Let $\mathbf{X}\in \mathscr{C}^{\beta}(\RR^d)$.
 Let $\tA$ be the mapping defined in Lemma~$\ref{estimate of tA0}$.
\begin{enumerate}
 \item[$(1)$] There exsits a solution $Z\in
 {\mathscr D}^{2\beta}_X(\RR^n)$ to $(\ref{prrde2})$ and
$Z$ has the estimate similarly to Theorem~$\ref{main theorem}$.
Let $Y_t=Z_t+\tA(Z)_t$ and $\Phi_t=\tL(Z)_t$.
Then
\begin{align}
 Y_t=Z_t+C(Y)_t+\Phi_te_n,\quad Z'_t=\sigma(Y_t),\quad
\Phi_t=L(Z+C(Y))_t \label{From Z to Y}
\end{align}
hold.
That is, $(Y,\Phi)$ is a solution to $(\ref{prrde1})$.
\item[$(2)$] Let $(Y,\Phi)$ be a solution to $(\ref{prrde1})$ 
and $Z$ be a controlled path
appearing in Definition~$\ref{definition of the solution of prrde}$ $(2)$.
Then $Z$ is a solution to $(\ref{prrde2})$.
Moreover, $Z$ is uniquely determined by $Y$ and $\mathbf{X}$.
\item[$(3)$] The transformations defined in $(1)$ and $(2)$ are inverse mapping
each other and the uniqueness of the solution of $(\ref{prrde1})$ and 
$(\ref{prrde2})$ is equivalent.
\end{enumerate}

\end{thm}
   
   \begin{proof}
(1)   By Lemma~\ref{estimate of tA0} and Theorem~\ref{main
   theorem},
   there exists a solution $Z$ to (\ref{prrde2})
and has the estimate given in Theorem~\ref{main theorem}.
By the definition of $\tL$ and $\tC$, we have
$\tL(Z)=L(Z+\tC(Z+\tL(Z)e_n))$ and
$\tC(Z+\tL(Z))=C(Y)$ which shows (\ref{From Z to Y}).

\noindent
$(2)$ The argument by which we derived the equation (\ref{prrde2})
shows the former half part.
$Z$ is uniquely defined by $Y$ and $\mathbf{X}$ only
because $Z_t=\xi+\int_0^t\sigma(Y_s)d\mathbf{X}_s$,
$Y$ is a sum of $Z$ and
a continuous bounded variation path and 
$Z'_t=\sigma(Y_t)$.

\noindent
$(3)$ The invertibility of the mapping follows from the definition.
The latter half statement follows from this property of the 
mapping.
   \end{proof}

 \begin{exm}\label{example for prrde1}
$(1)$
We consider $C$ in (\ref{example of C}).
If $a^i_j, b^i_j$ are sufficiently small, then 
the assumption on $C$ in Lemma~\ref{estimate of tA0} holds by
Proposition~\ref{explicit rho}, 
Lemma~\ref{estimate on tC via C} and
Lemma~\ref{sufficient condition for lemma}.

\noindent
$(2)$ We consider the example in Example~\ref{example for prde1}(2).
Suppose $\sum_{i}\rho_i<1/2$.
Then the assumption on $C$ in Lemma~\ref{estimate of tA0} holds.
This follows from Proposition~\ref{explicit rho},
Lemma~\ref{estimate on tC via C} and
Lemma~\ref{sufficient condition for lemma}.

\noindent
$(3)$
Let $a\in \RR$ and
 we consider Lipschitz functions $f_i$ $(1\le i\le n)$ in
Example~\ref{example for prde1} (2) and define
for $x=(x^i)_{i=1}^n\in C([0,T],\RR^n)$,
\begin{align}
 C(x)_t=\left(
\max_{0\le s\le t}f_1(x_s),\ldots, \max_{0\le s\le t}f_{n-1}(x_s),
\max_{0\le s\le t}f_n(x_s)+a\max_{0\le s\le t}x^n_s
\right).
\end{align}
Suppose $\xi$ is chosen so that $\eta\in \bar{D}$.
For example, if $a<1$, 
$f_n(\eta_1,\ldots,\eta_{n-1},0)\ge 0$ for all $\eta$ and 
$\|\frac{\partial f_n}{\partial\eta_n}\|_{\infty}$ is sufficiently small,
$\xi\in \bar{D}$ is sufficient for $\eta\in \bar{D}$.

We prove that if $a<1/2$ and $\sum_{i=1}^{n}\rho_i$ is sufficiently small,
$C$ satisfies the assumption in Lemma~\ref{estimate of tA0}.

Let
\begin{align*}
  C_f(x)_t=\left(\max_{0\le s\le t}f_1(x_s),\ldots, \max_{0\le s\le t}f_n(x_s)
\right),\quad
C_{f_n}(x)_t=\max_{0\le s\le t}f_n(x_s).
\end{align*}
The equation $y=x+C(y)$ is equivalent to
\[
 y=x+C_f(y)+\frac{a}{1-a}\max_{0\le s\le t}(C_{f_n}(y)_s+x^n_s)e_n=:\Phi_x(y)
\]
If $\sum_i\rho_i$ is sufficiently small, then the mapping
$y\mapsto \Phi_x(y)$ is a strict contraction mapping for all $x$.
Thus, $y=x+C(y)$ is uniquely solved and $\tilde{C}(x)=y-x$ is defined.
Note that 
\[
 \tilde{C}(x)=C_f(x+\tC(x))+
\frac{a}{1-a}\max_{0\le s\le t}(C_{f_n}(x+\tC(x))_s+x^n_s)e_n.
\]
By this expression, 
for any $\ep>0$, if $\sum_i\rho_i$ is sufficiently small,
we have, for any $x,x'\in C([0,T],\RR^n)$,
\begin{align*}
 &\|\tC(x)-\tC(x')\|_{\infty,[0,t]}\le
\ep\left(\|\tC(x)-\tC(x')\|_{\infty,[0,t]}+\|x-x'\|_{\infty,[0,t]}\right)
+\frac{|a|}{1-a}\|x-x'\|_{\infty,[0,t]},\\
&\|\tC(x)\|_{1\hyp var, [s,t]}
\le
\ep\left(1+\frac{|a|}{1-a}\right)\left(\|x\|_{\infty\hyp var,[s,t]}
+\|\tC(x)\|_{\infty\hyp var,[s,t]})\right)
+\frac{|a|}{1-a}\|x^n\|_{\infty\hyp var,[s,t]}.
\end{align*}
This shows that if $a<1/2$ and $\sum_i\rho_i$ sufficiently small,
the assumption of Lemma~$\ref{estimate of tA0}$ 
is satisfied.
\end{exm}

\begin{rem}[Remark on the It\^o and Stratonovich SDEs]
\label{remark on sde}
The equations, (\ref{prde1}) and (\ref{prrde1}) are formulated by using
rough integrals.
We now consider the equations replacing the rough integrals by
It\^o and Stratonovich integrals against the standard Brownian motion
$W_t$.
The solutions are semimartingales and the equations are well-defined.
We need to assume $\sigma$ is Lipschitz continuous and
$\sigma\in C^2_b$ for the It\^o and Stratonovich integrals
respectively.
Under the same assumptions on $C$
in Theorem~\ref{theorem for prde1} and Theorem~\ref{theorem for prrde1},
the existence and the pathwise uniqueness hold 
for the stochastic integral's version
of (\ref{path-rde5}) and (\ref{prrde2}) by the Lipschitz continuity of
their coefficients
which implies the uniqueness of the solutions of
the stochastic integral's version of
(\ref{prde1}) and (\ref{prrde1}).
In Section 6, we consider Stratonovich SDEs corresponding to
(\ref{prde1}) and (\ref{prrde1}) and prove the support theorem
of the solutions (Corollary~\ref{support theorem for psde}).

Consider Example~\ref{example for prrde1} (3).
  In the case of standard Brownian motion,
  this example extends 
  the existence results for solutions in Doney and Zang~\cite{doney-zhang}
slightly.
  Also we can extend the absolutely continuity property
  of the law of $Y_t$ in Yue and Zhang~\cite{yue-zhang}.
  We study this problem in a separate joint paper with Yuki
  Kimura.
\end{rem}

\subsection{Related path-dependent RDEs}\label{subsection 5.3}

We consider the H\"older rough path
$\mathbf{X}$.
Namely, $\omega(s,t)=t-s$.
In this Subsection, we consider RDEs depending on
the $L^p$-norm of the solution.
For simplicity, we consider the case where
$A(x)_t$ is a real-valued process.
That is, 
\begin{itemize}
 \item[(1)] $\sigma\in 
{\rm Lip}^{\gamma}(\RR^n\times \RR\to \mathcal{L}(\RR^d,\RR^n))$.
\item[(2)] We consider the following case:
\begin{itemize}
 \item[(2a)] Let $f\in {\rm Lip}^1(\RR^n,\RR)$ and $A(x)_t=\int_0^tf(x_s)ds$,
\item[(2b)] Let $R>0$, $1<p\le\frac{1}{\beta}$ and set
$A(x)_t=\left\{\int_0^t(|x_s|\wedge R)^pds\right\}^{1/p}$,
\item[(2c)] Let $0<\ep<R<\infty$,
$p>1$ and set
$A(x)_t=\left\{\int_0^t\left(\ep\vee |x_s|\wedge R\right)^p
ds\right\}^{1/p}$.
\end{itemize}
\end{itemize}
Our RDE is of the form,
\begin{align}
 Z_t=\xi+\int_0^t\sigma(Z_s,A(Z)_s))d\mathbf{X}_s, \label{path-rde2}
\end{align}
as before.
In the case (2a), the equation reads
$
 Z_t=\xi+\int_0^t\sigma(Z_s,\Psi_s)d\mathbf{X}_s,
\Psi_t=\int_0^tf(Z_s)ds,
$
which is usual RDE and we have the existence and the uniqueness
of the solutions.
We consider the case (2b).
Clearly, $({\rm Lip})_{1}$ holds.
For simplicity, we write $f(x)=|x|\wedge R$.
Note that
\begin{align}
 A(x)_t-A(x)_s&\le\left\{\int_s^t|f(x_u)-f(x_s)|^pdu\right\}^{1/p}+
|f(x_s)|(t-s)^{1/p}\nonumber\\
&\le 
\left(\|x\|_{\infty\hyp var,[s,t]}+R\right)
(t-s)^{1/p},\quad
0\le s<t\le T.\label{estimate of Axst}
\end{align}
Hence noting a remark in Example~\ref{example of A} (3), we see that
the solution exists and a priori estimate holds.
Actually, we can prove the uniqueness of the solution under the additional 
assumption that $\xi\ne 0$.

\begin{pro}\label{pro of 2b}
Assume $(1)$ and ${\rm (2b)}$ in the above.
Further we assume $f(|\xi|)(=:\ep)\ne 0$.
Then the solution of $(\ref{path-rde2})$ is unique.
\end{pro}

\begin{lem}\label{lemma for pro of 2b}
Assume the same assumption as in Proposition~$\ref{pro of 2b}$.
Here we allow $p>1$.
 We have the following estimates.
Below, $C_i$ $(i=1,2)$ are polynomial functions
and $\omega(s,t)=t-s$, $\tomega(s,t)=t^{1/p}-s^{1/p}$.
\begin{align}
 \|A(x)\|_{1\hyp var, [s,t]}&\le C_1(\ep^{-1},R,\|x\|_{\beta,[0,t]})
\left(\tomega(s,t)+\omega(s,t)\right),
\label{estimate of Ax2}\\
\|A(x)-A(y)\|_{1\hyp var, [s,t]}&\le
C_2(\ep^{-1},R,\|x\|_{\beta,[0,t]},\|y\|_{\beta,[0,t]})\|x-y\|_{\infty,[0,t]}
\left(\tomega(s,t)+\omega(s,t)\right). 
\label{estimate of Axy}
\end{align}
\end{lem}

\begin{proof}
 We have
\begin{align*}
 A(x)_t'=\frac{1}{p}f(x_t)^p\left(\int_0^tf(x_u)^pdu\right)^{\frac{1}{p}-1}.
\end{align*}
Also, we have
\begin{align*}
 |f(x_u)-f(x_0)|\le \|x\|_{\beta}u^{\beta}.
\end{align*}
Hence 
\begin{align*}
 |f(x_u)|\ge \frac{\ep}{2}\quad
\text{for 
$u\le \left(\frac{\ep}{2\|x\|_{\beta}}\right)^{1/\beta}$},
\end{align*}
which implies
\begin{align*}
 \left(\int_0^tf(x_u)^pdu\right)\ge
\left(\frac{\ep}{2}\right)^{p}
\left\{t\wedge \left(\frac{\ep}{2\|x\|_{\beta}}\right)^{1/\beta}\right\}
\end{align*}
and
\begin{align}
A(x)_t^{1-p}&\le \left(\frac{2}{\ep}\right)^{p-1}
\left\{\frac{1}{t^{1-1/p}}+\left(\frac{2\|x\|_{\beta}}{\ep}\right)
^{\frac{p-1}{\beta p}}
\right\}.
\label{lower bdd of Ax}
\end{align}
Therefore,
\begin{align*}
 \|A(x)\|_{1\hyp var, [s,t]}&
\le \frac{R^p}{p}\left(\frac{2}{\ep}\right)^{p-1}
\left\{
p(t^{1/p}-s^{1/p})+
\left(\frac{2\|x\|_{\beta}}{\ep}\right)
^{\frac{p-1}{\beta p}}(t-s)
\right\}.
\end{align*}
Let $y$ be another $\beta$-H\"older continuous path with
$|y_0|=\ep$.
We have
\begin{align*}
 A(x)_t'-A(y)_t'&=\frac{1}{p}\frac{f(x_t)^p-f(y_t)^p}{A(x)_t^{p-1}}
+\frac{f(y_t)^p}{p}\frac{A(y)_t^{p-1}-A(x)_t^{p-1}}
{(A(x)_tA(y)_t)^{p-1}}\\
&=:I_1(t)+I_2(t).
\end{align*}
Using the elementary inequality,
$\left|\frac{b^{r}-a^r}{b-a}\right|\le r\max\left(a^{r-1}, b^{r-1}\right)$
$(a,b>0, r\in \mathbb{R})$, we have
\begin{align*}
 \int_s^t|I_1(u)|du&\le
\left(\frac{2R}{\ep}\right)^{p-1}
\left(p\left(t^{1/p}-s^{1/p}\right)+
\left(\frac{2\|y\|_{\beta}}{\ep}\right)
^{\frac{p-1}{\beta p}}(t-s)
\right)\|x-y\|_{\infty, [s,t]}.
\end{align*}
Also we have
\begin{align*}
 &\left|A(x)_t^{p-1}-A(y)_t^{p-1}\right|\\
&\quad\le
(p-1)R^{p-1}\|x-y\|_{\infty, [0,t]}
\frac{2t}{\ep}\left(\frac{1}{t^{1/p}}+
\left(\frac{2\|x\|_{\beta}}{\ep}\right)^{1/(\beta p)}+
\left(\frac{2\|y\|_{\beta}}{\ep}\right)^{1/(\beta p)}
\right).
\end{align*}
Hence using (\ref{lower bdd of Ax}),
\begin{multline*}
\int_s^t|I_2(u)|du
\le
\frac{(p-1)R^{2p-1}}{p}\left(\frac{2}{\ep}\right)^{2p-1}
\|x-y\|_{\infty, [0,t]}\\
\times\int_s^t
\left\{u^{1-1/p}+
u\left(\left(\frac{2\|x\|_{\beta}}{\ep}\right)^{1/(\beta p)}+
\left(\frac{2\|y\|_{\beta}}{\ep}\right)^{1/(\beta p)
}\right)\right\}
 \left\{\frac{1}{u^{1-1/p}}+\left(\frac{2\|x\|_{\beta}}{\ep}\right)
^{\frac{p-1}{\beta p}}
\right\}^2du,
\end{multline*}
which completes the proof.
\end{proof}

\begin{proof}[Proof of Proposition~$\ref{pro of 2b}$]
 Let $Z_t$ and $\tZ_t$ be solutions to (\ref{path-rde2}) and
suppose $\|Z-\tZ\|\ne 0$.
We may assume $z_t:=\|Z-\tZ\|_{\infty, [0,t]}>0$ for all
$t>0$.
Otherwise, that is, if 
$t_0=\inf\{t\ge 0~|~\|Z-\tZ\|_{\infty, [0,t]}>0\}>0$ happens, then
it suffices to consider solutions $Z_t$ and $\tZ_t$ from
$t_0$.
We have
\begin{align*}
 Z_t-\tZ_t&=\int_0^t
[A(Z)_s-A(\tZ)_s]dG_s
+\int_0^t
[Z_s-\tZ_s]dH_s,
\end{align*}
where $G_s$ and $H_s$ are
$\mathcal{L}(\RR^d,\RR^n)$-valued maps which act from the right as
\begin{align*}
 \eta G_s&=\int_0^s\left(\int_0^1(D_2\sigma)(Z_u,A(\tZ)_u+
\theta(A(Z)_u-A(\tZ)_u))
d\theta\right)[\eta]dX_u,\\
 \eta H_s&=\int_0^s
\left(\int_0^1(D_1\sigma)(\tZ_u+\theta(Z_u-\tZ_u),A(\tZ)_u)
d\theta\right)[\eta]dX_u
\end{align*}
and the integrals are Stieltjes integral and the rough integral.
The rough integral is well-defined because we assume 
$\sigma\in\Lip^{\gamma}$.
Clearly, $G_s, H_s$ are controlled paths of $\mathbf{X}$.
We fix $t$ and consider the processes in the time interval $0\le s\le t$.
Let $F_s=z_{t}^{-1}(A(Z)_s-A(\tZ)_s)$
and set $\tF_s=\int_0^sF_udG_u$.
By Lemma~\ref{lemma for pro of 2b} and a priori estimates of
$Z, \tZ$, we have
$|F_{u,v}|\le K\bomega(u,v)^{\beta}$,
where $\bomega(u,v)=\tomega(u,v)+\omega(u,v)$ and
the positive constant $K$ depends only on
$\sigma, p, \beta, \mathbf{X}$.
Then we have the following expansion,
\begin{align}
& Z_s-\tZ_s=z_t\tF_s+\int_0^s[Z_u-\tZ_u]dH_u
=z_t\tF_s+\sum_{k=1}^{n-1}I_k(s)+J_{n}(s),\quad n\ge 1,
\label{expansion of Z-tZ}\\
& I_1(s)=z_t\tF_s,\quad J_0(s)=Z_s-\tZ_s,
\quad I_k(s)=\int_0^sI_{k-1}(u)dH_u,
\quad J_{n}(s)=\int_0^sJ_{n-1}(u)dH_u.
\end{align}
We now consider $(\bomega,\beta)$-H\"older rough path
$\mathbf{X}(A)$
whose first level path is $F_{u,v}\oplus X_{u,v}\in \RR^{n+d}$ and
the iterated integrals of them are defined in a natural way
using $\mathbb{X}_{u,v}$.
Let $X(A)^k_{u,v}(\in (\RR^{d+n})^{\otimes k})$ be the $k$-level path $(k=1,2)$.
Then
it holds that
$|X(A)^k_{u,v}|\le K\bomega(u,v)^{k\beta}$
$(k=1,2)$.
We can regard
$F, G, H, Z-\tZ$ as controlled paths of $(\bomega,\beta)$-H\"older rough
path $\mathbf{X}(A)$.
Therefore using the estimate of the higher order iterated integrals,
we obtain
\begin{align}
 \max_{0\le s\le t}|I_k(s)|\le
z_t C^k\frac{\bomega(0,t)^{k\beta}}{\left(k\beta\right)!},
\quad
\max_{0\le s\le t}|J_n(t)|\le
C^n\frac{\bomega(0,t)^{n\beta}}{\left(n\beta\right)!},
\end{align}
where $C$ is a certain constant which may depend on $\mathbf{X}$.
Thus, for all $0\le t\le T$, 
there exist positive numbers
$C, C'$ which may depend on $\mathbf{X}$ such that
\[
 z_t
\le z_t C' \bomega(0,t)^{\beta}
+C^n\frac{\bomega(0,t)^{n\beta}}{\left(n\beta\right)!}
\]
which implies for sufficiently small $t$ and for all $n$
\[
 z_t\le (1-C'\bomega(0,t)^{\beta})^{-1}
C^n\frac{\bomega(0,t)^{n\beta}}{(n\beta)!}
\]
and so $z_t=0$ for sufficiently small $t$.
This completes the proof.
\end{proof}

\begin{rem}
 In the above argument, we assume $1<p\le 1/\beta$
and we use a priori estimate of the H\"older norm of the
solution $Z$.
When $p>1/\beta$, the path of $A(x)$ just satisfies very low
regularity around $0$.
Hence we cannot apply our argument directly to this case.
However note that $\beta$-H\"older rough path $\mathbf{X}$ can be regarded
as a $1/p$-H\"older rough path and
$A(x)\in \C^{1\hyp var, 1/p}$.
Hence, we can extend our result by considering controlled paths
of $\mathscr{D}^{\lfloor p\rfloor/p}_X$ and $\C^{1\hyp var, 1/p}$.
That is, under the assumption $\sigma\in \Lip^{\lceil p\rceil}$ and
$x_0\ne 0$, for all $p>1$,
we can prove the existence and uniqueness of the solutions
of (\ref{path-rde2}) in the case of 
$A(x)_t=\left\{\int_0^t(|x_s|\wedge R)^p\right\}^{1/p}$.
However the assumption $\sigma\in \Lip^{\lceil p\rceil}$ seems unnecessary.
\end{rem}

In the case of (2c), 
we can prove the existence and the uniqueness of the solutions
in a similar argument to Proposition~\ref{pro of 2b}.
However, unfortunately, we cannot prove uniform estimate of the solutions
when $p\to\infty$ and so the estimates cannot be applied to
the case $A(x)_t=\max_{0\le s\le t}f(x_s)$.
Note that any $(\omega,\beta)$-H\"older rough paths
$\mathbf{X}$ is a $(\bomega,\beta)$-H\"older rough path.
Let $\mathscr{D}^{2\beta,\bomega}_X(\RR^n)$ be the associated controlled path
spaces defined by $\bomega$.
Similarly to the case of (\ref{path-rde1}),
the integral in (\ref{path-rde2}) is also well-defined.

\begin{pro}\label{pro of 2c}
Let us consider the situation $(1)$ and ${\rm (2c)}$ above.
Then there exists a unique solution to $(\ref{path-rde2})$.
\end{pro}

Similarly to Proposition~\ref{pro of 2b}, we need the following lemma.
The proof is almost similar to Lemma~\ref{lemma for pro of 2b}
and we omit it.

\begin{lem}\label{lemma for pro of 2c}
Assume $(1)$ and ${\rm (2c)}$.
 We have the following estimates.
\begin{align}
 \|A(x)\|_{1\hyp var, [s,t]}&\le C_1(\ep^{-1},R)\tomega(s,t),
\label{estimate of Ax3}\\
\|A(x)-A(y)\|_{1\hyp var, [s,t]}&\le
C_2(\ep^{-1},R)\|Df\|_{\infty}\|x-y\|_{\infty,[0,t]}\tomega(s,t). 
\label{estimate of Axy2}
\end{align}
\end{lem}

\begin{proof}[Proof of Proposition~$\ref{pro of 2c}$]
 We can proceed as in Section~3 
by adopting the control function
$\bar{\omega}(s,t)=|t-s|+t^{1/p}-s^{1/p}$
with the help of Lemma~\ref{lemma for pro of 2c}.
\end{proof}

     \section{Continuity property and Support theorem}\label{support theorem}
     Let $W_t$ be a standard $d$-dimensional Brownian motion.
     Then we have the notion of the It\^o and Stratonovich SDEs
     driven by $W_t$.
     Let ${\bf W}$ be the associated Brownian rough path
     defined by the Stratonovich integral.
     When $A\equiv 0$ and $\sigma\in C^3_b$,
     the solution $Z({\bf W})$ is equal to
     the solution to the Stratonovich SDE in It\^o's calculus
     for almost all $W$.
     This is checked, for example,
     by the Wong-Zakai theorem and Lyon's continuity
     theorem.
In our cases, we do not have the uniqueness.
However, under the assumption that $\sigma\in C^2_b$,
the Wong-Zakai theorems hold
     for reflected SDEs, perturbed SDEs and perturbed reflected
     SDEs. By using this and
     Proposition~\ref{continuity of Z},
     we can prove the continuity of the solution mapping
of the SDEs at Lipschitz paths
     in the rough path topology.
     We prove support theorem for the above mentioned processes
     by using the continuity.
          
Let us recall the definition of the Brownian rough path.
Let $\Theta^d=C([0,T], \RR^d)$
and $\mu$ be the Wiener measure on $\Theta^d$.
Let $W\in \Theta^d$ and
$W^N_t$ be the dyadic polygonal approximation of $W$, that is,
\begin{align}
 W^N_t=W_{t^N_{i-1}}+2^{-N}T^{-1}(t-t^N_i)
 W_{t^N_{i-1},t^N_i},
 \quad t^N_{i-1}\le t\le t^N_i,
~t^N_i=2^{-N}iT,~0\le i\le 2^N.
\end{align}
Let ${\mathbb W}^N_{s,t}=\int_s^tW^N_{s,u}\otimes dW^N_u$.
Let us define
\begin{align}
 \Omega_1&=\left\{W\in \Theta^d~|~
\mbox{
$(W^N_{s,t}, {\mathbb W}^N_{s,t})$ converges
in ${\mathscr C}^{\beta}(\RR^d)$ for all
$\beta<1/2$}
\right\}.
\end{align}
Here ${\mathscr C}^{\beta}$ is defined
by $\omega(s,t)=|t-s|$, that is,
${\mathscr C}^{\beta}$ denotes the
set of usual H\"older rough paths.

It is known that $\mu(\Omega_1)=1$ and the limit point
of $(W^N_{s,t}, {\mathbb W}^N_{s,t})$ for $W\in \Omega_1$
is called the Brownian rough path.
We identify the element of $W\in \Omega_1$
and the associated Brownian rough path ${\bf W}$.
Clearly $\{\mathbf{W}~|~W\in \Omega_1\}\subset
\mathscr{C}^{\beta}_g(\RR^d)$ holds for any $\beta<1/2$.

Let $\sigma\in C^2_b(\RR^n,{\cal L}(\RR^d,\RR^n))$ and
consider a Stratonovich reflected SDE,
\begin{align}
 dY_t&=\sigma(Y_t)\circ dW_t+d\Phi_t,\quad Y_0=\xi\in \bar{D}.
 \label{stratonovich rsde}
\end{align}
We write $Z_t=Y_t-\Phi_t$.
The corresponding solution
$Y^N_t$ which is obtained by replacing $W_t$ 
by $W^N_t$ is called the Wong-Zakai approximation of
$Y_t$.
We also denote the corresponding reflected term by
$\Phi^N$ and set $Z^N=Y^N-\Phi^N$.
It is proved in \cite{aida-sasaki, zhang, aida} that 
the Wong-Zakai approximations of the solution to
a reflected Stratonovich SDE under (A), (B) and (C)
converge to the solution in the uniform convergence topology
almost surely.
Note that a similar statement under the conditions (A) and (B)
is proved in \cite{aida}.
See also previous results \cite{doss, evans-stroock}.
By using the result in \cite{aida-sasaki, zhang},
a support theorem for the reflected diffusion under
the conditions (A), (B) and (C) is proved by
Ren and Wu~\cite{ren-wu}.
We now prove a support theorem for the reflected diffusion under
(A) and (B) by using the estimates in rough path analysis in this paper
and in \cite{aida, aida-sasaki}.
First we note the following results.

\begin{lem}\label{aida-previous}
 Assume $\sigma\in C^2_b(\RR^n,{\cal L}(\RR^d,\RR^n))$.
 We consider the solution $(Y, Z, \Phi)$
 to $(\ref{stratonovich rsde})$ and
 their Wong-Zakai approximations $(Y^N, Z^N, \Phi^N)$.
\begin{enumerate}
 \item[$(1)$]~Assume $D$ satisfies condition {\rm (A), (B), (C)}.
Let
\begin{align}
 \Omega_2&=
 \left\{W\in \Theta^d~\Big |~
 \max_{0\le t\le T}\left\{|Y^N_t-Y_t|+|Z^N_t-Z_t|+
 |\Phi^N_t-\Phi_t|\right\}\to
0~~\mbox{as $N\to \infty$}\right\}.
\end{align}
Then $\mu(\Omega_2)=1$.
\item[$(2)$]~Assume {\rm (A), (B)} hold.
Then there exists an increasing sequence $\{N_k\}\subset \mathbb{N}$
such that $\mu(\Omega_3)=1$ holds where,
\begin{align}
 \Omega_3&=\left\{W\in \Theta^d~\Big |~
 \max_{0\le t\le T}
 \left\{|Y^{N_k}_t-Y_t|
+|Z^{N_k}_t-Z_t|+|\Phi^{N_k}_t-\Phi_t|
 \right\}
 \to
0~~\mbox{as $k\to \infty$}
 \right\}.
 \label{omega3}
\end{align}
\end{enumerate}
\end{lem}

\begin{proof}
 (1)~This is proved in Lemma~5.1 in \cite{aida-rrde}.

 \noindent
 (2)~It is proved in \cite{aida}
 that $\max_{0\le t\le T}|Y^N_t-Y_t|$ converges to
$0$ in probability under (A) and (B).
This and the moment estimate in \cite{aida-sasaki}
 implies that
 $\lim_{N\to\infty}E[\max_{0\le t\le T}|Y^N_t-Y_t|^p]=0$ for any
 $p\ge 1$ and
 there exists a subsequence $N_k\uparrow \infty$
 such that $\max_{0\le t\le T}
 \left|Y^{N_k}_t-Y_t\right|\to 0$ $\mu$-almost surely.
 By using this and by a similar proof to Lemma 5.1 in \cite{aida-rrde},
 we can prove 
 (\ref{omega3}) by taking a subsequence if necessary.
\end{proof}

   We now consider the Stratonovich SDEs
   corresponding to (\ref{prde1}) and (\ref{prrde1}).

   \begin{align}
    Y^p_t=\xi+\int_0^t\sigma(Y^p_s)\circ dW_s
    +C(Y^p)_t,\label{psde1}\\
      Y^{pr}_t=\xi+\int_0^t\sigma(Y^{pr}_s)\circ dW_s
  +C(Y^{pr})_t+\Phi_te_n,\label{prsde1}
   \end{align}

   We assume $\sigma\in C^2_b$ and
   $C$ satisfies the same assumption as in
   Theorem~\ref{theorem for prde1} and
   Theorem~\ref{theorem for prrde1} respectively.
   
   We can transform these equations to the following equation
   with certain $A$ which satisfies
   ${\rm (Lip)}_{\rho}$ and ${\rm (BV)}_{\rho}$ for some $\rho>0$
   in a similar way to (\ref{prde1}) and (\ref{prrde1}),
   \begin{align}
    Z_t=\xi+\int_0^t\sigma(Z_s+A(Z)_s)\circ dW_s.
    \label{transformed prrde1}
   \end{align}
The relation between $Y(=Y^{p}\,\text{or}\, Y^{pr})$ and
$Z$ is given as
$Y_t=Z_t+A(Z)_t$.
      Clearly, if we consider an ODE which is obtained by
replacing $W$ by a Lipschitz path $h$,
then the solution is unique.

   In \cite{aida-kikuchi-kusuoka}, we proved a Wong-Zakai type
 theorem for the above Stratonovich SDEs
 under the conditions on $A$: 
(A1), (A2) and (A3).
These conditions are weaker than 
${\rm (Lip)}_{\rho}$ and 
 ${\rm (BV)}_{\rho}$.
Thus we have the following.

  \begin{lem}
   Suppose $A$ satisfies
   ${\rm (Lip)}_{\rho}$ and ${\rm (BV)}_{\rho}$ for some $\rho>0$
   and $\sigma\in C^2_b(\RR^n,{\cal L}(\RR^d,\RR^n))$.
  Let us consider the solution $Z$
 to $(\ref{transformed prrde1})$ and
 the Wong-Zakai approximation $Z^{N}$ defined by
$W^N$.
Let
\begin{align}
 \Omega_4&=
 \left\{W\in \Theta^d~\Big |~
 \max_{0\le t\le T}|Z^{N}_t-Z_t|
\to
0~~\mbox{as $N\to \infty$}\right\}.
\end{align}
Then $\mu(\Omega_4)=1$.
  \end{lem}

 We prove support theorems for the solutions
 to (\ref{stratonovich rsde}), (\ref{psde1})
 and (\ref{prsde1})
 as an application of the results in Section~\ref{a continuity property}.
For such purpose,
it is important to obtain the support of ${\bf W}$.
The following is due to Ledoux-Qian-Zhang~\cite{lqz}.
More general results can be found in \cite{friz-victoir}.
 \begin{thm}\label{support of BRP}
  Let $\beta<1/2$.
Let $\hat{\mu}$ be the law of $\mathbf{W}$.
Then we have
${\rm Supp}\,\hat{\mu}=\mathscr{C}^{\beta}_g(\RR^d)$,
where ${\rm Supp}\,\hat{\mu}$ denotes the topological support of
$\hat{\mu}$.
 \end{thm}

 In Remark~\ref{Solinfty},
 we define a subset of the solution mapping
 $Sol_{\infty}(\mathbf{X})$ $(\mathbf{X}\in \mathscr{C}^{\beta}_g(\RR^d))$.
We see the topological support of the selection mapping
with values in 
$\cup_{\mathbf{X}\in\mathscr{C}^{\beta}_g(\RR^d))}Sol_{\infty}(\mathbf{X})$
as follows.
 
 \begin{thm}\label{general support theorem}
Let $\nu$ be a probability measure on $\mathscr{C}^{\beta}_g(\RR^d)$.
Let $S$ be a subset of
$\mathscr{C}^{\beta}_g(\RR^d)$.
We assume ${\rm Supp}\,\nu=\mathscr{C}^{\beta}_g(\RR^d)$ and 
$\nu(S)=1$. Let us consider RDE $(\ref{path-rde1})$ and the
solution $Z(\mathbf{X})$ under the same assumption in
Theorem~$\ref{main theorem}$.
We assume the solution for any smooth rough path is unique.
Let ${\cal I} :
 \mathbf{X}(\in S)\mapsto
 Z(\mathbf{X})(\in Sol_{\infty}(\mathbf{X}))\in 
\C^{\beta-}([0,T],\RR^n)$ be a
measurable mapping with respect to the
$\nu$-completed Borel $\sigma$-field.
  Then we have
$
 {\rm Supp}\,({\cal I}_{\ast}\nu)=
 \overline{\{Z(h)~|~h\in \C^1\}}^{\C^{\beta-}},
$
  where ${\rm Supp}\,({\cal I}_{\ast}\nu)$ denotes the
  topological support of the image measure of $\nu$ by
${\cal I}$.
     \end{thm}
\begin{proof}
The inclusion ${\rm Supp}\,({\cal I}_{\ast}\nu)\subset
 \overline{\{Z(h)~|~h\in \C^1\}}^{\C^{\beta-}}$
 follows from the definition of $Sol_{\infty}(\mathbf{X})$.
  The converse inclusion follows from the continuity of 
the multivalued solution mapping at the set of smooth rough paths
which follows from Proposition~\ref{continuity of Z} and
 the assumption on $\nu$.
\end{proof}

At the moment, we do not have the uniqueness theorem for
(\ref{rrde1}), (\ref{prde1}) and (\ref{prrde1}).
However, the strong solutions exist uniquely
for the corresponding Stratonovich SDEs driven by 
Brownian motion under the smoothness assumption on
$\sigma$.
Moreover, the Wong-Zakai theorem hold for them and 
this convergence theorem gives selection mappings $\mathcal{I}$
in Theorem~\ref{general support theorem}
for such cases and we can obtain the support theorem for them.

 \begin{cor}\label{support theorem for reflected diffusion}
  Assume $D$ satisfies {\rm (A)} and {\rm (B)}
  and $\sigma\in C^2_b$.
  Let $Y$ be the solution to $(\ref{stratonovich rsde})$.
   Let $0<\beta<1/2$.
 Let $P^Y$ be the law of $Y$ on
  $\C^{\beta}([0,T], \RR^n~|~Y_0=\xi)$.
Then the support of $P^Y$ is given by
  \begin{align}
 {\rm Supp}\,(P^Y)=
   \overline{\{Y(h)~|~h\in \C^1(\RR^d)\}}^{\C^{\beta}}.
   \label{reflected sde skelton set}
\end{align}
 \end{cor}

  \begin{proof}
For $\mathbf{X}=(X_{s,t},\XX_{s,t})\in \mathscr{C}^{\beta}_g(\RR^d)$, let
$X^N_t$ be the dyadic polygonal approximation of $X$
similarly defined as $W^N$.
Let $Y^N, Z^N, \Phi^N$ be the solution to (\ref{stratonovich rsde})
driven by $X^N$.
Let $\{N_k\}$ be the increasing sequence in 
Lemma~\ref{aida-previous} (2).
Define
\begin{align}
 \tilde{\Omega}_3&=\left\{\mathbf{X}\in \mathscr{C}^{\beta}_g(\RR^d)~\Big |~
 \text{
$Y^{N_k}_t$, $Z^{N_k}_t$ and
$\Phi^{N_k}_t$
converges uniformly on $[0,T]$}~\text{as $k\to \infty$}
 \right\}
 \label{tomega3}
\end{align}
and 
\begin{align}
 Y_t(\mathbf{X})=\lim_{k\to\infty}Y^{N_k}_t,
\,\,
Z_t(\mathbf{X})=\lim_{k\to\infty}Z^{N_k}_t,
\,\,
\Phi_t(\mathbf{X})=\lim_{k\to\infty}\Phi^{N_k}_t,
\quad \mathbf{X}\in \tilde{\Omega}_3.
\label{YZPhi on tOmega3}
\end{align}
$\tilde{\Omega}_3$ is a Borel measurable subset of
$\mathscr{C}^{\beta}_g(\RR^d)$
and $Y$, $Z$ and $\Phi$ are Borel measurable mapping
defined on $\tilde{\Omega}_3$.
By $Y^{N}_t=Z^N_t+\Phi^N_t=Z^N_t+L(Z^N)_t$
and the continuity property of $L$,
$Y_t(\mathbf{X})=Z_t(\mathbf{X})+L_t(Z(\mathbf{X}))$
$(\mathbf{X}\in \tilde{\Omega}_3)$ holds.
Let $\hat{\Omega}_3=\tilde{\Omega}_3\cap \{\mathbf{W}~|~W\in \Omega_1\}$.
Then $\hat{\mu}(\hat{\Omega}_3)=1$ and
\begin{align*}
 Y=Y(\mathbf{W})=Z(\mathbf{W})+L(Z(\mathbf{W})),\quad 
\mathbf{W}\in \hat{\Omega}_3
\end{align*}
holds.
Note that
$L : \C^{\beta-}([0,T],\RR^n; x_0=\xi)\to \C^{\beta-}([0,T],\RR^n)$
is a continuous mapping.
This follows from Lemma~\ref{lemma for A} (2), 
Lemma~\ref{estimate of phi} and Lemma~\ref{Holder continuity}.
Hence it suffices to apply Theorem~\ref{general support theorem} to the case
$\mathcal{I}(\mathbf{W})=Z(\mathbf{W})$, $S=\hat{\Omega}_3$ and
$\nu=\hat{\mu}$ to obtain the support theorem in the topology of
$\C^{\beta-}$.
Since we can choose any $\beta\in (0,1/2)$,
this completes the proof.
           \end{proof}

	   Similarly, we have the following result.
	   Since the proof is similar to that of
	   Corollary~\ref{support theorem for reflected diffusion},
	   we omit the proof.

  \begin{cor}\label{support theorem for psde} We consider the
   solutions $Y^{p}$ and $Y^{pr}$ to $(\ref{psde1})$ and
   $(\ref{prsde1})$
   respectively. Let $0<\beta<1/2$.
   We consider the laws of $P^{Y^p}$ and $Y^{pr}$ on
   $\C^{\beta}$. Then we have
      \begin{align}
 {\rm Supp}\,(P^{Y^p})&=
 \overline{\{Y^p(h)~|~h\in \C^1(\RR^d)\}}^{\C^{\beta}},\\
 {\rm Supp}\,(P^{Y^{p,r}})&=
 \overline{\{Y^{pr}(h)~|~h\in \C^1(\RR^d)\}}^{\C^{\beta}}.
   \end{align}
  \end{cor}

  \section{Appendix}
  \subsection{Uniqueness of the Gubinelli derivative}

  Let $Y$ be a continuous path on $\RR^n$
 and suppose that
 there exist $(Z, Z'), (\tZ,\tZ')\in {\mathscr D}^{2\theta}_X(\RR^n)$
 and continuous bounded variation paths
 $(\Psi_t)_{0\le t\le T}, (\tPsi_t)_{0\le t\le T}$
 such that $Y_t=Z_t+\Psi_t=\tZ_t+\tPsi_t$~$(0\le t\le T)$.
 We discuss conditions under which $Z'_t=\tZ'_t$ holds for
 all $t$.

 We consider the following condition 
$\mathrm{C}(\delta, p, \xi)$
on a continuous curve $X$ on $\RR^d$ $(d\ge 2)$,
 where $0<\delta<\pi/2$, $p>1$ and $\xi$ is a unit vector.
This is related to the property of the truly roughness of the
   path (\cite{friz-hairer}).
Below we denote the angle between two non-zero vectors $v_1$ and $v_2$ by
 $\theta(v_1,v_2)$, where $0\le \theta(v_1,v_2)\le \pi$.

\bigskip
\noindent
$\mathrm{C}(\delta,p,\xi)$~
For any interval $[s,t]\subset [0,T]$,
 there exists a sequence of partitions
 ${\cal P}_N=\{s=t^N_0<\cdots <t^N_{k(N)}=t\}$
 of $[s,t]$ such that $\lim_{N\to\infty}|{\cal P}_N|=0$
 and
 \begin{align}
  \lim_{N\to\infty}
  \sum_{i\in {\cal P}_N(X,\xi,\delta)}
    |X_{t_{i-1},t_i}|^{p}
=\infty,\label{divergence of variation norm}
 \end{align}
 where
  \begin{align}
   {\cal P}_N(X,\xi,\delta)=\left\{1\le i\le k(N)~|
\mbox{$\theta\left(X_{t_{i-1},t_i}, \xi\right)\le\frac{\pi}{2}-\delta$
  \,or\,
 $\theta\left(X_{t_{i-1},t_i}, -\xi\right)\le\frac{\pi}{2}-\delta$
 holds}\right\}.\label{restriction on X}
  \end{align}

\bigskip

   It is easy to check
   that almost all sample paths 
of a fractional Brownian motion(=fBm) $W_t$
whose Hurst parameter $0<H<1$ satisfies
${\rm C}(\delta,p,\xi)$ for
$p<1/H$ and sufficiently small $\delta$.

The proof is as follows.
  Let $e_1=\xi$. We choose orthonormal vectors $e_2,\ldots,e_d$
 in $\RR^d$ which are orthogonal to $e_1$.
 Let
 $W^k_t:=(W_t,e_k)$
 and $W^{\perp}_t=\sum_{k=2}^dW^k_te_k$.
 Then by the rotational invariance of
 $W_t$,
 $(W^1_t,\ldots,W^d_t)$ are also
 fBm with the Hurst parameter $H$.
 Let ${\cal P}_N$ be the dyadic partition of
 $[s,t]$, that is, $t^N_i=s+i2^{-N}(t-s)$~$(0\le i\le 2^N)$.
 The condition $i\in {\cal P}_N(W,e_1,\delta)$ is equivalent to
 \begin{align}
  |W^1_{t^N_{i-1},t^N_i}|\ge (\tan\delta)|W^{\perp}_{t^N_{i-1},t^N_i}|.
  \label{restriction equivalent}
 \end{align}
Below, we write $W_{t^N_{i-1},t^N_{i}}=W_{i-1,i}$ and so on for simplicity.
Let $p>1$.
We have
\begin{align*}
& I_N:=\sum_{i\in \mathcal{P}_N(W,e_1,\delta)}|W_{i-1,i}|^p
=\sum_{i=1}^{2^N}|W_{i-1,i}|^p
1_{|W^1_{i-1,i}|\ge (\tan\delta) |W^{\perp}_{i-1,i}|}\\
&\quad\ge \sum_{i=1}^{2^N}|W^1_{i-1,i}|^p
1_{|W^1_{i-1,i}|\ge (\tan\delta) |W^{\perp}_{i-1,i}|}\\
&\ge \sum_{i=1}^{2^N}|W^1_{i-1,i}|^p
1_{|W^1_{i-1,i}|\ge (\tan\delta) |W^{\perp}_{i-1,i}|}
+\sum_{i=1}^{2^N}
\left(
|W^1_{i-1,i}|^p-(\tan\delta)^p |W^{\perp}_{i-1,i}|^p
\right)
1_{|W^1_{i-1,i}|<(\tan\delta) |W^{\perp}_{i-1,i}|}\\
&\ge \sum_{i=1}^{2^N}|W^1_{i-1,i}|^p
-(\tan\delta)^p\sum_{i=1}^{2^N}
|W^{\perp}_{i-1,i}|^p\\
&\ge\sum_{i=1}^{2^N}|W^1_{i-1,i}|^p
-(\tan\delta)^p(d-1)^{(p-1)/p}\sum_{k=2}^d\sum_{i=1}^{2^N}|W^k_{i-1,i}|^p.
\end{align*}
By Remark D.3.2 in \cite{nourdin-peccati}, for all $1\le k\le d$,
we have
\begin{align*}
\lim_{N\to\infty}\frac{2^{NHp-N}}{(t-s)^p}\sum_{i=1}^{2^N}
\left|W^k_{i-1,i}\right|^p=
\int_{\RR}|x|^pd\mu(x)\quad\text{in $L^2$},
\end{align*}
where $\mu$ is the 1 dimensional standard normal distribution.
This implies that if $pH<1$
for sufficiently small $\delta$ there exists a subsequence
$N_k\uparrow \infty$ such that
$\lim_{N_k\to\infty}I_{N_k}=+\infty$ for almost all $W$.
  
\begin{pro}
 Let $\mathbf{X}\in \mathscr{C}^{\beta}(\RR^d)$ $(1/3<\beta \le 1/2)$.
Let us choose $p$ such that $1/(2\beta)<p<1/\beta$.
Assume that the first level path $X_t$ satisfies ${\rm C}(\delta, p, \xi)$
for any $\xi\in K$ and a fixed positive $\delta$,
where $K$ is a countable dense subset of $S^{d-1}$.
Let $Y_t$ be a continuous path on $\RR^n$.
  Suppose $Y_t=Z_t+\Psi_t=\tilde{Z}_t+\tilde{\Psi}_t$,
where $Z, \tZ\in \mathscr{D}^{2\beta}_X(\RR^n)$ and $\Psi, \Psi'$
are continuous bounded variation paths.
Then $Z'_t=\tZ'_t$ $(0\le t\le T)$ holds.
\end{pro}

\begin{proof}
 Suppose that there exists $s_0$ such that
 $Z'_{s_0}-\tZ'_{s_0}\ne 0$.
 Then there exist $0<s_1<s_0<s_2<T$ and
 unit vectors $\xi_1, \xi_2\in K$ such that
 \begin{align}
&  \ep_{s_1,s_2}=\inf_{s_1<s<s_2}|(Z'_s-\tZ'_s)^{\ast}\xi_1|>0,\\
  & \sup_{s_1<s<s_2}\theta\left((Z'_s-\tZ'_s)^{\ast}\xi_1, \xi_2\right)
  \le\delta/2.
 \end{align}
 We use ${\rm C}(\delta,p,\xi)$ 
for $\xi=\xi_2$ and the interval $[s_1,s_2]$.
 Let $t^N_i,t^N_{i+1}$ be partition points of ${\cal P}_N$.
If $i\in {\cal P}_N(X,\xi_2,\delta)$,
 then
 \begin{align}
\mbox{$\theta\left(X_{t^N_{i-1},t^N_i},
 (Z'_s-\tZ'_s)^{\ast}\xi_1\right)\le\frac{\pi}{2}-\frac{\delta}{2}$\,
 or\,
 $\theta\left(-X_{t^N_{i-1},t^N_i},
 (Z'_s-\tZ'_s)^{\ast}\xi_1\right)\le\frac{\pi}{2}-\frac{\delta}{2}$ holds.}
 \end{align}
   Therefore, for $i\in {\cal P}_N(X,\xi_2,\delta)$,
 \begin{align}
  \left|
  \left(X_{t^N_{i-1},t^N_i}, (Z'_{t^N_{i-1}}-
  \tZ'_{t^N_{i-1}})^{\ast}\xi_1\right)
  \right|
  &\ge \sin\left(\frac{\delta}{2}\right)\left|X_{t^N_{i-1},t^N_i}\right|
  \left|(Z'_{t^N_{i-1}}-\tZ'^N_{t_{i-1}})^{\ast}\xi_1\right|\nn\\
&  \ge \ep_{s_1,s_2}\sin\left(\frac{\delta}{2}\right)
  \left|X_{t^N_{i-1},t^N_i}\right|.
   \end{align}
 Also we note that
 \begin{align}
  \left(X_{t^N_{i-1},t^N_i},
  (Z'_{t^N_{i-1}}-\tZ'_{t^N_{i-1}})^{\ast}\xi_1\right)
  &=
  \left(\tPsi_{t^N_{i-1},t^N_i}-\Psi_{t^N_{i-1},t^N_i}, \xi_1\right)
  +\left(R^{\tZ}_{t^N_{i-1},t^N_i}-R^Z_{t^N_{i-1},t^N_i}, \xi_1\right).
 \end{align}
 Thus we obtain
 \begin{align}
  \sum_{i\in {\cal P}_N(X,e_2,\delta)}
   \left|X_{t^N_{i-1},t^N_i}\right|^{p}
  \le 2^{p-1}\left(\ep_{s_1,s_2}\sin\left(\frac{\delta}{2}\right)\right)^{-p}
    \left(\|\tPsi-\Psi\|_{p\hyp var, [s_1,s_2]}^p+
  \|R^{\tZ}-R^Z\|_{p\hyp var, [s_1,s_2]}^p\right).
 \end{align}
Since $p>1/(2\beta)$, the right hand side of the above inequality is
bounded.
 This contradicts the assumption on $X$.
\end{proof}
  
  \subsection{Path-dependent RDE with drift}
    We consider path-dependent rough differential equations with
  drift term.
  It is necessary to consider such kind of equations for the study
  of reflected diffusions with drift.
  In the case of $n$-dimensional Brownian motion $W_t=(W^1_t,\ldots,W^n_t))$,
  one possible approach to include the drift term is to consider the
  geometric rough path defined by
  $\tilde{W}_t=(W_t,t)\in \RR^{d+1}$.
  By considering the geometric rough path which is naturally defined by
  Brownian rough path ${\bf W}_{s,t}$, we may extend all results in
  previous sections to the corresponding results for the solutions
  to the equation,
  \begin{align}
   Z_t&=\xi+\int_0^t\sigma(Z_s,A(Z)_s)d{\bf
   W}_s+\int_0^tb(Z_s,A(Z)_s)ds.
  \end{align}
  However, we need to assume $b\in \Lip^{\gamma-1}(\RR^n\times \RR^n,\RR^n)$
  ~$(2<\gamma\le 3)$
  to do so and the assumption on $b$ is too strong.
  Hence, we explain different approach to deal with the drift term.
Let us consider $\beta$-H\"older
  rough path, that is, the case where $\omega(s,t)=|t-s|$.
 Let $b\in C^1_b(\RR^n\times \RR^n,\RR^n)$
 and consider the equation,
\begin{align}
 Z_t&=\xi+\int_0^t\sigma(Z_s,\Psi_s)d\mathbf{X}_s+\int_0^tb(Z_s,\Psi_s)ds,
 \label{Z equation''}\\
 \Psi_t&=A\left(\xi+\int_0^{\cdot}\sigma(Z_s,\Psi_s)d\mathbf{X}_s+
 \int_0^{\cdot}b(Z_s,\Psi_s) ds\right)_t
 \label{Psi equation''}.
\end{align}
 The meaning of this equation is as follows.
 The controlled path $(Z,Z')$ and $\Psi$
 are elements as in the definition of
 $\Xi_{s,t}$ and $I(Z,\Psi)_{s,t}$.
 In the present case, we consider
 \begin{align}
  \tilde{\Xi}_{s,t}:=
  \sigma(Y_s)X_{s,t}+(D_1\sigma)(Y_s)Z'_s\XX_{s,t}+
  (D_2\sigma)(Y_s)\int_s^t\Psi_{s,r}\otimes dX_r
  +b(Y_s)(t-s).\label{tXi}
 \end{align}
 Then, 
 $(\delta\tilde{\Xi})_{s,u,t}=(\delta\Xi)_{s,u,t}+
 \left(b(Y_s)-b(Y_u)\right)(t-u)$ for $s<u<t$
 and
 \begin{align}
  \lefteqn{|\left(b(Y_s)-b(Y_u)\right)(t-u)|}\nonumber\\
  &\le
  \|Db\|_{\infty}\left(\|Z'\|_{\infty}\|X\|_{\beta}(t-s)^{\beta}+
  \|R^Z\|_{2\beta}(t-s)^{2\beta}+\|\Psi\|_{q\hyp var,\ta}
  (t-s)^{\ta}\right)(t-s).
 \end{align}
 By using this, we define
 \begin{align}
  I(Z,\Psi)_{s,t}&:=\int_s^t\sigma(Z_u,\Psi_u)d\mathbf{X}_u+
  \int_s^tb(Z_u,\Psi_u)du=
  \lim_{|{\cal P}|\to 0}\sum_{{\cal P}}\tilde{\Xi}_{u,v}.
 \end{align}
 For this, $\left(I(Z,\Psi)_{0,t},\sigma(Z_t,\Psi_t)\right)\in
 {\mathscr D}^{2\beta}_{X}$ and similar estimates to
 Lemma~\ref{estimate of rough integral} holds.
 Moreover, Lemma~\ref{continuity} and
 Lemma~\ref{invariance and compactness} hold.
 Thus, Theorem~\ref{main theorem} holds
 for suitable constants
   which depend only on $\sigma, b, \beta, p, \gamma$.
   In the case of reflected rough differential equation,
   all statements can be extended to
   differential equations with drift term
   $b\in C^1_b$.
   In particular, the extension of 
   Corollary~\ref{support theorem for reflected diffusion} 
to reflected diffusion with the drift term
   gives an extension of
   the support theorem 
   in \cite{ren-wu}.

\section*{Acknowledgment}
The author would like to thank the referee for the 
suggestions which improve the 
quality of the paper.

\section*{Funding}
Open Acceess funding provided by The University of Tokyo.
 This research was partially supported by
Grant-in-Aid for Scientific Research (B) No.~20H01804.

\section*{Data Availability Statement}
The manuscript is
available in the arXiv:
https://doi.org/10.48550/arXiv.1608.03083


\begin{thebibliography}{99}
	 %
\bibitem{aida-kikuchi-kusuoka}
S.~Aida, T.~Kikuchi and S.~Kusuoka,
The rates of the $L^p$-convergence of 
the Euler-Maruyama and Wong-Zakai approximations
of path-dependent stochastic differential equations
under the Lipschitz condition,
Tohoku Math. J. (2) 70 (2018), no. 1, 65--95.
%
\bibitem{aida-rrde}
S.~Aida,
Reflected rough differential equations,
Stochastic Process. Appl.
125 (2015), no.9, 3570--3595.
%
\bibitem{aida}
S.~Aida,
Wong-Zakai approximation of solutions to reflecting stochastic differential
equations on domains in Euclidean spaces II,
Springer Proceedings in Mathematics \& Statistics Volume 100, 2014, 1--23.
%
\bibitem{aida-sasaki}
S.~Aida and K.~Sasaki,
Wong-Zakai approximation of solutions to reflecting stochastic differential
equations on domains in Euclidean spaces,
Stochastic Process. Appl. Vol. 123 (2013), no.10, 3800-3827.
%
\bibitem{anderson-orey}
R.F.~Anderson and S.~Orey,
Small random perturbations of dynamical systems with reflecting
boundary,
Nagoya Math. J. 60 (1976), 189-216.
%
\bibitem{bailleul1}
I.~Bailleul,
Flows driven by rough paths,
Rev. Mat. Iberoam. 31 (2015), no.3,
901--934.
%
\bibitem{cpy}
P.~Carmona, F.~Petit and M.~Yor,
Beta variables as times spent in 
$[0,\infty[$ by certain perturbed Brownian motions.
J. London Math. Soc. (2) 58 (1998), no.1, 239--256.
%
 \bibitem{chaumont-doney}
	 L.~Chaumont and R.A.~Doney,
	 Pathwise uniqueness for perturbed versions of Brownian motion
	 and
	 reflected Brownian motion,
	  Probab. Theory Related Fields 113 (1999), no. 4, 519--534.
	%
\bibitem{davie}
A.M.~Davie,
Differential equations driven by rough paths:
an approach via discrete approximations,
Appl. Math. Res. Express. AMRX 2007, no. 2, Art. ID abm009, 40 pp.
	%
 \bibitem{davis}
	 B.~Davis,
	 Weak limits of perturbed random walks and the equation
	 $Y_t=B_t+
	 \alpha\sup\{Y_s : s\le t\}+\beta\inf\{Y_s : s\le t\}$,
	  Ann. Probab. 24 (1996), no. 4, 2007--2023.
	 	%
\bibitem{davis2}
B.~Davis,
Brownian motion and random walk perturbed at extrema,
Probab. Theory Related Fields 113 (1999), no.4, 501--518.
%
 \bibitem{dght}
	 A.~Deya, M.~Gubinelli, M.~Hofmanov\'a and
	 S.~Tindel,
	 One-dimensional reflected rough differential equations,
 Stochastic Process. Appl. 129 (2019), no. 9, 3261--3281. 
	 %
 \bibitem{doney-zhang}
	 R.A.~Doney and T.~Zhang,
	 Perturbed Skorohod equations and perturbed reflected diffusion
	 processes,
	 Ann. Inst. H. Poincar\'e Probab. Statist. 41 (2005), no. 1,
	 107--121.
	 %
\bibitem{doss}
H.~Doss and P.~Priouret,
Support d'un processus de r\'eflexion,
Z. Wahrsch. Verw. Gebiete 61 (1982), no. 3, 327--345. 
	%
 \bibitem{dudley}
	 R.~M.~Dudley,
	 Real analysis and probability,
	 Cambridge Studies in Advanced Mathematics, 74.
	 Cambridge University Press, Cambridge, 2002.
	 %
\bibitem{dupuis-ishii}
P.~Dupuis and H.~Ishii,
On Lipschitz continuity of the solution mapping to the Skorokhod
problem, with applications. 
Stochastics Stochastics Rep. 35 (1991), no. 1, 31--62. 
%
\bibitem{evans-stroock}
L.C. Evans and D.W.~Stroock,
An approximation scheme for reflected stochastic differential 
equations,
Stochastic Process. Appl. 121 (2011), no. 7, 1464--1491. 
	%
	\bibitem{falkowski-slominski1}
		A.~Falkowski and L.~S\l omi\'nski,
		Stochastic differential equations with constraints
		driven by processes with bounded p-variation.
		Probab. Math. Statist. 35 (2015), no. 2, 343--365.
	%
\bibitem{ferrante-rovira}
M. Ferrante and C. Rovira,
Stochastic differential equations with non-negativity constraints
driven by fractional Brownian motion,
J. Evol. Equ. 13 (2013), 617-632.
%
\bibitem{friz-hairer}
P.~Friz and M.~Hairer,
A Course on Rough Paths\,
With an Introduction to Regularity Structures,
Second edition,
Universitext, Springer, Cham, (2020).
%
 \bibitem{fp}
	 P.~Friz and D.~Pr\"omel,\,
Rough path metrics on a Besov-Nikolskii-type scale. 
Trans. Amer. Math. Soc. 370 (2018), no. 12, 8521--8550.
%
\bibitem{friz-victoir}
P.~Friz and N.~Victoir,
Multidimensional Stochastic Processes as Rough Paths\,
{\small Theory and Applications},
Cambridge Studies in Advanced Mathematics, 120,
Cambridge University Press (2010).
%
\bibitem{gassiat}
P.~Gassiat,
Non-uniqueness for reflected rough differential equations. 
Ann. Inst. Henri Poincar\'e Probab. Stat. 57 (2021), no. 3, 1369--1387.
%
\bibitem{gubinelli}
M.~Gubinelli,
Controlling rough paths.
J. Funct. Anal.  216  (2004),  no. 1, 86-140.
	%
\bibitem{legall-yor}
	 J.F.~Le Gall and M.~Yor,
	 Enlacements du mouvement brownien autour des courbes de
	 l\'espace,
	  Trans. Amer. Math. Soc. 317 (1990), no. 2, 687--722.
	%
 \bibitem{lqz}
	 M.~Ledoux, Z.~Qian and T.~Zhang,
	 Large deviations and support theorem for
	 diffusion processes via rough paths.
	 Stochastic Process. Appl. 102 (2002), no. 2, 265--283.
	%
\bibitem{lions-sznitman}
P.L.~Lions and A.S.~Sznitman,
Stochastic differential equations with reflecting boundary
conditions,
Comm. Pure Appl. Math. 37 (1984), no. 4, 511–537.
%
\bibitem{lyons98}
T.~Lyons,~
Differential equations driven by rough signals,
Rev.Mat.Iberoamer., 14 (1998), 215-310.
%
\bibitem{lq}
T.~Lyons and Z.~Qian,~
System control and rough paths, (2002),
	Oxford Mathematical Monographs.
	%
 \bibitem{nourdin-peccati}
	 I.~Nourdin and G.~Peccati,
	 Normal approximations with Malliavin calculus.
	 From Stein's method to universality.
	 Cambridge Tracts in Mathematics, 192.
	 Cambridge University Press, Cambridge, 2012. 
	%
 \bibitem{perman-werner}
M.~Perman and W.~Werner,
	 Perturbed Brownian motions,
	 Probab. Theory Related Fields 108 (1997), no. 3, 357--383.
%
\bibitem{pettersson}
R.~Pettersson,
Wong-Zakai approximations for reflecting stochastic differential
	equations. Stochastic Anal.Appl. 17 (1999), no. 4, 609–617.
%
\bibitem{ren-wu}
J.~Ren and J.~Wu,
On approximate continuity
and the support of reflected stochastic differential equations.
Ann. Probab. 44 (2016), no. 3, 2064--2116.
%
\bibitem{saisho}
Y.~Saisho,
Stochastic differential equations for multi-dimensional
domain with reflecting boundary,
Probab. Theory Related Fields 74 (1987), no. 3, 455–477. 
%
\bibitem{tanaka}
H.~Tanaka,
Stochastic differential equations with reflecting boundary condition
in convex regions,
Hiroshima Math. J. 9 (1979), no. 1, 163–177. 
	%
 \bibitem{yue-zhang}
	 W.~Yue and T.~Zhang,
	 Absolute continuity of the laws of perturbed diffusion
	 processes and perturbed reflected diffusion processes,
	  J. Theoret. Probab. 28 (2015), no. 2, 587--618.
	%
\bibitem{zhang}
T-S. Zhang,
Strong Convergence of Wong-Zakai Approximations
of Reflected SDEs in A Multidimensional
General Domain,
Potential Anal. 41 (2014), no.3, 783--815.
%
\end{thebibliography}
\end{document}